%% file: graded.tex
\documentclass[reqno]{amsart}
\usepackage{amsmath,amssymb,amsthm,geometry,graphics,color,mdwlist,mathrsfs}
\usepackage[all]{xy}

\input{definitions2}

\def\tC{\widetilde{\C}}

\def\st{\mathrm{st}}

\renewcommand{\ell}{l}
\DeclareMathOperator{\Md}{Mod}
\numberwithin{equation}{subsection}

\title[Enhanced Auslander-Reiten duality and Morita theorem for singularity categories]{Enhanced Auslander-Reiten duality and\\ Morita theorem for singularity categories}
\author{Norihiro Hanihara}
\address{N. Hanihara: Kavli Institute for the Physics and Mathematics of the Universe (WPI), The University of Tokyo Institutes for Advanced Study, The University of Tokyo, 5-1-5 Kashiwanoha, Kashiwa, Chiba, 277-8583, Japan}
\email{norihiro.hanihara@ipmu.jp}
\address{Current address: Faculty of Mathematics, Kyushu University, 744 Motooka, Nishi-ku, Fukuoka, 819-0395, Japan}
\email{hanihara@math.kyushu-u.ac.jp}

\author{Osamu Iyama}
\address{O. Iyama: Graduate School of Mathematical Sciences, The University of Tokyo, 3-8-1 Komaba, Meguro-ku, Tokyo, 153-8914, Japan}
\email{iyama@ms.u-tokyo.ac.jp}

\thanks{The first author is supported by JSPS KAKENHI Grant Number JP19J21165 and JP22J00649/JP22KJ0737. The second author is supported by JSPS Grant-in-Aid for Scientific Research (B) 22H01113 and (C) 18K3209.}
\subjclass[2020]{18G80, 18G35, 13C14, 13D09, 13A02, 16E45, 16G30, 14F08}
\keywords{Cohen-Macaulay module, singularity category, cluster category, tilting theory, enhanced Auslander-Reiten duality, dg orbit category, Calabi-Yau dg category, Grassmannian cluster category}
\dedicatory{Dedicated to Bernhard Keller for his 60th Birthday}

\begin{document}
\begin{abstract}
We establish a Morita theorem to construct triangle equivalences between the singularity categories of (commutative and non-commutative) Gorenstein rings and the cluster categories of finite dimensional algebras over fields, and more strongly, quasi-equivalences between their canonical dg enhancements.
More precisely, we prove that such an equivalence exists as soon as we find a quasi-equivalence between the graded dg singularity category of a Gorenstein ring and the derived category of a finite dimensional algebra which can be done by finding a single tilting object.

Our result is based on two key theorems on dg enhancements of cluster categories and of singularity categories, which are of independent interest. First we give a Morita-type theorem which realizes certain $\Z$-graded dg categories as dg orbit categories.
Secondly, we show that the canonical dg enhancements of the singularity categories of symmetric orders have the bimodule Calabi-Yau property, which lifts the classical Auslander-Reiten duality on singularity categories.

We apply our results to such classes of rings as Gorenstein rings of dimension at most $1$, quotient singularities, and Geigle-Lenzing complete intersections, including finite or infinite Grassmannian cluster categories, to realize their singularity categories as cluster categories of finite dimensional algebras.
\end{abstract}

\maketitle
\setcounter{tocdepth}{1}
\tableofcontents
\part*{Introduction}

Quiver representations and Cohen-Macaulay representations are two of the main subjects in representation theory of orders.
The classical theorems of Gabriel \cite{ASS,ARS} and Buchweitz-Greuel-Schreyer \cite{LW,Yo90} assert that the class of representation-finite quivers and that of representation-finite Gorenstein rings are parametrized by ADE Dynkin diagrams. Moreover, if $R$ is a simple surface singularity of Dynkin type $Q$, then the Auslander-Reiten quiver of the derived category $\D^b(\mod kQ)$ gives a $\Z$-covering of that of the stable category $\sCM R$ of Cohen-Macaulay $R$-modules.
A theoretical explanation of this observation is given by a triangle equivalence 
\begin{equation}\label{RQ}\sCM R\simeq\C_1(kQ),\end{equation}
where $\C_1(kQ)$ is the $1$-cluster category $\D^b(\mod kQ)/\tau$ of $kQ$ {(see \cite[Remark 5.9]{AIR})}.
The aim of this paper is to construct similar type of equivalences between these two classes of categories as we explain below.

The category $\CM R$ of Cohen-Macaulay modules over a Gorenstein ring $R$ forms a Frobenius category, and its stable category $\sCM R$ has a canonical structure of a triangulated category. By \cite{Bu}, it is triangle equivalent to the \emph{singularity category} defined as the Verdier quotient $\sg R:=\D^b(\mod R)/\per R$:
\begin{equation}\label{sCM Dsg}
\sCM R\simeq\sg R.
\end{equation}
It is also classical in Auslander-Reiten theory that, if $R$ is a local isolated singularity of dimension $d$, then $\sg R$ is a $(d-1)$-Calabi-Yau triangulated category \cite{Au78} (that is, $[d-1]$ gives a Serre functor), see Section \ref{section 3.1} for details.

On the other hand, cluster categories are Calabi-Yau triangulated categories introduced in this century. They have been extensively studied in representation theory and around \cite{BMRRT,Am09,Guo,Ke11}, especially in the categorification of cluster algebras \cite{Re,Ke08}. 
Given a finite dimensional algebra $A$, its {\it $n$-cluster category} $\C_n(A)$ is defined as the triangulated hull of the orbit category
\[ \per A/\nu_n\hookrightarrow\C_n(A)  \]
for the autoequivalence $\nu_n:=-\lotimes_A DA[-n]$ of $\per A$, 
see Section \ref{section 2.2} for details.


In this paper we propose a systematic approach to study singularity categories by constructing equivalences with well-studied cluster categories. We crucially work over the differential graded (dg) enhancements of both singularity categories and cluster categories. We expect our approach based on representation theory provides a powerful tool for commutative algebra and algebraic geometry.

\subsection{Our main results}
Let us first describe the main results in a simplified setting. We refer to the next subsection for general forms.
One of the main tools toward our aim is tilting theory on $\Z$-graded singularity categories. For a $\Z$-graded Gorenstein ring $R$, we have the $\Z$-graded singularity category $\sg^\Z\!R:=\D^b(\mod^\Z\!R)/\per^\Z\!R$, which is equivalent to the stable category $\sCM^\Z\!R$ of $\Z$-graded Cohen-Macaulay $R$-modules.
Tilting theory controls derived equivalences of rings. More generally if an algebraic triangulated category $\T$ has a tilting object $U$, then $\T$ is triangle equivalent to the perfect derived category of $\End_\T(U)$.
While ungraded singularity categories $\sg R$ never have tilting objects (unless $\sg R=0$), many examples where the graded singularity category $\sg^\Z\!R$ admits a tilting object have been discovered.
For instance, for a simple surface singularity $R$ of Dynkin type $Q$, there exists a triangle equivalence 
\[\sg^\Z\!R\simeq\per kQ,\]
which is a $\Z$-graded version of \eqref{RQ} (see \cite[Section 5.1]{Iy18}\cite{KST1}). Tilting theory allows us to study $\Z$-graded singularity categories by using well established methods in quiver representation theory, and hence it is actively studied in various branches of mathematics including representation theory, commutative algebra, algebraic geometry and mathematical physics, see e.g.\ \cite{AIR,BIY,DL1,DL2,FU,Ha1,Ha2,Hap,HI,HIMO,HO,IKU,IO13,IT,KST1,KST2,Ki1,Ki2,KLM,LP,LZ,MY,MU,N,SV,U2,Ya} and a survey article \cite{Iy18}. 



Let $R$ be a $\Z$-graded Gorenstein ring of dimension $d$ and with Gorenstein parameter $p$. If there exists a tilting object $U\in\sg^\Z\!R$ with $A:=\End_{\sg R}^\Z(U)$, then by comparing Serre functors,
we have a commutative diagram of equivalences
\[\xymatrix@R=3mm{\per A\ar[r]\ar@{-}[d]^\rsimeq&\per A/\nu_{d-1}\ar@{-}[d]^\rsimeq\ar@{^(->}[r]&\C_{d-1}(A)\\
\sg^\Z\!R\ar[r]&\sg^\Z\!R/(p)\ar@{^(->}[r]&\sg^{\Z/p\Z}\!R.}\]
Since the right inclusions are natural triangulated hulls,
one would naively expect an equivalence $\C_{d-1}(A)\simeq\sg^{\Z/p\Z}\!R$ on the triangulated hulls, and also an equivalence between $\sg R$ and a certain cluster-like category of $A$. However, this is far from being obvious since these triangulated hulls are defined using (a priori) different dg enhancements of both categories and functors. Therefore this was shown only in some special cases on a case-by-case basis \cite{AIR,KRac,KMV}, see also \cite{Am09,ART,IO13,KY,TV} for similar type of results in different settings.

Our main result below justifies the naive expectation above in large generality by constructing quasi-equivalences of dg enhancements of the relevant categories.
For simplicity, here we state our main result in the easiest form. We refer to \ref{tilt intro} for the general version, e.g.\ the same statement holds true for an arbitrary symmetric $R$-order with Gorenstein parameter $p$.

\begin{Thm}[{Morita Theorem for Singularity categories,\,$\subset$\,\ref{tilt}}]\label{intro tilt}
Let $R=\bigoplus_{i\geq0}R_i$ be a positively graded Gorenstein isolated singularity of dimension $d\geq0$ with $R_0=k$ and Gorenstein parameter $p\neq0$. Suppose $\sg^\Z\!R$ has a tilting object $M$ with $A=\End_{\sg R}^\Z(M)$. Then there is a commutative diagram
\[ \xymatrix@R=3mm{
	\per A\ar[r]\ar@{-}[d]^-\rsimeq&\C_{d-1}(A)\ar@{-}[d]^-\rsimeq\ar[r]&\C^{(1/p)}_{d-1}(A)\ar@{-}[d]^-\rsimeq\\
	\sg^\Z\!R\ar[r]&\sg^{\Z/p\Z}\!R\ar[r]&\sg R.} \]
In fact, we have the corresponding diagram for the canonical dg enhancements.
\end{Thm}

Here the category $\C_{d-1}^{(1/p)}(A)$ is the triangulated hull of the orbit category of $\per A$ modulo a $p$-th root of $\nu_{d-1}$ given as the automorphism of $\per A$ corresponding to the degree shift functor on $\sg^\Z\!R$. 
In Part \ref{part2} we apply \ref{intro tilt} to various commutative or non-commutative rings admitting tilting objects.

\bigskip
Let us first discuss enhancements of singularity categories.
Let $R$ be a commutative Gorenstein ring. The {\it dg singularity category} of $R$ is the dg quotient $\D^b_\dg(\mod R)/\per_\dg\!R$ in the sense of \cite{Ke99,Dr}, which is quasi-equivalent to the dg category $\C_\ac(\proj R)_\dg$ of acyclic complexes over finitely generated projective modules.
Our first key result below lifts the Calabi-Yau property of singularity categories to their dg enhancements, where we denote by $(-)^\vee$ the Matlis dual.
\begin{Thm}[{Enhanced Auslander-Reiten duality,\,$\subset$\,\ref{iso}}]\label{intro iso}
Let $R$ be a commutative Gorenstein local isolated singularity of dimension $d$, and $\C$ its dg singularity category. 
Then we have an isomorphism in $\D(\C^\op\otimes_R\C)$:
\[\C^\vee\simeq\C[d-1].\]
\end{Thm}
As a consequence, we obtain the classical Auslander-Reiten duality on the category $\sCM R$ by taking the $0$-th cohomologies of this isomorphism (see \ref{classical AR}). Note that this isomorphism can be viewed as a {\it weak right Calabi-Yau structure} \cite{KS,BD,KW} on $\C$. We also give a graded version \ref{intro gr} of the above enhanced Auslander-Reiten duality as a step toward the proof of \ref{intro tilt}.

\bigskip
We next discuss the canonical enhancements of cluster categories. 
Let $A$ be a dg algebra and $V$ a bimodule over $A$. We will write $V^n$ for $V\otimes_AV\otimes_A\cdots\otimes_AV$ ($n$ factors). Following \cite{Ke05}, we denote by $A/V$ the {\it dg orbit algebra} of $A$ by $V$, see Section \ref{section: cluster category}.
The $n$-cluster category of a finite dimensional algebra $A$ is defined as
\[ \C_n(A):=\per(A/DA[-n]), \]
thus the dg orbit algebra $A/DA[-n]$ gives an enhancement of $\C_n(A)$.
Under a suitable setting, the cluster category $\C_n(A)$ has a special object called an $n$-cluster tilting object \cite{Am09,Guo}. 
By construction, the dg orbit algebra $A/V$ has a structure of an {\it Adams graded dg algebra}, that is, a dg algebra with an additional grading to the usual cohomological grading.
For an Adams graded dg algebra $\Ga$ we can form the {\it graded derived category $\D^\Z(\Ga)$} as well as the {\it graded perfect derived category} $\per^\Z\!\Ga$, see Section \ref{graded} for details.

Another step toward the proof of \ref{intro tilt} is the following characterization of dg orbit algebras among $\Z$-graded dg algebras, where we refer to \ref{qe} for the notion of {\it $\Z$-graded quasi-equivalences}.

\begin{Thm}[{Morita Theorem for Adams graded dg categories,\,$\subset$\,\ref{bGqe}}]\label{orbit-intro}
Let $\Ga$ be a $\Z$-graded dg algebra.
Suppose that $\per^\Z\!\Ga=\thick\Ga$ and put $A=\Ga_0$ and $V=\Ga_{-1}$. Then $\Ga$ is $\Z$-graded quasi-equivalent to the dg orbit algebra $A/V$:
\[ \Ga\simeq A/V. \]
\end{Thm}

As an application, we get a structure theorem of {Calabi-Yau dg algebras}.
We say that an Adams graded dg algebra $\Ga$ is {\it $p$-shifted $d$-Calabi-Yau} if there is an isomorphism $D\Ga\simeq\Ga(-p)[d]$ in $\D^\Z(\Ga^e)$.
\begin{Cor}[{Morita Theorem for Calabi-Yau dg algebras,\,$\subset$\,\ref{futatsu}}]
In the setting of \ref{orbit-intro}, suppose moreover that $\Ga$ is $p$-shifted $d$-Calabi-Yau. Then $A$ and $V$ in \ref{orbit-intro} satisfy $V^{\lotimes_Ap}\simeq DA[-d]$ in $\D(A^e)$, and we have a $\Z$-graded quasi-equivalence $\Gamma\simeq A/V$.
\end{Cor}

\subsection{Our general results}\label{section 0.2}

In this subsection, we explain our results in more detail. 
In fact, we will prove Theorems \ref{intro tilt} and \ref{intro iso} for (not necessarily commutative) module-finite $R$-algebras over commutative Gorenstein rings $R$. Moreover, we consider $R$-algebras which are graded by an arbitrary abelian group $G$, possibly with torsion. For this we first need to prepare some basics on commutative algebra, module-finite algebras, and dg singularity categories in the $G$-graded setting.

Let $R=\bigoplus_{g\in G}R_g$ be a $G$-graded commutative Noetherian ring, and $\La=\bigoplus_{g\in G}\La_g$ a $G$-graded $R$-algebra such that the structure morphism $R\to\La$ preserves the $G$-gradings. We call $\La$ an {\it $R$-order} if $\La$ is (maximal) Cohen-Macaulay as an (ungraded) $R$-module. We say that $\La$ is {\it symmetric} if $\Hom_R(\La,R)\simeq\La$ as (ungraded) $(\La,\La)$-bimodules.
We consider the {\it $G$-graded dg singularity category} $\C$ of $\La$ which enhances the singularity category $\sg\La$, see \ref{grdgsg}, and 
$\C^{\fl}$ the full dg subcategory of $\C$ corresponding to the category $\sg_0\!\La$, see \eqref{define C'}. 

Furthermore, we need the notions of {\it graded dimension $\dim^G\!R$}, the {\it Gorenstein parameter $p_R$}, and {\it $G$-graded Matlis dual}. These graded versions are necessary, for example, to state Auslander-Reiten duality correctly in the graded setting. We refer to Appendix \ref{G-graded rings} for the definitions.
Under a mild assumption, a $G$-graded symmetric $R$-order $\L$ has {\it relative Gorenstein parameter $p_{\La/R}$} in the sense that we have $\Hom_R(\La,R)\simeq\La(-p_{\La/R})$ as $G$-graded $(\La,\La)$-bimodules, see \ref{relative}. More details can be found in Section \ref{grCM}.

Now we are ready to state our general results. The first one on enhanced Auslander-Reiten duality is as follows. We denote by $(-)^{\vee_G}$ the graded Matlis dual, and by $(-)^\ast=\RHom_R(-,R)$ the functor on $\D(\Mod^G\!R)$. The case $G=0$ and $\Lambda=R$ is an isolated singularity is \ref{intro iso}.

\begin{Thm}[{Graded Enhanced Auslander-Reiten duality,\,=\ref{greAR},\ \ref{griso}}]\label{intro gr}
Let $R$ be a $G$-graded commutative Gorenstein ring of $\dim^G\!R=d<\infty$, and $\L$ a $G$-graded symmetric $R$-order with Gorenstein parameter $p_\L$ and relative Gorenstein parameter $p_{\Lambda/R}$.
There exist isomorphisms
\[
\begin{aligned}
\C^\ast&\simeq\C(-p_{\L/R})[-1] &&\mbox{ in } \D^G(\C^\op\otimes_R\C),\\
{\C^{\fl}{}^{\vee_G}}&\simeq\C^{\fl}(-p_\L)[d-1] &&\mbox{ in } \D^G((\C^{\fl})^\op\otimes_R\C^{\fl}).
\end{aligned}
\]
In particular, if $\L$ satisfies the {\rm(R$_{d-1}^G$)}-condition (see \ref{Sing^G}), then there exists an isomorphism:
\[ {\C^{\vee_G}}\simeq\C(-p_\L)[d-1]\ \mbox{ in }\  \D^G(\C^\op\otimes_R\C).\]
Moreover, if $R_0$ is a finite dimensional algebra over a field $k$, then the graded Matlis dual $(-)^{\vee_G}$ in the isomorphisms above can be replaced by the graded $k$-dual $D$ sending $M=\bigoplus_{g\in G}M_g$ to $DM=\bigoplus_{g\in G}\Hom_k(M_{-g},k)$.
\end{Thm}

Taking the $0$-th cohomology we recover the following classical AR duality in the graded setting (cf.\ \cite{AR}), which implies the existence of almost split sequences in the category $\CM^G_0\!\La$.

\begin{Cor}[{Classical Auslander-Reiten duality,\,=\ref{grAR}}]
For each $M,N\in\sCM^G_0\!\L$ we have a natural isomorphism
\[ D\sHom^G_\La(M,N)\simeq\sHom_\La^G(N,M(-p_\L)[d-1]). \]
Therefore, we have the following.
\begin{enumerate}
\item The triangulated category $\sCM^G_0\!\La$ has a Serre functor $(-p_\L)[d-1]$.
\item If $\Lambda$ satisfies {\rm(R$_{d-1}^G$)} condition, then $\sCM^G_0\!\La=\sCM^G\!\L$ has a Serre functor $(-p_\L)[d-1]$.
\end{enumerate}
\end{Cor}



Now we state the general version below of \ref{intro tilt}, which is far more general in the following four points.
\begin{itemize}
\item We deal with (possibly non-commutative) symmetric orders $\La$ over $R$,
\item we do not assume that $R$ and/or $\La$ have an isolated singularity,
\item we allow our grading group $G$ to have torsion and/or to have higher rank, and
\item we deal with any generator of $\sg_0^G\!\La$ which is not necessarily a tilting subcategory.
\end{itemize}
For this we have to consider the \emph{$G$-prime spectrum} $\Spec^G\!R$ and
the \emph{$G$-singular locus} $\Sing_R^G\!\La:=\{\p\in\Spec^G\!R\mid\sg^G\!\La_{\p,G}\neq0\}$, where $\La_{\p,G}$ is the $G$-homogeneous localization (see Appendix \ref{G-graded rings} for details).
The latter is denoted by $\Sing_R\La$ if $G=0$.

Let $\sg_{0,\dg}^G\!\La$ be the canonical dg enhancement of $\sg_0^G\!\La$.
Our main theorem \ref{tilt intro} shows that, for each full dg subcategory $\A$ of $\sg_{0,\dg}^G\!\La$ which generates $\sg_0^G\!\La$, 
there exists a fully faithful functor $\C_{d-1}(\A)\subset\sg_0^{G/(p)}\!\La$ whose image is equivalent to $\sg_0^{G/(p)}\La_{\m,G/(p)}$, and also shows that the functor is an equivalence if $\Sing_R^{G/(p)}\!\L\subset\{\m\}$.
Notice that all commutative diagrams in \ref{tilt intro} lift to their canonical dg enhancements.

\begin{Thm}[{Morita Theorem for Singularity categories in general form,\,$\subset$\,\ref{CM}}]\label{tilt intro}
Let $G$ be an abelian group, $(R,\m)$ a graded Gorenstein $G$-local $k$-algebra with $\dim^G\!R=d$ such that $R_0$ is finite dimensional over $k$, and $\L=\bigoplus_{g\in G}\L_g$ a symmetric $R$-order with Gorenstein parameter $p\in G$ which is torsion-free.
For each full dg subcategory $\A\subset\sg_{0,\dg}^G\!\La$ which generates $\sg_0^G\!\La$ as a thick subcategory,
the following assertions hold.
\begin{enumerate}
\item There exists a commutative diagram of triangle equivalences
\[ \xymatrix@R=3mm{
	\per\A\ar[r]\ar@{-}[d]^-\rsimeq&\C_{d-1}(\A)\ar@{-}[d]^-\rsimeq\\
	\sg_0^G\!\L\ar[r]&\sg_0^{G/(p)}\!\L_{\m,G/(p)}.} \]
If $\Sing_R^{G/(p)}\!\L\subset\{\m\}$, then we can replace $\sg_0^G\!\L$ and $\sg_0^{G/(p)}\!\L_{\m,G/(p)}$ above by $\sg^G\!\L$ and $\sg^{G/(p)}\!\L$ respectively.
\item If $G=\Z$, then there exists a commutative diagram of triangle equivalences
\[ \xymatrix@R=3mm{
	\per\A\ar[r]\ar@{-}[d]^-\rsimeq&\C_{d-1}(\A)\ar@{-}[d]^-\rsimeq\ar[r]&\C^{(1/p)}_{d-1}(\A)\ar@{-}[d]^-\rsimeq\\
	\sg_0^\Z\!\L\ar[r]&\sg_0^{\Z/p\Z}\!\L_{\m,\Z/p\Z}\ar[r]&\sg_0\!\L_\m.} \]
If $\Sing_R\L\subset\{\m\}$, then we can replace $\sg_0^\Z\!\L$, $\sg_0^{\Z/p\Z}\!\L_{\m,\Z/p\Z}$ and $\sg_0\!\L_\m$ above by $\sg^\Z\!\L$, $\sg^{\Z/p\Z}\!\L$ and $\sg\L$ respectively.
\end{enumerate}
\end{Thm}

We have a version of \ref{tilt intro} for hypersurface singularities, where $\Sing^G\!R:=\Sing_R^G\!R$. Recall by matrix factorization that hypersurfaces have $2$-periodic singularity categories. It allows us to give equivalences with cluster categories of various Calabi-Yau dimensions.
\begin{Cor}[{Morita Theorem for hypersurfaces,\,=\ref{hypersurface}}]
In the above setting, assume that $R$ is a hypersurface singularity with Gorenstein parameter $p$.
Let $\A\subset\sg_{0,\dg}^G\!\La$ be a full dg subcategory which generates $\sg_0^G\!\La$ as a thick subcategory. Let $\ell\in\Z$, and assume that $p+\ell c\in G$ is torsion-free.
\begin{enumerate}
\item There exists a commutative diagram of triangle equivalences
\[ \xymatrix@R=3mm{
	\per\A\ar@{-}[d]^-\rsimeq\ar[r]&\C_{d+2\ell-1}(\A)\ar@{-}[d]^-\rsimeq\\
	\sg_0^G\!R\ar[r]&\sg_0^{G/(p+\ell c)}\!R_{\m,G/(p+\ell c)}. } \]
If $\Sing^{G/(p+\ell c)}\!R\subset\{\m\}$, then we can replace $\sg_0^G\!R$ and $\sg_0^{G/(p+\ell c)}\!R_{\m,G/(p+\ell c)}$ above by $\sg^G\!R$ and $\sg^{G/(p+\ell c)}\!R$ respectively.
\item  If $G=\Z$, then there exists a commutative diagram of triangle equivalences
\[ \xymatrix@R=3mm{
	\per\A\ar@{-}[d]^-\rsimeq\ar[r]&\C_{d+2\ell-1}(\A)\ar@{-}[d]^-\rsimeq\ar[r]&\C_{d+2\ell-1}^{(1/p+\ell c)}(\A)\ar@{-}[d]^-\rsimeq\\
	\sg_0^\Z\!R\ar[r]&\sg_0^{\Z/(p+\ell c)\Z}R_{\m,\Z/(p+\ell c)\Z}\ar[r]&\sg_0\!R_\m.} \]
If $\Sing R\subset\{\m\}$, then we can replace $\sg_0^\Z\!R$, $\sg_0^{\Z/(p+\ell c)\Z}\!R_{\m,\Z/(p+\ell c)\Z}$ and $\sg_0\!R_\m$ above by $\sg^\Z\!R$, $\sg^{\Z/(p+\ell c)\Z}\!R$ and $\sg R$ respectively.
\end{enumerate}
\end{Cor}


\subsection{Applications}

As an application of our results, we give a number of equivalences between singularity categories and cluster categories of certain finite dimensional algebras.

First, we apply our results to rings with Krull dimension at most one.
Applying \ref{tilt intro} to tilting objects in the $\Z$-graded singularity categories given in \cite{Ya} for dimension $0$ and \cite{BIY} for dimension $1$, we obtain the following results.

\begin{Thm}[Small Dimensions, $=$\,\ref{dim0}, \ref{1jigen}]\label{intro low}
For dimensions 0 and 1, the following assertions hold.
\begin{enumerate}
\item Let $\L=\bigoplus_{i\ge0}\L_i$ be a finite dimensional non-semisimple symmetric algebra over a field $k$ with $\gd\La_0<\infty$ and with Gorenstein parameter $p$. 
There exists a commutative diagram of triangle equivalences
\[ \xymatrix@R=3mm{
	\D^b(\mod A)\ar@{-}[d]^-\rsimeq\ar[r]&\C_{-1}(A)\ar@{-}[d]^-\rsimeq\ar[r]&\C_{-1}^{(1/p)}(A)\ar@{-}[d]^-\rsimeq\\
	\sg^\Z\!\La\ar[r]&\sg^{\Z/p\Z}\!\La\ar[r]&\sg\La. } \]
\item Let $R=\bigoplus_{i\ge0}R_i$ be a $\Z$-graded commutative Gorenstein ring with dimension one and Gorenstein parameter $p$ such that $R_0$ is a field and $\m:=\bigoplus_{i>0}R_i$.
If $p<0$, then there exists a commutative diagram of triangle equivalences
\[ \xymatrix@R=3mm{
	\per A\ar@{-}[d]^-\rsimeq\ar[r]&\C_{0}(A)\ar@{-}[d]^-\rsimeq\ar[r]&\C_{0}^{(1/p)}(A)\ar@{-}[d]^-\rsimeq\\
	\sg_0^\Z\!R\ar[r]&\sg_0^{\Z/p\Z}\!R_{\m,\Z/p\Z}\ar[r]&\sg_0\!R_\m. } \]
\end{enumerate}
\end{Thm}

In the rest of this section, we study two classes of commutative Gorenstein rings.

The first one is a quotient singularity $R$ by a finite subgroup $G$ of $\SL_d(k)$. Applying \ref{tilt intro} to the tilting object of $\sg^\Z\!R$ with respect to the standard $\Z$-grading of $R$ given in \cite{IT}, we obtain the following realization of the singularity category of $R$ as a cluster category. 

\begin{Thm}[Quotient singularities, =\ref{quot}]
Let $k$ be an algebraically closed field of characteristic $0$ and $G\subset\SL_d(k)$ a finite subgroup. Let $S=k[x_1,\ldots,x_d]$ be the polynomial ring with $\deg x_i=1$ and $R=S^G$ be the quotient singularity which is an isolated singularity. Let $A=\sEnd_R^\Z(T)$ for the maximal Cohen-Macaulay direct summand $T$ of $\bigoplus_{i=1}^d\Om_S^ik(i)$.
Then there exists a commutative diagram of equivalences
\[ \xymatrix@R=3mm{
	\D^b(\mod A)\ar[r]\ar@{-}[d]^-\rsimeq&\C_{d-1}(A)\ar[r]\ar@{-}[d]^-\rsimeq&\C_{d-1}^{(1/d)}(A)\ar@{-}[d]^-\rsimeq\\
	\sg^\Z\!R\ar[r]&\sg^{\Z/d\Z}\!R\ar[r]&\sg R.} \]
\end{Thm}

In particular, we obtain equivalences given by Keller-Reiten \cite{KRac} and Keller-Murfet-Van den Bergh \cite{KMV}, which were the first examples of equivalences between singularity categories and cluster categories.

The second one is a Geigle-Lenzing complete intersection $R$, which is graded by an abelian group $\LL$ of rank one possibly with torsion elements. Applying \ref{tilt intro} to the tilting object given in \cite{HIMO}, we obtain the following equivalence between the singularity category of $R$ and the cluster category of a finite dimensional algebra $A^{\CM}$ called the CM canonical algebra.

\begin{Thm}[Geigle-Lenzing complete intersections, =\ref{GL}]\label{GL in intro}
Let $R$ be a Geigle-Lenzing complete intersection of dimension $d+1$ and with Gorenstein parameter $-\vec{\om}\in\LL$. Suppose that $\vec{\om}\in\LL$ is torsion-free and that $\Sing^{\LL/(\vec{\om})}R\subset\{\m\}$. Then there exists a commutative diagram of equivalences
\[ 	\xymatrix@R=3mm{
	\D^b(\mod A^{\CM})\ar[r]\ar@{-}[d]^-\rsimeq&\C_d(A^{\CM})\ar@{-}[d]^-\rsimeq\\
	\sg^\LL\!R\ar[r]&\sg^{\LL/(\vec{\om})}\!R. } \]

\end{Thm}

As important examples, \ref{GL in intro} and \ref{intro low} include {\it Grassmannian cluster categories.} To categorify cluster algebra structure of the homogeneous coordinate ring of the Grassmannian ${\rm Gr}(n,l)$, Jensen-King-Su \cite{JKS} studied the category
\[\CM^{\Z/n\Z} k[x,y]/(x^\ell-y^{n-\ell})\ \mbox{ with}\ \deg x=1,\deg y=-1,\]
which is called the Grassmannian cluster category. Recently, August-Cheung-Faber-Gratz-Schroll \cite{ACFGS} introduced the infinite analogue as
\[\CM^\Z k[x,y]/(x^\ell)\ \mbox{ with}\ \deg x=1,\deg y=-1.\]
It categorifies the cluster algebra structure of the homogeneous coordinate ring of the infinite Grassmannian, and is called the infinite Grassmannian cluster category.
The following results show that their stable categories are triangle equivalent to usual cluster categories.

\begin{Thm}[Grassmannian cluster categories, =\ref{gra},\ \ref{gra2}]
Let $\ell$ be a positive integer.
\begin{enumerate}
\item For a positive integer $n>\ell$, we have a triangle equivalence
\[\sCM^{\Z/n\Z} k[x,y]/(x^\ell-y^{n-\ell})\simeq\C_2(kA_{\ell-1}\otimes kA_{n-\ell-1}). \]
\item  Let $A$ be a $\Z$-graded finite dimensional $k$-algebra given by a quiver with relations
\[ \xymatrix@!R=2mm{1\ar@2[r]^-x_-y&2\ar@2[r]^-x_-y&\cdots\ar@2[r]^-x_-y&\ell-2\ar[r]_y&\ell-1\ar@(ur,dr)^{w}&& xy-yx,\ \ w^\ell,\ \ wy^2=yx}\]
and degrees $\deg x=1$, $\deg y=0$, $\deg w=1$.
Then there exists a triangle equivalence
\[\sCM_0^\Z k[x,y]/(x^\ell)\simeq\C_2(\proj^\Z\!A).\]
\end{enumerate}
\end{Thm}

Here, $\C_2(\proj^\Z\!A)$ means the $2$-cluster category of the additive category $\proj^\Z\!A$, that is, the triangulated hull of $(\per^\Z\!A)/\nu_2$.

\subsection*{Acknowledgements}
The authors would like to thank Bernhard Keller and Claire Amiot for valuable discussions. We are also grateful to Mitsuyasu Hashimoto and Yuji Yoshino for valuable information on graded rings. We thank Martin Herschend, Lutz Hille, Martin Kalck, Yuta Kimura, Tsutomu Nakamura, Kenta Ueyama, Michael Wemyss and Kota Yamaura for stimulating discussions.

\part{Quasi-equivalences between singularity categories and cluster categories}

The aim of this first part is to give general theoretical results to give equivalences between singularity categories and cluster categories. We will essentially work over enhancements of relevant categories. After preparing some background on dg categories, we give Morita Theorem (Section \ref{sec2}), Enhanced Auslander-Reiten duality (Section \ref{sec3}), which lead to the main results in Section \ref{sec4}.

\section{Preliminaries}

\subsection{Preliminaries on dg categories}

{Throughout this section we fix a base field $k$.
A \emph{dg category} means a dg category over $k$. For dg categories $\A$ and $\B$, an \emph{$(\A,\B)$-bimodule} means a dg $\A^{\op}\otimes\B$-module.
Let $\dgcat=\dgcat_k$ be the category whose objects are small dg categories over $k$ and morphisms are $k$-linear dg functors.} 
When viewing dg categories as enhancements of triangulated categories, it is natural to identify dg categories which have the same derived category. This point of view leads to the following notion.

\begin{Def}
We say that a morphism $\A\to\B$ in $\dgcat$ is a {\it Morita functor} if it induces an equivalence $-\lotimes_\A\B\colon\D(\A)\to\D(\B)$. The {\it Morita homotopy category} $\Hmo$ is the localization of $\dgcat$ with respect to Morita functors.
\end{Def}

Thanks to the model structure on $\dgcat$ whose weak equivalences are Morita functors \cite[5.3]{Tab05b}, the localization $\Hmo$ has small morphism sets. Moreover by homotopy theory of dg categories, this set of morphisms is described quite nicely.
\begin{Prop}[{\cite[5.10]{Tab05b}\cite{To}, see also \cite[Section 4.6]{Ke06}}]\label{TabTo}
The set of morphisms $\A\to\B$ in $\Hmo$ is in one-to-one correspondence with the isomorphism classes in $\D(\A^\op\otimes\B)$ of $(\A,\B)$-bimodules which are perfect as right $\B$-modules.
\end{Prop}


For two dg categories $\A$ and $\B$, we call a $(\A,\B)$-bimodule $X$ \emph{invertible} if $-\lotimes_\A X\colon\D(\A)\to\D(\B)$ is an equivalence. By Proposition \ref{TabTo}, the isomorphism classes of invertible bimodules in $\D(\A^\op\otimes\B)$ are precisely the isomorphisms from $\A$ to $\B$ in $\Hmo$. We say two dg categories $\A$ and $\B$ are {\it derived Morita equivalent} if there exists an invertible $(\A,\B)$-bimodule, or equivalently, they are isomorphic in $\Hmo$.

We prepare the following left-right symmetry of invertible bimodules.

\begin{Lem}\label{derMo}
Let $\A$ and $\B$ be dg categories and $X$ an $(\A,\B)$-bimodule. 
\begin{enumerate}
	\item The following are equivalent.
	\begin{enumerate}
		\item The functor $-\lotimes_\A X\colon\D(\A)\to\D(\B)$ is an equivalence.
		\item The functor $X\lotimes_\B-\colon\D(\B^\op)\to\D(\A^\op)$ is an equivalence.
	\end{enumerate}
	\item Under the above situation, there is an isomorphism $\RHom_\B(X,\B)\simeq\RHom_{\A^\op}(X,\A)$ in $\D(\B^\op\otimes \A)$.
	\end{enumerate}
\end{Lem}
\begin{proof}
	(1)  We only show (a) implies (b). The quasi-inverse to $-\lotimes_\A X$ is given by its adjoint $\RHom_\B(X,-)$. Then the unit and counit maps
	\begin{eqnarray}
			&\A \rightarrow \RHom_\B(X,X),\label{first}\\
			&\RHom_\B(X,\B)\lotimes_\A X\rightarrow \B\label{second}
	\end{eqnarray}
	are isomorphisms in $\D(\A^e)$ and in $\D(\B^e)$, respectively. Now the functor $\RHom_\B(X,-)$ is isomorphic to $-\lotimes_\B\RHom_\B(X,\B)$ since $X$ is compact in $\D(\B)$, thus \eqref{first} becomes
	\begin{equation}
		\A\simeq X\lotimes_\B\RHom_\B(X,\B)\label{third}
	\end{equation}
	in $\D(\A^e)$. Then \eqref{second} and \eqref{third} imply $X\lotimes_\B-\colon\D(\B^\op)\to\D(\A^\op)$ and $\RHom_\B(X,\B)\lotimes_\A-\colon\D(\A^\op)\to\D(\B^\op)$ are mutually inverse equivalences.\\
	(2)  Applying the same argument to the equivalence in $X\lotimes_\B-\colon\D(\B^\op)\to\D(\A^\op)$, we obtain isomorphisms $X\lotimes_\B\RHom_{\A^{\op}}(X,\A)\simeq\A$ in $\D(\A^e)$ and $\B\simeq\RHom_{\A^{\op}}(X,\A)\lotimes_\A X$ in $\D(\B^e)$. Thus we have isomorphisms
$\RHom_\B(X,\B)\simeq\RHom_\B(X,\B)\lotimes_\A X\lotimes_\B\RHom_{\A^\op}(X,\A)\simeq\RHom_{\A^\op}(X,\A)$ in $\D(\B^\op\otimes\A)$.
\end{proof}

Let us recall several facts around exact sequences of dg categories.
We say that sequence $\N\xrightarrow{F}\T\xrightarrow{G}\U$ of triangulated categories is called {\it exact} if $F$ is fully faithful, $G\circ F=0$, and $G$ induces a triangle equivalence $\T/F(\N)\xrightarrow{\simeq}\U$.
\begin{Def}\label{short}
\begin{enumerate}
\item A sequence $\A\xrightarrow{}\B\xrightarrow{}\C$ of dg categories is {\it exact} if the induction functors give an exact sequence of triangulated categories:
\[ \xymatrix{\D(\A)\ar[r]^-{-\lotimes_\A\B}&\D(\B)\ar[r]^-{-\lotimes_\B \C}&\D(\C) }. \]
\item Let $\B$ be a dg category and $\A$ its full dg subcategory. A {\it dg quotient} of $\B$ by $\A$ is a dg category $\B/\A$ together with a morphism $\B\to\B/\A$ which fits into an exact sequence of dg categories $\A\to\B\to\B/\A$.
\end{enumerate}
\end{Def}


The following existence and uniqueness theorem for dg quotients is fundamental.
\begin{Thm}[{\cite[4.6]{Ke99}}]\label{ses}
Let $\A\subset\B$ be dg categories.
\begin{enumerate}
\item The dg quotient of $\B$ by $\A$ exists, which is unique up to unique isomorphism in $\Hmo$.
\item Let $\C$ be the dg quotient of $\B$ by $\A$. Then there exists a triangle in $\D(\B^e)$:
\[ \xymatrix{ \B\lotimes_\A\B\ar[r]&\B\ar[r]&\C\ar[r]&\B\lotimes_\A\B[1] }. \]
\end{enumerate}
\end{Thm}
We will include a proof in Appendix \ref{dgquotient}.


We record one more notion which we will use later.
\begin{Def}
A dg functor $\B\xrightarrow{}\C$ is a {\it localization} if the restriction functor $\D(\C)\to\D(\B)$ is fully faithful.
\end{Def}
For example, any dg quotient $\B\to\B/\A$ is a localization.

\subsection{Dg singularity categories}
Recall that the {\it singularity category} of a Noetherian ring $A$ is the Verdier quotient
\[ \sg A:=\D^b(\mod A)/\per A. \]
We are interested in the singularity category of Gorenstein rings. Let us define the canonical enhancements of these categories.

For a Noetherian ring $A$ we denote by $\D^b_\dg(\mod A)$ (resp. $\per_\dg\!A$) the enhancement of the bounded derived category $\D^b(\mod A)$ (resp. the perfect derived category $\per A$). Explicitly, we may take $\D^b_\dg(\mod A)=\C^{-,b}_\dg(\proj A)$ and $\per_\dg\!A=\C_\dg^b(\proj A)$, the dg category of bounded (above) complexes of finitely generated projective $A$-modules (with bounded cohomologies).
\begin{Def}\label{define sgdg}
The {\it dg singularity category} $\sg_\dg\!A$ of $A$ is the dg quotient of $\D^b_\dg(\mod A)$ by $\per_\dg\!A$, that is,
\[ \sg_\dg\!A:=\D^b_\dg(\mod A)/\per_\dg\! A. \]
\end{Def}

Let us give another description of the dg singularity category. For this suppose that $A$ is \emph{Iwanaga-Gorenstein} in the sense that the free module $A$ has finite injective dimension in $\Mod A$ and in $\Mod A^\op$. 
Then Buchweitz's theorem \cite{Bu} gives an equivalence with the stable category of Cohen-Macaulay modules
\[ \xymatrix{ \sCM A\ar[r]^-\simeq&\sg A }. \]
We also have a triangle equivalence $\sCM A\simeq\K_\ac(\proj A)$ with the homotopy category of acyclic complexes of finitely generated projective modules, which gives an enhancement $\C_\ac(\proj A)_\dg$, the corresponding dg category, of $\K_\ac(\proj A)$. This leads to the following more explicit description of $\sg_\dg\!A$.
\begin{Prop}\label{IGdg}
Let $A$ be an Iwanaga-Gorenstein ring, $\B=\C^{-,b}_\dg(\proj A)$, and $\C=\C_\ac(\proj A)_\dg$. Then the $(\B,\C)$-bimodule $Y$ defined by
\[ Y(C,B)=\cHom_A(C,B) \text{ for } B\in\B, C\in\C \]
induces an isomorphism $\sg_\dg\!A\xsimeq\C_\ac(\proj A)_\dg$ in $\Hmo$.
\end{Prop}
\begin{proof}
	We let $\A\to\B$ be the canonical inclusion. By \cite[A.22]{IO13} we have that $Y\in\Hom_{\Hmo}(\B,\C)$ and $-\lotimes_\B Y\colon\per\B\to\per\C$ identifies with $\D^b(\mod A)\to\sg A$ under canonical equivalences $\per\C=\K_\ac(\proj A)=\sCM A=\sg A$. It follows that the sequence $\per A\to\D^b(\mod A)\to\sg A$ of triangulated categories induced by the sequence $\A\hookrightarrow\B\xrightarrow{Y}\C$ of dg categories is the canonical one, which means that the sequence of dg categories is exact (in the sense of \ref{define exact}). Then both $\B/\A=\sg_\dg\!A$ as well as $\C$ are the cokernel of $\A\to\B$ in $\Hmo$, hence are isomorphic.
\end{proof}

We will use the following instance of \ref{ses}(2).
\begin{Prop}
For each $X,Y\in\D^b(\mod A)$ there is a triangle
\[ \xymatrix{ Y\lotimes_A\RHom_A(X,A)\ar[r]&\RHom_A(X,Y)\ar[r]&\C(X,Y) } \]
functorial in $X$ and $Y$.
\end{Prop}
\begin{proof}
	One can replace $\per_\dg\!A$ by $A$ under the Morita functor $A\xsimeq\per_\dg\!A$. Then the assertion follows from \ref{ses}(2).
\end{proof}

\section{Morita Theorem for Adams graded dg categories and cluster categories}\label{sec2}
We will be interested in dg categories with additional gradings, sometimes called {\it Adams gradings}. After collecting some basic definitions for the graded setting in the first subsection, we give the main result of this section which is a structure theorem of certain $\Z$-graded dg categories. Finally we apply the structure theorem for Calabi-Yau dg categories. 
\subsection{Graded dg categories and graded derived categories}\label{graded}
Throughout this section let $G$ be an abelian group. A {\it $G$-graded dg category} is a $\Z\times G$-graded category $\A$ endowed with a differential of degree $(1,0)$ subject to the Leibniz rule: $d(fg)=df\cdot g+(-1)^if\cdot dg$ for each composable morphisms $f$ and $g$ with $\deg f=(i,a)$.
In other words, it is a category enriched in $\C(\Md^G\!k)$, the category of complexes of $G$-graded $k$-vector spaces and degree $0$ morphisms.

The Adams grading allows such constructions as follows. If $\A$ is a $G$-graded dg category, we have its {\it degree $0$ part $\A_0$} which is the (ordinary) dg category whose morphism complexes are (Adams) degree $0$ part of those of $\A$. More generally, for each subgroup $H$ of $G$, we have the {\it $H$-Veronese subcategory} $\A^{(H)}$ whose morphism complexes are the (Adams) degree $h\in H$ parts of those of $\A$. Then $\A^{(H)}$ is an $H$-graded dg category.

For a $G$-graded dg category $\A$, we have the {\it $G$-graded derived category $\D^G(\A)$}, the localization of the (homotopy) category graded dg $\A$-modules with respect to quasi-isomorphisms. It has the {\it degree shift functor $(a)$} for each $a\in G$ given by the shift of (Adams) grading, with {strict} inverse $(-a)$. Similarly we have the {\it $G$-graded perfect derived category $\per^G\!\A$}, the thick subcategory of $\D^G(\A)$ generated by the shift of representable functors:
\[ \per^G\!\A:=\thick\{\A(-,A)(a)\mid A\in \A, a\in G\}\subset\D^G(\A). \]
For a $G$-graded dg category $\A$ we denote by $\per^G_\dg\!\A$ the canonical enhancement of 
$\per^G\!\A$, that is, the smallest full dg subcategory of the dg category of $G$-graded dg $\A$-modules which is closed under 
mapping cones, $[\pm1]$ and direct summands, and contains $\A(-,A)(a)$ for all $A\in \A$ and $a\in G$.


Formally, these graded derived categories are can be defined using the smash product as follows. Define the dg category $\A\# G$ by 
\begin{itemize}
	\item objects: $(A,a)$ with $A\in\A$ and $a\in G$,
	\item morphisms: $(\A\# G)((A,a),(B,b)):=\A(A,B)_{b-a}$, the (Adams) degree $b-a$ part of $\A(A,B)$.
\end{itemize}
Then we have canonical isomorphisms
\[ \D^{G}(\A)=\D(\A\# G), \qquad \per^G\!\A=\per(\A\# G). \]
The degree shift functor $(a)$ on $\D^G(\A)$ corresponds to the isomorphism of $\D(\A\# G)$ induced by the dg {automorphism} $(A,b)\mapsto(A,b+a)$ of $\A\# G$.

Let us formulate some equivalence relations on graded dg categories which respects the gradings.

\begin{Def}\label{grMo}
We say $G$-graded dg categories $\A$ and $\B$ are {\it {\rm($G$-)}graded derived Morita equivalent} if there exists $X\in\D^G(\A^\op\otimes\B)$ such that $-\lotimes_\A X\colon\D^G(\A)\to\D^G(\B)$ is an equivalence. 
For $G=0$, graded Morita equivalence is simply called {\it Morita equivalence}.
\end{Def}

A particularly special derived Morita equivalence is the following graded notion of quasi-equivalences.
\begin{Def}\label{qe}
We say $G$-graded dg categories $\A$ and $\B$ are {\it {\rm ($G$-)}graded quasi-equivalent} if the there exists $X\in\D^G(\A^\op\otimes\B)$ satisfying the following.
\begin{enumerate}
\renewcommand\labelenumi{(\roman{enumi})}
\renewcommand\theenumi{\roman{enumi}}
\item The functor $-\lotimes_\A X\colon\D^G(\A)\to\D^G(\B)$ is an equivalence.
\item The isomorphism closures of subcategories $\{X(-,A)\mid A\in\A\}$ and $\{\B(-,B)\mid B\in\B\}$ of $\D^G(\B)$ coincide.
\end{enumerate}
For $G=0$, a graded quasi-equivalence is simply called a {\it quasi-equivalence}.
\end{Def}

For given dg categories $\A$ and $\B$, we have the following relationship between these notions, where each implication $\Rightarrow$ is strict.
\[\xymatrix@R=4mm{
	{\begin{array}{c}
			\mbox{$\exists$ quasi-equivalent}\\ \mbox{dg functor $\A\to\B$}\end{array}}\ar@{=>}[r]\ar@{=>}[d]&{\begin{array}{c}\mbox{$\exists$ Morita functor}\\ \A\to\B\end{array}}\ar@{=>}[d]\\
	\mbox{quasi-equivalent}\ar@{=>}[r]&\mbox{Morita equivalent}\ar@{<=>}[r]&\mbox{isomorphic in $\Hmo$}
}\]



Let us note that $G$-graded quasi-equivalences restrict to Veronese subcategories.
Thanks to the condition \ref{qe}(ii), one can replace ``equivalence'' in \ref{qe}(i) by ``fully faithful''.
\begin{Lem}\label{qeVeronese}
If a bimodule $X\in\D^G(\A^\op\otimes\B)$ gives a $G$-graded quasi-equivalence $\A\to\B$, then for every subgroup $H\subset G$, the $H$-Veronese bimodule $X^{(H)}$ yields an $H$-graded quasi-equivalence $\A^{(H)}\to\B^{(H)}$ such that $X\simeq X^{(H)}\lotimes_{\B^{(H)}}\B$ in $\D^G((\A^{(H)})^\op\otimes\B)$. Therefore the following diagram in $\Hmo$ is commutative.
\[ \xymatrix@R=3mm{
	\A^{(H)}\ar[r]\ar[d]_-{X^{(H)}}&\A\ar[d]^-X\\
	\B^{(H)}\ar[r]&\B } \]
\end{Lem}
\begin{proof}
	Since the bimodule $X\in\D^G(\A^\op\otimes\B)$ gives a $G$-graded quasi-equivalence, for each $A\in\A$ there is $B_A\in\B$ such that $\B(-,B_A)\xsimeq X(-,A)$ in $\D^G(\B)$. Taking their $H$-Veronese submodules, we get $\B^{(H)}(-,B_A)\xsimeq X^{(H)}(-,A)$, thus the $(\A^{(H)},\B^{(H)})$-bimodule $X^{(H)}$ is a quasi-functor. 
	
	We next prove that the functor $-\lotimes_{\A^{(H)}}X^{(H)}\colon\D^H(\A^{(H)})\to\D^H(\B^{(H)})$ is fully faithful, which is to show that the map
	\begin{equation}\label{want}
		\A^{(H)}(A,A^\prime)_0\to\RHom^\Z_{\B^{(H)}}(X^{(H)}(-,A),X^{(H)}(-,A^\prime))
	\end{equation}
	is a quasi-isomorphism.
	The right-hand-side of \eqref{want} is equal to $\RHom^\Z_{\B^{(H)}}(\B^{(H)}(-,B_A),\B^{(H)}(-,B_{A^\prime}))=\B^{(H)}(B_A,B_{A^\prime})_0=\B(B_A,B_{A^\prime})_0$.
	On the other hand, since $-\lotimes_\A X\colon\D^G(\A)\to\D^G(\B)$ is an equivalence we have a quasi-isomorphism
	\begin{equation}\label{G}
		\A(A,A^\prime)_0\xsimeq\RHom_\B^\Z(X(-,A),X(-,A^\prime)).
	\end{equation}
	Similarly as above, its right-hand-side is $\RHom_\B^\Z(\B(-,B_A),\B(-,B_{A^\prime}))=\B(B_A,B_{A^\prime})_0$, hence \eqref{want} is a quasi-isomorphism, as desired.
	
	Now we show $\{X^{(H)}(-,A)\mid A\in\A^{(H)}\}=\{\B^{(H)}(-,B)\mid B\in\B^{(H)}\}$ up to isomorphism closure. This is obtained by taking the $H$-Veronese of $\{X(-,A)\mid A\in\A\}=\{\B(-,B)\mid B\in\B\}$.
	
	Finally we prove the isomorphism $X\simeq X^{(H)}\lotimes_{\B^{(H)}}\B$. We have a natural map $X^{(H)}\lotimes_{\B^{(H)}}\B\to X$, so it is enough to show that this gives an isomorphism $X^{(H)}(-,A)\lotimes_{\B^{(H)}}\B\to X(-,A)$ in $\D^G(\B)$ for each $A\in\A$. The left-hand-side is $\B^{(H)}(-,B_A)\lotimes_{\B^{(H)}}\B\simeq\B(-,B_A)$, which is isomorphic to the right-hand-side.
\end{proof} 

We conclude this subsection by introducing the following notion for the grading. Recall that $\per^G\!\A$ is generated by the shifts of representable functors $\{\A(-,A)(a)\mid A\in\A, a\in G\}$. We will write
\[ \thick\A:=\thick\{\A(-,A)\mid A\in\A \}, \]
the thick subcategory of $\D^G(\A)$ generated by the representable functors {\it without degree shifts}.
\begin{Def}
We say a $G$-graded dg category $\A$ is {\it strongly graded} if $\per^G\!\A=\thick\A$.
\end{Def}
Note that this is invariant under graded quasi-equivalence but not under graded derived Morita equivalence.

\subsection{Morita Theorem for Adams graded dg categories}\label{section 2.2}
Let $\A$ be a dg category. We will identify an object $A\in\A$ and an $\A$-module $\A(-,A)$ represented by $A$.
\begin{Def}[{\cite{Ke05}}]
For a cofibrant bimodule $V$ over $\A$, the {\it dg orbit category} \cite{Ke05} $\A/V$ is defined as the dg category with the same objects as $\A$ and the morphism complex
\[ \A/V(L,M)=\colim\left( \xymatrix{\disoplus_{n\geq0}\cHom_\A(L\otimes V^n,M)\ar[r]^-{-\otimes V}&\disoplus_{n\geq0}\cHom_\A(L\otimes V^n,M\otimes V)\ar[r]^-{-\otimes V}&\cdots }\right),  \]
where the tensor products are over $\A$ and $V^n$ is the $n$-fold tensor product of $V$. When $V$ is not cofibrant, we replace it by its cofibrant resolution $pV\to V$ and put $\A/V:=\A/pV$, which does not depend on the choice of a resolution up to quasi-equivalence \cite[9.4]{Ke05}.
\end{Def}


Note that since the morphism complex $\A/V(L,M)$ is the direct sum of colimits of the diagonal maps below induced by $-\otimes V$, 
\[ \xymatrix@!C=36mm{
	\cHom_\A(L,M)\ar[dr]&\cHom_\A(L,M\otimes V)\ar[dr]&\cHom_\A(L,M\otimes V^2)\ar[dr]&\cdots\\
	\cHom_\A(L\otimes V,M)\ar[dr]&\cHom_\A(L\otimes V,M\otimes V)\ar[dr]&\cHom_\A(L\otimes V,M\otimes V^2)\ar[dr]&\cdots\\
	\cHom_\A(L\otimes V^2,M)\ar[dr]&\cHom_\A(L\otimes V^2,M\otimes V)\ar[dr]&\cHom_\A(L\otimes V^2,M\otimes V^2)\ar[dr]&\cdots\\
	\cdots&\cdots&\cdots&\cdots} \]
the orbit category $\A/V$ has with a natural $\Z$-grading with
\[ \A/V(L,M)_i=\colim_{m\gg0}\left( \xymatrix{\cHom_\A(L\otimes V^{i+m},M\otimes V^m)\ar[r]^-{-\otimes V}&\cHom_\A(L\otimes V^{i+m+1},M\otimes V^{m+1})\ar[r]^-{-\otimes V}&\cdots }\right).  \]

In what follows we assume that $V$ is {\it invertible}, that is, the functor $-\lotimes_\A V\colon\D(\A)\to\D(\A)$ is an equivalence. 
Let us note the following fact on which the main result of this section is based.
\begin{Lem}\label{strong}
Consider the dg orbit category $\A/V$ with the above $\Z$-grading.
\begin{enumerate}
\item There are quasi-isomorphisms $\A\xrightarrow{\simeq}(\A/V)_0$ of dg categories and $V\xrightarrow{\simeq}(\A/V)_{-1}$ of $(\A,\A)$-bimodules.
\item $\A/V$ is strongly graded.
\end{enumerate}
\end{Lem}
\begin{proof}
	(1) is clear from the definition of the grading. (2) is because $\A/V(-1)\simeq V\lotimes_\A\A/V$ and $\A/V(1)\simeq\RHom_\A(V,\A)\lotimes_\A\A/V$ in $\D^\Z(\A^\op\otimes(\A/V))$.
\end{proof}

We prove that these properties in fact characterize the dg orbit categories among graded dg categories.
\begin{Thm}[{Morita Theorem for Adams graded categories}]\label{bGqe}
Let $\B$ be a strongly $\Z$-graded dg category and put $\A=\B_0$ and $V=\B_{-1}$. Then $\B$ is $\Z$-graded quasi-equivalent to $\A/V$, and there is a commutative diagram in $\Hmo$.
	\[ \xymatrix@R=3mm{
			\B_0\ar[r]\ar@{=}[d]&\B\ar[d]^-\rsimeq\\
			\A\ar[r]&\A/V
}\]
\end{Thm}

Before proving \ref{bGqe}, we give the following more flexible version which we use later, where $\B$ is not necessarily strongly $\Z$-graded. 

\begin{Cor}\label{many}
Let $\B$ be a $\Z$-graded dg category, and let $\A\subset\per^\Z_\dg\!\B$ be a full dg subcategory which generates $\per^\Z\!\B$. Define a dg bimodule $V$ over $\A$ by $V(X,Y):=\cHom_\B(X,Y)_{-1}$ for $X,Y\in\A$. Then $\B$ is $\Z$-graded derived Morita equivalent to $\A/V$.
\end{Cor}

\begin{proof}
Let $\C$ be a full dg subcategory of $\per_\dg\B$ whose objects are the image of $\A$ under the forgetful functor $\per^\Z_\dg\!\B\to\per_\dg\B$. Then $\C$ is a $\Z$-graded dg category since each object is perfect over $\B$.
Since $\A$ generates $\per^\Z\!\B$, 
it follows that $\B$ and $\C$ are $\Z$-graded derived Morita equivalent and moreover
$\C$ is strongly graded. Thus $\C$ is graded quasi-equivalent to $\C_0/\C_{-1}=\A/V$ by applying \ref{bGqe} to $\C$, and hence $\B$ is graded Morita equivalent to $\A/V$.
\end{proof}

To prove the theorem we need a small preparation.
\begin{Lem}\label{VaX}
	In the setting of \ref{bGqe} we have the following.
	\begin{enumerate}
		\item\label{Va} For each $a\in\Z$ there is an isomorphism $\B_{-a}\lotimes_\A \B\simeq \B(-a)$ in $\D^\Z(\A^\op\otimes\B)$.
		\item\label{ath} For all $a\geq0$ we have $\B_{-1}^{\lotimes_\A a}\simeq \B_{-a}$ in $\D(\A^e)$.
	\end{enumerate}
\end{Lem}
\begin{proof}
	(\ref{Va})  Since $\per^\Z\!\B$ is generated by $\{\B(-,B)\mid B\in\B\}$ and $\RHom^\Z_\B(\B,\B)=\A$, we have mutually inverse equivalences $-\lotimes_\A\B\colon\per \A\to\per^\Z\!\B$ and $\RHom_\B^\Z(\B,-)=(-)_0\colon\per^\Z\!\B\to\per \A$. Then the natural map
	\[ \B_{-a}\lotimes_\A \B\to \B(-a) \]
	in $\D^\Z(\A^\op\otimes\B)$ is nothing but the counit of the adjunction, thus is an isomorphism.\\
	(\ref{ath}) We have a natural multiplication map $\B_{-1}^{\lotimes_\A a}\to \B_{-a}$ in $\D(\A^e)$. It is enough to prove that this is an isomorphism in $\per \A$, which we see by applying the equivalence $-\lotimes_\A\B$ and using (\ref{Va}).
\end{proof}

Now we are ready to prove \ref{bGqe}.

\begin{proof}[Proof of \ref{bGqe}]
	We replace $\B_{-1}$ by its cofibrant resolution $V$ over $\A^e$. Then by \ref{VaX}(\ref{Va}) we have a quasi-isomorphism
	\[ \xymatrix{ \varphi_1\colon V\otimes_\A \B\ar[r]&\B(-1) } \]
	of graded dg $(\A,\B)$-bimodules. We inductively define quasi-isomorphisms $\varphi_n\colon V^n\otimes_\A \B\to \B(-n)$ by
	\[ \xymatrix{ V^{n+1}\otimes_\A \B\ar[rr]^-{1_V^{n}\otimes\varphi_1}&&V^{n}\otimes_\A \B(-1)\ar[r]^-{\varphi_{n}(-1)}& \B(-n-1)}. \]
	It is easily seen that for any $p,q\geq0$ we have $\varphi_{p+q}=\varphi_p(-q)\circ(1_V^p\otimes\varphi_q)$, where $\varphi_0:=1_\B$.

	In what follows we will not distinguish a morphism from its degree shift, for example, the rightmost equality below. We write ${}_\A(-,-)$ for $\cHom_\A(-,-)$ and similarly ${}^\Z_\B(-,-)$ for $\cHom^\Z_\B(-,-)$. Also all tensor products are over $\A$. 
	Consider the diagram of quasi-isomorphisms of $(\A,\A)$-bimodules
	\[ \xymatrix@C=5.4mm{
		{}_\A(V^n,V^p)\ar[r]^-{-\otimes \B}\ar[d]_-{-\otimes V}&{}_\B^\Z(V^n\otimes \B,V^p\otimes \B)\ar[r]^-{\varphi_p\cdot}& {}_\B^\Z(V^n\otimes\B,\B(-p))\ar[d]^-{\cdot(1_V^n\otimes\varphi_1)}&{}_\B^\Z(\B(-n),\B(-p))\ar@{=}[d]\ar[l]_-{\cdot\varphi_n} \\
		{}_\A(V^{n+1},V^{p+1})\ar[r]^-{-\otimes \B}&{}_\B^\Z(V^{n+1}\otimes \B,V^{p+1}\otimes \B)\ar[r]^-{\varphi_{p+1}\cdot}& {}_\B^\Z(V^{n+1}\otimes \B,\B(-p-1))&{}_\B^\Z(\B(-n-1),\B(-p-1)),\ar[l]_-{\cdot\varphi_{n+1}} } \]
	which can easily be seen to be commutative. Let $T$ be the colimit of some coproduct of the third column;
	\[ T=\colim\left( \xymatrix{\disoplus_{n\geq0}{}_\B^\Z(V^n\otimes \B,\B)\ar[r]&\disoplus_{n\geq0}{}_\B^\Z(V^n\otimes \B,\B(-1))\ar[r]&\disoplus_{n\geq0}{}_\B^\Z(V^n\otimes \B,\B(-2))\ar[r]&\cdots }\right). \]
	
	We claim that $T$ gives rise to a graded $(\B,\A/V)$-bimodule, which induces a graded quasi-equivalence $\B\to\A/V$.
	Note first that $T$ takes $B\in\B$ and $A\in\A$ to the graded complex
	\[ T(A,B)=\colim\left( \xymatrix{\disoplus_{n\geq0}{}_\B^\Z(A\otimes V^n\otimes \B,\B(-,B))\ar[r]&\disoplus_{n\geq0}{}_\B^\Z(A\otimes V^n\otimes \B,\B(-,B)(-1))\ar[r]&\cdots }\right), \]
	whose degree $i$ part is
	\[ T(A,B)_i=\colim_{m\gg0}\left( \xymatrix@C=4mm{{}_\B^\Z(A\otimes V^{i+m}\otimes \B,\B(-,B)(-m))\ar[r]&{}_\B^\Z(A\otimes V^{i+m+1}\otimes \B,\B(-,B)(-m-1))\ar[r]&\cdots }\right). \]
	We see from the above descriptions that $T$ has a structure of a left $\B$-module. It also has a well-defined right $\A/V$-action as follows: for $c\in\A/V(A^\prime,A)$ presented by a morphism $c\colon A^\prime\otimes V^n\to A\otimes V^p$ of $\A$-modules and $t\in T(A,B)$ presented by a graded dg $\B$-module morphism $t\colon A\otimes V^m\otimes \B(-,B)\to \B(-,B)(-q)$, set $tc\colon A^\prime\otimes V^{n+m}\otimes \B\xrightarrow{c\otimes1_V^m\otimes1_\B}A\otimes V^{p+m}\otimes \B(-,B)\xrightarrow{1_V^m\otimes\varphi_p}A\otimes V^m\otimes \B(-,B)(-p)\xrightarrow{t}\B(-,B)(-p-q)$. These actions make $T$ into a graded dg $(\B,\A/V)$-bimodule. By the commutativity of left hexagon in the above diagram, there is a homogeneous quasi-isomorphism $f\colon \A/V\to T$ of $(\A,\A)$-bimodules, which can be seen to be right $\A/V$-linear. Therefore $T$ is a quasi-functor $\B\to \A/V$, that is, $T(-,B)$ is isomorphic in $\D^\Z(\A/V)$ to a representable $\A/V$-module $\A/V(-,B)$. Also by the commutativity of right square, there is a quasi-isomorphism $g\colon\B\to T$ of $\B$-modules.
	
	To prove that $T$ is a graded quasi-equivalence we have to show that the morphism $\RHom_\B(\B,\B)\to\RHom_{\A/V}(T,T)$ induced by $-\lotimes_\B T$ is a quasi-isomorphism, but this follows from the commutativity of the diagram below.
	\[ \xymatrix@R=5mm{
		\RHom_\B(\B,\B)\ar@{=}[d]\ar[rr]&&\RHom_{\A/V}(T,T)\ar[d]^-{\cdot f} \\
		\B\ar[r]^-g&T\ar@{=}[r]&\RHom_{\A/V}(\A/V,T) } \]
	Since $\B$ and $\A/V$ have the same objects, it is clear that the functor $H^0(\B)_0\to H^0(\A/V)_0$ induced by $T$ is essentially surjective. This completes the proof that $T$ is a quasi-equivalence.
	
	The commutativity of the diagram follows from \ref{qeVeronese}.
\end{proof}

We can reprove \cite[7.1]{ha3} using Morita theorem above. The proof is simplified thanks to the Adams grading empolyed here. Recall that bimodules $X$ and $Y$ over a dg category $\A$ are {\it mutually inverse} if $X\lotimes_\A Y\simeq\A\simeq Y\lotimes_\A X$ in $\D(\A^e)$.
\begin{Cor}
Let $\A$ be a dg category, and let $X$ and $Y$ be mutually inverse bimodules over $\A$. Then the dg orbit categories $\A/X$ and $\A/Y$ are quasi-equivalent.
\end{Cor}
\begin{proof}
	Let $\B=\A/X$ be the dg orbit category of $\A$ by $X$, which by \ref{strong} is a strongly $\Z$-graded dg category with a quasi-equivalence $\B_0\simeq\A$ and isomorphisms $\B_{-1}\simeq X$ and also $\B_1\simeq Y$ in $\D(\A^e)$. Now give an inverted grading on $\B$, which we denote by $\C$, so that $\C$ is a strongly $\Z$-graded dg category with $\C_i=\B_{-i}$. By \ref{bGqe} we obtain $\C\simeq\C_0/\C_{-1}=\A/Y$. Since $\B$ and $\C$ have the same underlying (ungraded) dg category we deduce the conclusion.
\end{proof}

For later use, we note the following general observation relating orbit categories and Veronese subcategories.
\begin{Prop}\label{veron}
	Let $\C$ be a strongly $G$-graded dg category. Then for any subgroup $H\subset G$, there exists a commutative diagram
	\[ \xymatrix@R=3mm{
		\C_0\ar[r]\ar[d]^-\rsimeq&\C^{(H)}\ar[d]^-\rsimeq\\
		\per^G_\dg\!\C\ar[r]&\per^{G/H}_\dg\!\C} \]
	where the lower horizontal map is the forgetful functor, $\C^{(H)}$ is the Veronese subcategory, the upper horizontal map is the inclusion, and the vertical maps are Morita functors.
\end{Prop}
\begin{proof}
	Since $\per^G\C=\thick\C$ and $\REnd_\C^G(\C)=\C_0$ we have a Morita functor $\C_0\to\per_\dg^G\C$ in the left column. Similarly, we also have $\per^{G/H}\!\C=\thick\C$ and $\REnd_\C^{G/H}(\C)=\C^{(H)}$, which give the Morita functor $\C^{(H)}\to\per^{G/H}_\dg\C$ in the right column. It is clear that the diagram is commutative.
\end{proof} 


\subsection{Cluster categories}\label{section: cluster category}

Let us first prepare some terminologies on dg categories.
\begin{Def}\label{kotoba}
Let $\A$ be a dg category.
\begin{enumerate}
\item We say $\A$ is {\it component-wise proper} if each cohomology of each morphism complex is finite dimensional.
\item We say that a component-wise proper dg category $\A$ is {\it Gorenstein} if $\thick\{\A(-,A)\mid A\in\A\}=\thick \{D\A(A,-)\mid A\in\A\}$ in $\D(\A)$.
\item A Gorenstein dg category $\A$ is called {\it $\nu_d$-finite} if for each $X,Y\in\per\A$ we have $\Hom_{\D(\A)}(X,\nu_d^{-i}Y)=0$ for almost all $i\in\Z$, where $\nu_d=-\lotimes_\A D\A[-d]$ is the composite of the Serre functor and $[-d]$.
\end{enumerate}
\end{Def}
Clearly, $\A$ is component-wise proper if and only if $\per\A$ is $\Hom$-finite, and it is well-known that $\A$ is Gorenstein if and only if $\per \A$ has a Serre functor. Note that in the definition of $\nu_d$-finiteness we do not assume $\gd A\leq d$ even if $\A$ is an ordinary algebra $A$, so our $\nu_d$-finiteness is weaker than the notion already established in \cite{Am09,Iy11}.

Let us recall some basic results on cluster categories, from the viewpoint of their enhancements. Let $\A$ be a dg category which is {component-wise proper and Gorenstein (see \ref{kotoba}), and let $d$ be an arbitrary integer.} 
Let
\[\bG_d(\A):=\A/D\A[-d]\]
be the dg orbit category. Then the {\it $d$-cluster category} of $\A$ is the perfect derived category
\[ \C_d(\A):=\per\bG_d(\A),\]
see \cite{Am09,Ke05}.
The induction functor $-\lotimes_\A\bG_d(\A)\colon\per\A\to\C_d(\A)$ induces a fully faithful functor $\per\A/-\lotimes_\A D\A[-d]\hookrightarrow\C_d(\A)$ so that $\C_d(\A)$ is the triangulated hull of the orbit category.
We also need the following more general definition.

\begin{Def}
Let $p\neq0$ be an integer. For a bimodule complex $V\in\D(\A^e)$ such that $V^{\lotimes_\A p}\simeq D\A[-d]$ in $\D(\A^e)$, let
\[\bG_d^{(1/p)}(\A):=\A/V\]
be the dg orbit category $\A/V$. We define the \emph{$p$-folded $d$-cluster category} of $\A$ as
\[ \C_d^{(1/p)}(\A):=\per\bG_d^{(1/p)}(\A).\]
\end{Def}

Notice that $\bGa_d^{(1/p)}(\A)$ and $\C_d^{(1/p)}(\A)$ depend on a choice of $V$, hence the notation contains an ambiguity
Similarly to the case $p=1$ it is the triangulated hull of $\per\A/-\lotimes_\A V$, and can be seen as a $\Z/p\Z$-quotient of $\C_d(\A)$ \cite{ha3}. The dg category $\bG_d(\A)$ is quasi-equivalent to the $p$-th Veronese subcategory of $\bG_d^{(1/p)}(\A)$, and the induction functors along the canonical maps $\A\to\bG_d(\A)\to\bG_d^{(1/p)}(\A)$ induce projection functors
\[ \per\A\to\C_d(\A)\to\C_d^{(1/p)}(\A). \] 

\begin{Rem}
\begin{enumerate}
\item Let $A$ be a finite dimensional $k$-algebra which is Gorenstein. Regarding $A$ as a dg category, we obtain the $d$-cluster category $\C_d(A)$. When $\gd A\leq d$ and $A$ is $\nu_d$-finite, it coincides with the usual $d$-cluster category \cite{BMRRT,Ke05,Am09,Guo}
\item Let $A$ be a finite dimensional $k$-algebra of global dimension $\leq d$ and $\nu_d$-finite. The dg orbit algebra $\bG_d(A)$ is related to the $(d+1)$-Calabi-Yau completion $\bPi_{d+1}(A)$ \cite{Ke11} in the following way. Since there is a quasi-isomorphism $\bigoplus_{n\geq0}\cHom_A(X^n,A)\xsimeq\bPi_{d+1}(A)$ for (any) cofibrant resolution $X\to DA[-d]$, we have a natural morphism $\bPi_{d+1}(A)\to\bG_d(A)$. One can show (e.g. by \cite[3.3]{ha4}) the induction functor along this map yields an equivalence
\[\per\bPi_{d+1}(A)/\D^b(\bPi_{d+1}(A))\xsimeq\per\bG_d(A)=\C_d(A)\]
given in \cite{Am09,Guo}.
\end{enumerate}
\end{Rem}

\subsection{Morita Theorem for Calabi-Yau dg categories}

Now we specialize Morita theorem \ref{bGqe} to a CY-situation. Let us define the class of dg categories we are interested in. Note that the following notion ``Calabi-Yau'' means ``right'' Calabi-Yau in the sense of \cite{KS,BD,KW}, which is different from the usual notion of ``left'' Calabi-Yau algebras in the sense of Ginzburg and Keller \cite{G,Ke11}.

\begin{Def}
Let $G$ be an abelian group. We say that a $G$-graded dg category $\C$ is {\it $p$-shifted $d$-Calabi-Yau} if we have an isomorphism
\[ D\C\simeq\C(-p)[d] \text{ in } \D^G(\C^e), \]
that is, we have a natural isomorphism $D\C(N,M)\simeq\C(M,N)(-p)[d]$ in $\D(\Mod^G\!k)$ for all $M,N\in\C$.
\end{Def}

Note that we can modify the CY dimension $d$ of $\C$ flexibly as follows (see e.g.\,\cite{IQ}).

\begin{Rem}
For a $G$-graded dg category $\C$ and a group homomorphism $\phi:G\to \Z$, define a new $G$-graded dg category $\C(\phi)$ as follows:
\begin{itemize}
\item objects: same as $\C$,
\item morphisms: $\C(\phi)(L,M)^i_a:=\C(L,M)^{i+\phi(a)}_a$,
\item composition: $f\cdot g:=(-1)^{i\phi(b)} fg$ for $f\in\C(\phi)^i_a$ and $g\in\C(\phi)^j_b$, where the right-hand-side is the composition in $\C$,
\item differential: $d^\phi(f):=(-1)^{\phi(a)}df$ for $f\in\C(\phi)^i_a$, where the right-hand-side is the differential of $\C$.
\end{itemize}
This definition satisfies the Leibniz rule for $\C(\phi)$, making it into a $G$-graded dg category.

Then if $\C$ is $p$-shifted $d$-Calabi-Yau, then $\C(\phi)$ is $p$-shifted $(d+\phi(p))$-Calabi-Yau.
\end{Rem}

This CY property is Morita invariant in the following sense.
\begin{Lem}\label{mi}
Suppose $\A$ and $\B$ are two $G$-graded derived Morita equivalent dg categories. Then $\A$ is $p$-shifted $d$-CY if and only if $\B$ is $p$-shifted $d$-CY.
\end{Lem}
\begin{proof}
	There is a graded $(\A,\B)$-bimodule $X$ such that $-\lotimes_\A X\colon\D^G(\A)\to\D^G(\B)$ is an equivalence. By \ref{derMo} we have an isomorphism $\RHom_{\A^\op}(X,\A)\simeq\RHom_\B(X,\B)$ in $\D^G(\B^\op\otimes\A)$. Suppose that $\A$ is $p$-shifted $d$-CY, i.e. we have $D\A\simeq \A(-p)[d]$ in $\D(\A^e)$. In particular, each cohomology of $\A$ is finite dimensional, hence so is that of $\B$ by the derived equivalence $\per\A\simeq\per\B$. Dualizing the above isomorphism gives $D\A\lotimes_\A X\simeq X\lotimes_\B D\B$ in $\D^G(\A^\op\otimes\B)$. Then, using \ref{derMo} we obtain isomorphisms $D\B\xrightarrow{\simeq}\RHom_{\A^\op}(X,X\lotimes_\B D\B)\simeq\RHom_{\A^\op}(X,D\A\lotimes_\A X)=\RHom_{\A^\op}(X,X)(-p)[d]\xleftarrow{\simeq}\B(-p)[d]$ in $\D^G(\B^e)$.
\end{proof}

Let $\C$ be a $\Z$-graded dg category which is $p$-shifted $d$-CY. Suppose as in \ref{bGqe} that $\C$ is strongly graded and put $\A=\C_0$, $V=\C_{-1}$, and also $X=\C_{-p}$. In this case we have some additional information to \ref{VaX} on the dg category $\A$ and the bimodules $V$ and $X$.
\begin{Lem}\label{root}
\begin{enumerate}
	\item\label{gor} $\A$ is component-wise proper and Gorenstein.
	\item\label{nf} $\A$ is $\nu_d$-finite.
	\item\label{nud} $X\simeq D\A[-d]$ in $\D(\A^e)$, thus $-\lotimes_\A V$ gives an $p$-th root of $\nu_d$ on $\per \A$.
\end{enumerate}
\end{Lem}
\begin{proof}
	(\ref{gor})  We have an equivalence $\per^\Z\!\C\simeq\per\A$. Since $\per^\Z\!\C$ has a Serre functor $-\lotimes_\A D\A=(-p)[d]$, so does $\per \A$, hence $\per\A$ is $\Hom$-finite and $\A$ is Gorenstein.\\	
	(\ref{nud})  Taking the degree $0$ part of the isomorphism $D\C[-d]\simeq\C(-p)$, we have $D\A[-d]\simeq\C_{-p}$ in $\D(\A^e)$, in which the right-hand-side is $X$. Then the second assertion follows from \ref{VaX}(\ref{ath}).\\
	(\ref{nf})  By \ref{bGqe} we may assume $\C=\A/V$. By (\ref{nud}) is enough to show $\Hom_{\D(\A)}(X,Y\lotimes_\A V^i)=0$ for almost all $i\in\Z$, which is to say $\per\C$ is $\Hom$-finite. Now we have $D\C\simeq\C$ in $\D(\C^e)$, thus each cohomology of $\C$ is finite dimensional, hence $\per\C$ is $\Hom$-finite. 
\end{proof}


Now we state the following prototype of the main result of this paper.

\begin{Thm}[Morita Theorem for Calabi-Yau dg categories]\label{futatsu}
Let $\C$ be a strongly $G$-graded dg category which is $p$-shifted $d$-CY for some torsion-free $p\in G$, and put $\A=\C_0$.
\begin{enumerate}
\item $\C^{(p)}$ is $\Z$-graded quasi-equivalent to $\bG_d(\A)$, and there is a commutative diagram
\[ \xymatrix@R=4mm{
	\A\ar[r]\ar[d]^-\rsimeq&\bG_d(\A)\ar[d]^-\rsimeq\\
	\per^G_\dg\C\ar[r]&\per^{G/(p)}_\dg\C}\]
whose vertical maps are isomorphisms in $\Hmo$.
\item Suppose moreover that $G=\Z$ and put $V=\C_{-1}$. Then $V^{\lotimes_\A p}\simeq D\A[-d]$ holds in $\D(\A^e)$, $\C$ is $\Z$-graded quasi-equivalent to $\bG_d^{(1/p)}(\A)=\A/V$, and the above diagram extends to the following.
\[ \xymatrix@R=4mm{
	\A\ar[r]\ar[d]^-\rsimeq&\bG_d(\A)\ar[d]^-\rsimeq\ar[r]&\bG_d^{(1/p)}(\A)=\A/V\ar[d]^-\rsimeq\\
	\per^\Z_\dg\C\ar[r]&\per^{\Z/p\Z}_\dg\C\ar[r]&\per_\dg\C}\]
\end{enumerate}
\end{Thm}
\begin{proof}
	(1)  By \ref{veron} we have a commutative diagram
	\[ \xymatrix@R=3mm{
		\C_0\ar[r]\ar[d]^-\rsimeq&\C^{(p)}\ar[d]^-\rsimeq\\
		\per_\dg^G\C\ar[r]&\per_\dg^{G/(p)}\C} \]
	with vertical Morita functors. We have $\C_0=\A$ by definition, and consider the Veronese subcategory $\B:=\C^{(p)}$. Since $p\in G$ is torsion-free, it is a $\Z$-graded dg category. Taking the $p$-th Veronese of the isomorphism $D\C\simeq\C(-p)[d]$, we get $D\B\simeq\B(-1)[d]$ in $\D^\Z(\B^e)$. Also, in view of the functor $\per^G\!\C\to\per^\Z\!\B$ given by $M\mapsto M^{(p)}$, we see that $\per^G\!\C=\thick\C$ implies $\per^\Z\!\B=\thick\B$. Then by \ref{bGqe} we conclude that $\B$ is quasi-equivalent to $\B_0/\B_{-1}$, with $\B_0=\A$ and $\B_{-1}=D\A[-d]$ by \ref{root}.\\
	(2)  We have a $\Z$-graded quasi-equivalence $\A/V\simeq\C$ by \ref{bGqe}. Also \ref{VaX}(\ref{ath}) shows $V^{\lotimes_\A p}=\C_{-p}$, which is isomorphic to $D\A[-d]$ by \ref{root}(\ref{nud}). Therefore we have $\A/V=\bG_d^{(1/p)}(\A)$ and also $(\A/V)^{(p)}=\A/V^p=\bG_d(\A)$.
	
	Now, taking the $p$-th Veronese subcategories of $\bG_d^{(1/p)}(\A)\simeq\C$ we get by \ref{qeVeronese} the commutative square on the right below. Also, taking the degree $0$-part of $\bG_d(\A)\simeq\C^{(p)}$ we get again by \ref{qeVeronese} the commutative square on the left below.
	\[ \xymatrix@R=3mm@!C=15mm{
		\A\ar[r]\ar@{=}[d]&\bG_d(\A)\ar[d]^-\rsimeq\ar[r]&\bG^{(1/p)}_d(\A)\ar[d]^-\rsimeq\\
		\C_0\ar[r]&\C^{(p)}\ar[r]&\C\\ } \]
	On the other hand, applying \ref{veron} to $G=\Z$ and $H=p\Z$ we get a commutative square on the left below. Next we view $\C$ as a dg category graded by $G=\Z/p\Z$ and $H=G$, we get a commutative square on the right below.
	\[ \xymatrix@R=3mm{
		\C_0\ar[r]\ar[d]^-\rsimeq&\C^{(p)}\ar[r]\ar[d]^-\rsimeq&\C\ar[d]^-\rsimeq\\
		\per_\dg^\Z\!\C\ar[r]&\per_\dg^{\Z/p\Z}\!\C\ar[r]&\per_\dg\!\C. } \]
	We get the desired result by combining these diagrams.
\end{proof}

Immediately we obtain the following structure theorem of strongly $\Z$-graded $1$-shifted $d$-Calabi-Yau dg categories.

\begin{Cor}
Let $\C$ be a strongly $\Z$-graded dg category which is $1$-shifted $d$-CY, and $\A:=\C_0$. Then $\C$ is $\Z$-graded quasi-equivalent to $\bG_d(\A)$.
\end{Cor}

\section{Enhanced Auslander-Reiten duality for singularity categories}\label{sec3}
We give an important class of CY triangulated categories whose dg enhancement is CY, namely the singularity categories of commutative Gorenstein rings, or more generally, of symmetric orders.
After collecting some background on Cohen-Macaulay representations, we give our main result \ref{eAR} and its graded version \ref{greAR}.

\subsection{Preliminaries on Cohen-Macaulay representations}\label{section 3.1}

Let $R$ be a commutative Noetherian ring of dimension $d$. We denote by $\md R$ the category of finitely generated $R$-modules. The \emph{dimension} $\dim X$ of $X\in\mod R$ is the Krull dimension of $R/\Ann X$.
When $R$ is {local}, the \emph{depth} $\depth X$ of $X$ is the maximal length of the $X$-regular sequence. Then the inequalities
\[\depth X\le\dim X\le d=\dim R\]
hold. We call $X$ (maximal) \emph{Cohen-Macaulay} (\emph{CM}) if the equality $\depth X=d$ holds or $X=0$.

For general $R$, we call $X\in\mod R$ \emph{Cohen-Macaulay} (\emph{CM}) if, for each maximal ideal $\m$ of $R$, $X_\m$ is a CM $R_\m$-module. In this case, $X_\p$ is a CM $R_\p$-module for each $\p\in\Spec R$. We denote by $\CM R$ the category of CM $R$-modules.


For an $R$-algebra $\Lambda$ which is finitely generated as an $R$-module, let
\begin{eqnarray*}
\CM\Lambda&:=&\{X\in\md\Lambda\mid X\in\CM R\ \mbox{ as an $R$-module}\},\\
\CM_0\!\Lambda&:=&\{X\in\CM\Lambda\mid \mbox{$X_\p\in\proj \Lambda_\p$ for all $\p\in\Spec R$ with $\height\p<d$}\}.
\end{eqnarray*}

\begin{Def}\label{symm}
Let $\L$ be a module-finite $R$-algebra.
\begin{enumerate}
\item We call $\L$ an {\it $R$-order} if $\L\in\CM\L$.
\item We call $\L$ a {\it symmetric $R$-algebra} if we have $\L\simeq\Hom_R(\L,R)$ as $(\L,\L)$-bimodules.
\item We call $\L$ a \emph{symmetric $R$-order} if it is an $R$-order which is a symmetric $R$-algebra.
\end{enumerate}
\end{Def}

Now we assume that $R$ is a \emph{Gorenstein} ring, that is, for any prime ideal $\p\in\Spec R$ of $R$, the localization $R_\p$ has finite injective dimension.
In this case, each symmetric $R$-order $\L$ is an Iwanaga-Gorenstein ring whose injective dimensions on both sides equal $\dim R$.
Let $\fl R$ be the category of $R$-modules of finite length, and 
\[ \fl_0\!R=\{ M\in\fl R\mid \text{any }\p\in\Supp_R M \text{ satisfies }\height\p=d\}. \]
For a symmetric $R$-order $\Lambda$, consider a subcategory 
\begin{equation}\label{ungraded Dsg0}
\sg_0\!\L:=\thick\{X\in\mod\L\mid X\in\fl_0\!R\}
\end{equation}
of the singularity category.
On the other hand, by the first equality of \ref{B16 2} we have
\begin{equation*}\label{sCM0}
	\sCM_0\!\L=\{ X\in\sCM\La\mid \sEnd_\La(X)\in\fl_0\!R\}.
\end{equation*}
By \ref{B16}(2) (for $G=0$) we have a triangle equivalence
\begin{equation*}\label{CM0=sg0}
\sCM_0\!\L\simeq\sg_0\!\L.
\end{equation*}

Let us introduce the following which will be important in the sequel.
\begin{Def}\label{Sing}
Let $\La$ be a module-finite $R$-algebra. The {\it singular locus} of $\La$ is
\[ \Sing_R\!\La:=\{ \p\in\Spec R\mid \Mod\La_\p \text{ has infinite global dimension} \}. \]
Also, Serre's {\it {\rm(}R$_{n}${\rm)}-condition} on $\La$ is the following
\[ \Sing_R\!\La\subset\{ \p\in\Spec R\mid \height\p>n\}. \]
We say that $\La$ is an \emph{isolated singularity} if it satisfies (R$_{d-1}$)-condition for $d=\dim R$.
\end{Def}
Note that $\Lambda$ satisfies (R$_{d-1}$) if and only if $\sg\L=\sg_0\!\L$ holds, see \ref{B16}.
Our terminology of isolated singularities is slightly stronger than the usual one, that is, $\Sing_R\!\La\subset\Max R$. They are equivalent if $R$ is local, or more generally, all maximal ideals of $R$ have height $d$.

\subsection{Enhanced and classical Auslander-Reiten duality}
Let $R$ be a commutative Gorenstein ring of finite Krull dimension, and let $\L$ be a symmetric $R$-order. Thus, $\L$ is finitely generated and Cohen-Macaulay as an $R$-module, and there is an isomorphism $\Hom_R(\L,R)\simeq\L$ as $(\L,\L)$-bimodules. We put
\[ \D_\L:=\D_{\mod\L}(\Mod\L)=\{ X\in\D(\Mod\L)\mid H^iX\in\mod\L \text{ for all } i\in\Z\}. \]
Then there is a duality \cite[V.2.1]{RD}
\[ \xymatrix{
	(-)^\ast=\RHom_R(-,R)\colon \D_R\ar@{<->}[r]& \D_R }, \]
which restricts to a duality $\D_\L\leftrightarrow\D_{\L^\op}$.
Note that this coincides with $\RHom_\L(-,\L)$. Indeed, we have $\L\simeq\Hom_R(\L,R)=\RHom_R(\L,R)$ since $\L$ is a symmetric $R$-order, so $\RHom_\L(-,\L)=\RHom_\L(-,\RHom_R(\L,R))=\RHom_R(-,R)$ by adjunction. It follows that there is a functorial isomorphism $1\xrightarrow{\simeq}\RHom_R(\RHom_\L(-,\L),R)$ on $\D_\L$.

Now we state our main result in this section. We refer to \ref{define sgdg} for the definition of $\sg_\dg\!\L$. For example, one can take $\C=\C_\ac(\proj\La)_\dg$ by \ref{IGdg}.

\begin{Thm}\label{eAR}
Let $\C:=\sg_\dg\!\L$ be the dg singularity category of $\L$. Then we have an isomorphism
\[ \C\simeq\C^\ast[1] \]
in $\D(\C^\op\otimes_R\C)$, that is, we have a natural isomorphism $\C(M,N)\simeq\RHom_R(\C(N,M),R)[1]$ in $\D_R$ for each $M,N\in\sg\La$.
\end{Thm}

The essential structure for the proof of \ref{eAR} is the following commutative diagram.
\begin{Prop}\label{star}
For $M,N\in\D^b(\mod\La)$, the diagram below of canonical morphisms is commutative in $\D(\Mod R)$.
\[ \xymatrix{
	\RHom_\L(M,N)^\ast\ar[r]^-{\cdot\a_{M,N}}&(N\lotimes_\L\RHom_\L(M,\L))^\ast \\
	M\lotimes_\L\RHom_\L(N,\L)\ar[r]^-{\a_{N,M}}\ar[u]^\lsimeq&\RHom_\L(N,M)\ar[u]_\rsimeq} \]
\end{Prop}
\begin{proof}
	We give a concrete description of the relevant maps. Replacing $M$ and $N$ by their cofibrant resolutions we assume that $M,N\in\C^{-,b}(\proj\L)$. Then $\a_{M,N}$ is presented by the natural morphism denoted again by $\a_{M,N}$
	\[ \xymatrix{ {\a_{M,N}}\colon N\otimes_\L\cHom_\L(M,\L)\ar[r]& \cHom_\L(M,N), } \]
	and $\C(M,N)$ is its canonical mapping cone by \ref{ses}(2). Next let $i\colon R\to I$ be the minimal injective resolution of $R$. Then $(-)^\ast=\cHom_R(-,I)$ and the natural isomorphism $1\simeq\RHom_R(\RHom_\L(-,\L),R)$ is presented by
	\[ \b_M\colon M\to \cHom_R(\cHom_R(M,R),I)=\cHom_R(\cHom_\L(M,\cHom_R(\L,R)),I)\to \cHom_R(\cHom_\L(M,\L),I),  \]
	where the first map is the composite of the evaluation map and $i\colon R\to I$, the second one adjunction, and the third one is induced from a fixed isomorphism $\L\simeq\Hom_R(\L,R)$.
	Then the above diagrams become
	\[ \xymatrix@R=1mm{
		\cHom_R(\cHom_\L(M,N),I)\ar[r]^-{\cdot\a_{M,N}}& \cHom_R(N\otimes_\L\cHom_\L(M,\L),I)\\ \\
		\cHom_R(\cHom_\L(M,\cHom_R(\cHom_\L(N,\L),I)),I)\ar[uu]^\lsimeq_{\b_N}&\\&\cHom_\L(N,\cHom_R(\cHom_\L(M,\L),I))\ar@{=}[uuu]\\
		\cHom_R(\cHom_R(M\otimes_\L\cHom_\L(N,\L),I),I)\ar@{=}[uu]&\\ \\
		M\otimes_\L\cHom_\L(N,\L)\ar[uu]^\lsimeq\ar[r]^-{\a_{N,M}}&\cHom_\L(N,M)\ar[uuu]_\rsimeq^{\b_M}, } \]
	where each $\simeq$ is a natural quasi-isomorphism and each equality is given by adjunction. Then one can verify the desired commutativity.
\end{proof}

Our proof of \ref{eAR} is based on \ref{star} and the sequence in \ref{ses}(2).

\begin{proof}[Proof of \ref{eAR}]
	Consider the exact sequence of $R$-linear dg categories
	\[ \xymatrix{ \per_\dg\!\La\ar[r]&\D^b_\dg(\mod\La)\ar[r]&\C } \]
	which the dg singuarity category $\C=\sg_\dg\!\La$ fits into. By \ref{ses}(2) we have a trianlge
	\[ \xymatrix{ N\lotimes_\La\RHom_\La(M,\La)\ar[r]^-{\al_{M,N}}&\RHom_\La(M,N)\ar[r]&\C(M,N) } \]
	in $\D(\Mod R)$ functorial for all $M,N\in\D^b(\mod \La)$. We apply $(-)^\ast=\RHom_R(-,R)$ to the above sequence. By \ref{star} there is a diagram of triangles
	\[ \xymatrix{
		\RHom_\L(M,N)^\ast\ar[r]^-{\cdot\a_{M,N}}&(N\lotimes_\L\RHom_\L(M,\L))^\ast\ar[r]&\C(M,N)^\ast[1] \\
		M\lotimes_\L\RHom_\L(N,\L)\ar[r]^-{\a_{N,M}}\ar[u]^\lsimeq&\RHom_\L(N,M)\ar[u]_\rsimeq\ar[r]&\C(N,M)} \]
	in which the left square is commutative.
	It follows that we have a quasi-isomorphism $\C(M,N)^\ast[1]\simeq\C(N,M)$ which is functorial over $M,N\in\D^b_\dg(\mod\La)$, in other words, an isomorphism $\C^\ast[1]\simeq\C$ in $\D(\B^\op\otimes_R\B)$ for $\B=\D^b_\dg(\mod\L)$. Now the morphism $\B\to\C$ is a localization of dg categories, thus so is $\B^e\to\C^e$ (\cite[3.10(a)]{Ke11}), hence the isomorphism of $\B^e$-modules implies the isomorphism of $\C^e$-modules, which completes the proof.	
	%
\end{proof}

Let us note that the {\it object-wise} version of \ref{eAR} can be shown by using the dg singularity category $\C_a=\C_\ac(\proj\La)_\dg$ (see \ref{IGdg}) as in \ref{objectwise} below.
Note however that it is {\it not} enough to conclude the desired isomorphism in \ref{eAR} since we cannot see functoriality from the argument below.

\begin{Rem}\label{objectwise}
	For each $X,Y\in\C_a$ there exists an isomorphism in $\D(\Mod R)$:
	\[ \C_a(X,Y)\simeq\RHom_R(\C_a(Y,X),R)[1]. \]
	\begin{proof}
	For each complex $X=(\cdots\to X^{-1}\to X^0\to X^1\to\cdots)$ of $\La$-modules, we denote by $X^{>0}$ and $X^{\le0}$ the following truncations:
	\[ \xymatrix@R=2mm{
		X^{>0}\colon&& \cdots\ar[r]&0\ar[r]&0\ar[r]&X^1\ar[r]&X^2\ar[r]&\cdots,\\
		X^{\leq0}\colon&& \cdots\ar[r]&X^{-1}\ar[r]&X^0\ar[r]&0\ar[r]&0\ar[r]&\cdots.} \]
	Then we have an exact sequence
	\begin{equation}\label{X}
		0\to X^{>0}\to X\to X^{\le0}\to0
	\end{equation}
	of complexes $\La$-modules.
	Applying $\cHom_\La(X,-)$ to the exact sequence $0\to Y^{>0}\to Y\to Y^{\le0}\to0$, we have an exact sequence
	\[0\to\cHom_\La(X,Y^{>0})\to\cHom_\La(X,Y)\to\cHom_\La(X,Y^{\le0})\to0\]
	of dg $R$-modules.
	
	Since $X$ is acyclic and $Y$ is bounded below,  $\cHom_\La(X,Y^{>0})$ is acyclic. Thus we have a quasi-isomorphism
	\[ \cHom_\La(X,Y)\to\cHom_\La(X,Y^{\le0}).\]
	
	Applying $(-)^*$ to the exact sequence $0\to X^{>0}\to X\to X^{\le0}\to0$, we have an exact sequence $0\to (X^{\le0})^*\to X^*\to(X^{>0})^*\to0$. Applying $Y^{\le0}\otimes_\Lambda-$ , we have an exact sequence $0\to Y^{\le0}\otimes_\Lambda(X^{\le0})^*\to Y^{\le0}\otimes_\Lambda X^*\to Y^{\le0}\otimes_\Lambda(X^{>0})^*\to0$.
	Since $Y^{\le0}\otimes_\Lambda X^*$ is acyclic, we have a quasi-isomorphism
	\[\cHom_\La(X^{>0}[1],Y^{\le0})=Y^{\le0}\otimes_\Lambda(X^{>0})^*[-1]\to Y^{\le0}\otimes_\Lambda(X^{\le0})^*.\]
	Applying $\cHom_\La(-,Y^{\le0})$ to \eqref{X}, we obtain a triangle
	\[\cHom_\La(X^{>0}[1],Y^{\le0})\to \cHom_\La(X^{\le0},Y^{\le0})\to\cHom_\La(X,Y^{\le0})\to\]
	in the homotopy category. Using the quasi-isomorphisms above, we have a traingle
	\[Y^{\le0}\otimes_\Lambda(X^{\le0})^*\to \cHom_\La(X^{\le0},Y^{\le0})\to\cHom_\La(X,Y)\to.\qedhere\]
	\end{proof}
\end{Rem}
%

Using \ref{eAR}, we obtain the enhanced CY property \ref{iso}(1) below of the full dg subcategory
\[\C^{\fl}:=\{X\in\C\mid H^0\C(X,X)\in\fl_0R\}\]
of $\C$ which enhances $\sg_0\L$ given in \eqref{ungraded Dsg0}.
In particular, when $\L$ has only an isolated singularity, we obtain the enhanced CY property \ref{iso}(2) below of the dg category $\C$.
We denote by ${(-)^\vee}=\Hom_R(-,I^d)$ the Matlis dual on $\fl_0\!R$, where $I^d$ is the last term of the minimal injective resolution of $R$.
\begin{Thm}\label{iso}
Let $R$ be a commutative Gorenstein ring of dimension $d$, and $\L$ a symmetric $R$-order.
\begin{enumerate}
\item $\C^{\fl}\simeq \C^{\fl}{}^{\vee}[-d+1]$ in $\D((\C^{\fl})^\op\otimes_R\C^{\fl})$.
\item If $\L$ satisfies $(R_{d-1})$-condition, then $\C=\C^{\fl}$ and $\C\simeq \C^{\vee}[-d+1]$ in $\D(\C^\op\otimes_R\C)$.
\end{enumerate}
\end{Thm}

\begin{proof}
(1) We know that $H^0\C^{\fl}$ is $\Hom$-finite, that is, each cohomology of $\C$ is of finite length. Then we have $\RHom_R(\C^{\fl},R)\simeq \C^{\fl}{}^{\vee}[-d]$, and thus $\C^{\fl}\simeq \C^{\fl}{}^{\vee}[-d+1]$.

(2) Since $\L$ is an $R$-order satisfying $(R_{d-1})$, we know that $\C=\C^{\fl}$.
\end{proof}

Taking the cohomology, we recover the classical Auslander-Reiten duality on $\sg\La$, which immediately implies the existence of almost split sequences.
\begin{Cor}[{\cite{Au78}, see also \cite[3.10]{Yo90}}]\label{classical AR}
Let $R$ be a commutative Gorenstein ring of dimension $d$, and $\L$ a symmetric $R$-order.
\begin{enumerate}
\item The triangulated category $\sCM_0\!\L\simeq\sg_0\La$ is $(d-1)$-CY.
\item If $\L$ satisfies $(R_{d-1})$-condition, then $\sCM \L\simeq\sg\La$ is $(d-1)$-CY.
\end{enumerate}
\end{Cor}

Next we will show that, without any assumptions on the singular locus of $\L$, the singularity category $\sCM\L\simeq\sg\La$ enjoys a certain version of $(d-1)$-Calabi-Yau property.
It follows that in $\sg\La$ a $\Hom$-space and its dual are related by a spectral sequence, which can be viewed as the Auslander-Reiten duality for arbitrary singular locus.
\begin{Cor}[{\cite[proof of 3.10]{Yo90}}]\label{spec}
Let $M, N\in\CM\L$. There exists a spectral sequence
\[ E_2^{p,q}=\Ext_R^p(\sHom_\L(M,N[-q]),R)\Rightarrow \sHom_\L(N,M[p+q-1]). \]
\end{Cor}
\begin{proof} Straightforward construction of a spectral sequence (see e.g. \cite[III.7]{GM}), using the injective resolution of $R$.\end{proof}

Note that the spectral sequence in \ref{spec} collapses at $E_k$ if the $\L$ satisfies (R$_{d-k}$)-condition. In particular for $k=2$, we recover a version of Auslander-Reiten duality due to the second author and Wemyss.

\begin{Cor}[{\cite[3.7]{IW14}}]\label{sing1}
	Let $R$ be a commutative Gorenstein ring of dimension $d$, and $\L$ a symmetric $R$-order satisfying (R$_{d-2}$). Then there is an exact sequence
	\[ \xymatrix{0\ar[r]&\Ext_R^d(\sHom_\L(M,N),R)\ar[r]&\Ext_\L^{d-1}(N,M)\ar[r]&\Ext_R^{d-1}(\sHom_\L(M,N[-1]),R)\ar[r]&0} \]
	with $\Ext_R^d(\sHom_\L(M,N),R)\in\fl R$ and $\fl\Ext_R^{d-1}(\sHom_\L(M,N[-1]),R)=0$. Consequently we have isomorphisms
	\begin{equation*}
		\begin{aligned}
			\fl\Ext_\L^{d-1}(N,M)&\simeq\Ext_R^d(\fl\sHom_\L(M,N),R), \\
			\frac{\Ext_R^d(\sHom_\L(M,N),R)}{\fl\Ext_R^d(\sHom_\L(M,N),R)}&\simeq\Ext_R^{d-1}\left(\frac{\sHom_\L(M,N[-1])}{\fl\sHom_\L(M,N[-1])},R\right).
		\end{aligned}
	\end{equation*}
\end{Cor}
\begin{proof}
	The existence of the exact sequence follows from \ref{spec}. Since $R$ is $d$-dimensional Gorenstein $\Ext_R^d(X,R)$ vanishes at non-maximal prime ideals for all $X\in\md R$, thus the first term is in $\fl R$. We now show that the last term has no finite length submodule. For this we prove $\fl\Ext_R^{d-1}(X,R)=0$ for all $X\in\md R$ of with $\dim X\leq1$. Let $0\to R\to I^0\to\cdots\to I^{d-1}\to I^d\to 0$ be the minimal injective resolution of $R$. Since $X$ has dimension $\leq1$ we have an injection $\Ext_R^{d-1}(X,R)\hookrightarrow\Hom_R(X,I^{d-1})$. It follows that all the associated prime ideals of $\Ext_R^{d-1}(X,R)$ has height $d-1$, hence $\fl\Ext^{d-1}_R(X,R)=0$.
\end{proof}

We refer to \cite{OY} for a related work on the Auslander-Reiten duality.

\subsection{Graded Cohen-Macaulay representations}\label{grCM}
We shall need the graded version of \ref{eAR}. Let us collect some preliminaries on Cohen-Macaulay representation theory in the graded setting. To make our results as general as possible, we need to introduce the graded version of the classical notions in commutative ring theory.

Let $G$ be an abelian group, $R$ a commutative Noetherian $G$-graded ring, and $\Lambda$ a $G$-graded $R$-algebra such that the structure morphism $R\to\Lambda$ preserves the $G$-grading. We refer to Appendix \ref{G-graded rings} for a background on group graded rings. In particular, we have the set $\Spec^G\!R$ of $G$-prime ideals of $R$ (\ref{GSpec}), and the notions of {\it $G$-height $\height^G\!\p$} of a $G$-prime ideal $\p$, and {\it $G$-dimension} $\dim^G\!R$ (\ref{Gdim}) of a graded ring $R$.
We assume that $R$ has graded dimension $\dim^G\!R=d$ and define the category of graded Cohen-Macaulay $\L$-modules by
\begin{eqnarray*}
	\CM^G\!\Lambda:=\{X\in\md^G\!\Lambda\mid X\in\CM\Lambda\ \mbox{ as a $\Lambda$-module}\}.
\end{eqnarray*}
To state Auslander-Reiten duality we also consider the following full subcategory of $\CM^G\!\L$:
\begin{eqnarray*}
	\CM_0^G\!\Lambda:=\{X\in\CM^G\!\Lambda\mid \mbox{$X_{\p,G}\in\proj\L_{\p,G}$ for each $\p\in\Spec^G\!R$ with $\height^G\!\p<d$ }\}
\end{eqnarray*}
Consider the category
\[\fl_0^G\!R:=\{ X\in\mod^G\!R\mid \text{any }\p\in\Supp_R^G\!X \text{ satisfies }\height^G\!\p=d\}\]
consisting of all finite length objects in $\Mod^G\!R$ whose composition factor has the form $R/\m$ with $\height^G\!\m=d$.
Note that $\p\in\Spec R$ in the support of $M\in\fl^G_0R$ is not necessarily $G$-graded: For example, let $R=k[x]/(x^2-1)$ with $G=\Z/2\Z$ and $\deg x=1$. Then $\ass R=\{ (x-1),(x+1) \} $, and these are not graded.

Let $\L$ be a $G$-graded symmetric $R$-order. The \emph{$G$-graded singularity category} of $\L$ is the Verdier quotient
\[\sg^G\!\La:=\D^b(\mod^G\!\L)/\per^G\!\L.\]
Then there is a triangle equivalence $\sCM^G\L\simeq\sg^G\!\L$ \cite{Bu}. Consider also thick subcategory
\begin{equation}\label{Dsg0}
	\sg_0^G\!\La:=\thick\{X\in\mod^G\!\L\mid X\in\fl_0^G\!R\}\subset\sg^G\!\La.
\end{equation}
On the other hand, by the first equality of \ref{B16 2} we have
\begin{equation*}
	\sCM_0^G\!\La=\{X\in\sCM^G\!\Lambda\mid \sEnd_\La(X)\in\fl_0^G\!R\}.
\end{equation*}
Then by \ref{B16}(2) we have a triangle equivalence
\begin{equation*}
	\sCM_0^G\!\L\simeq\sg_0^G\!\La.
\end{equation*}
%
%

As in \ref{Sing}, let us consider the graded version of the singular locus.
\begin{Def}\label{Sing^G}
Let $\La$ be a $G$-graded module-finite $R$-algebra. The {\it $G$-singular locus} of $\La$ is
\[ \Sing^G_R\!\La:=\{ \p\in\Spec^G\!R\mid \Mod^G\!\La_{\p,G} \text{ has infinite global dimension} \}. \]
Also, Serre's {\it {\rm(}R$^G_{n}${\rm)}-condition} on $\La$ is the following:
\[ \Sing^G_R\!\La\subset\{ \p\in\Spec^G\!R\mid \height^G\!\p>n\}. \]
\end{Def}
Then $\La$ satisfies {\rm(R$_{d-1}^G$)} if and only if $\sg^G_0\!\La=\sg^G\!\La$ hold, see \ref{B16}. Note again that our terminology of isolated singularities is slightly stronger than the usual one $\Sing_R^G\!\La\subset\Max^G\!R$.
We refer to \ref{sing} for the relationship between graded and ungraded singular loci.


\medskip
We next discuss the notion of Gorenstein parameter for module-finite algebras.
Let $R$ be a Gorenstein ring with $\dim^G\!R=d$ and recall from \ref{GP} the definition of a Gorenstein parameter. When $R$ has Gorenstein parameter $p_R\in G$ we have the following, where $(-)^{\vee_G}=\Hom_R(-,\bigoplus_{\height^G\!\m=d}E_R^G(R/\m))$ is the $G$-graded Matlis dual.
\begin{Prop}
	We have an isomorphism of functor on $\D_{\fl^G_0\!R}(\Mod^G\!R)$:
	\[(-)^{\vee_G}\circ\RHom_R(-,R)\simeq (-p_R)[d].\]
\end{Prop}

We define Gorenstein parameter for module-finite $R$-algebras by the above formula.
\begin{Def}
	We say that $\Lambda$ has a \emph{Gorenstein parameter} $p_\L\in G$ if there exists an isomorphism of functors on $\D_{\fl^G_0\!R}(\Mod^G\!\L)$:
	\[ (-)^{\vee_G}\circ\RHom_\Lambda(-,\Lambda)\simeq (-p_\L)[d]. \]
\end{Def}

{For $G$-graded symmetric orders, we introduce}
the following notion.

\begin{Def}\label{relative Gorenstein parameter}
	Let $\L$ be a $G$-graded symmetric order over $R$. We say $\L$ has {\it relative Gorenstein parameter $p_{\L/R}\in G$} if there is an isomorphism $\Hom_R(\L,R)\simeq\L(-p_{\L/R})$ in $\Mod\L^\op\otimes_R\L$.
\end{Def}

Let us note that a relative Gorenstein parameter exists under a mild assumption.
\begin{Lem}\label{relative}
	If $\L$ is ring-indecomposable and $\mod^G(\L^\op\otimes_R\L)$ is Krull-Schmidt, then relative a Gorenstein parameter exists.
\end{Lem}
\begin{proof}
	Note that ring-indecomposability of $\L$ is nothing but indecomposability of $\L$ in $\Mod\L^\op\otimes_R\L$. The assertion then follows from the following fact: If $\Ga$ is a graded ring such that the category $\mod^G\!\Ga$ of finitely presented graded modules is Krull-Schmidt, then for two indecomposable objects $X,Y\in\mod^G\!\Ga$ which are isomorphic in $\mod\Ga$, there is $a\in G$ such that $X\simeq Y(a)$ in $\mod^G\!\Ga$.
\end{proof}

We have the following relationship between these Gorenstein parameters.
\begin{Prop}\label{sum}
	Suppose two of $p_\L$, $p_R$ and $p_{\L/R}$ exist. Then the other one exists, which can be taken to satisfy $p_\L=p_R+p_{\L/R}$.
\end{Prop}
\begin{proof}
	This is a consequence of a canonical isomorphism $\RHom_R(\RHom_\L(-,\L),R)=-\lotimes_\L\RHom_R(\L,R)$ on $\D^b(\mod^G\!\L)$.
\end{proof}
In many cases our $R$ will be $G$-local with each $R_g$ being finite dimensional. In this case, all $p_R$, $p_{\L/R}$, and $p_\L$ exist.

\subsection{Graded version of enhanced and classical Auslander-Reiten duality}\label{grenh}
We next consider dg enhancements of graded singularity categories. We start from the following definition. 
\begin{Def}\label{grdgsg}
Let $G$ be an abelian group, and $A$ a $G$-graded Iwanaga-Gorenstein ring.
\begin{enumerate}
\item We define the canonical dg enhancement of the $G$-graded singularity category as the dg quotient
\[ \sg_\dg^G\!A:=\D^b_\dg(\mod^G\!A)/\per_\dg^G\!A \]
of the canonical dg enhancements of $\D^b(\mod^G\!A)$ by that of $\per^G\!A$.
\item We define the {$G$-graded dg} category $\C$ with
\begin{itemize}
	\item objects: same as $\sg_\dg^G\!A$,
	\item morphisms: $\C(L,M)=\bigoplus_{g\in G}(\sg_\dg^G\!A)(L,M(g))$.
\end{itemize}
We call this $\C$ the {\it $G$-graded dg singularity category}.
\end{enumerate}
\end{Def}


Thus $\C$ is naturally a $G$-graded dg category, and its degree $0$ part is $\C_0=\sg_\dg^G\!A$, the canonical dg enhancement of the graded singularity category $\sg^G\!A$. Also $\C$, as an ungraded dg category, is equivalent to the full dg subcategory of the (ungraded) dg singularity category $\sg_\dg^G\!A$ formed by gradable objects. Thus $\per\C$ is equivalent to the thick subcategory of $\sg A$ generated by gradable objects.

This discussion certainly lifts to the dg level and is summarized in the following commutative diagram in $\Hmo$.
\begin{align*}
\xymatrix@R=3mm{
	\per^G_\dg\!\C=\C_0\ar@{-}[r]^-\simeq\ar[d]&\sg^G_\dg\!A\ar[d]\\
	\C\ar@{-}[r]^-\simeq&\thick_\dg(\sg_\dg^G\!A)\ar@{^(->}[r]&\sg_\dg A}
\end{align*}
Here, $\thick_\dg(\sg_\dg^G\!A)$ is the smallest full dg subcategory of $\sg_\dg\!A$ which is closed under mapping cones, $[\pm1]$ and direct summands, and contains the image of the forgetful functor $\sg_\dg^G\!A\to\sg_\dg A$.
More generally, for each subgroup $H\subset G$, we have the following commutative diagram in $\Hmo$, which will be used later.
\begin{align}\label{sg_dg and sg}
\xymatrix@R=3mm{
	\per^G_\dg\!\C=\C_0\ar@{-}[r]^-\simeq\ar[d]&\sg^G_\dg\!A\ar[d]\\
	\per^{G/H}_\dg\!\C\ar@{-}[r]^-\simeq&\thick_\dg(\sg_\dg^G\!A)\ar@{^(->}[r]&\sg_\dg^{G/H}\!A}
\end{align}

Now we return to the following setup of module-finite algebras.
\begin{itemize}
\item $R=\bigoplus_{g\in G}R_g$ is a $G$-graded commutative Gorenstein ring with $\dim^G\!R<\infty$.
\item $\L=\bigoplus_{g\in G}\L_g$ is $G$-graded symmetric $R$-order with relative Gorenstein parameter $p_{\L/R}$, see \ref{relative Gorenstein parameter}.
\end{itemize}

We have the following graded version of \ref{eAR}.

\begin{Thm}\label{greAR}
Let $\C$ be the $G$-graded dg singularity category of $\L$, see {\rm \ref{grdgsg}(2)}. Then we have an isomorphism
\[ \C^\ast[1] \simeq \C(-p_{\L/R})\ \mbox{ in }\ \D^G(\C^\op\otimes_R\C).\]
\end{Thm}

\begin{proof}
Taking the isomorphism $\RHom_\L(-,\L)\simeq\RHom_R(-,R)(p_{\L/R})$ into account, the maps appearing in the proof of \ref{eAR} becomes
\[ \xymatrix@R=1mm{
	\RHom_\L(M,N)^\ast\ar[r]&(N\lotimes_\L\RHom_\L(M,\L))^\ast\ar[r]& \C(M,N)^\ast[1]\ar[r]&\\ \\
	\RHom_\L(M,\RHom_\L(N,\L)^\ast(p_{\L/R}))^\ast\ar[uu]^\lsimeq&&\\&\RHom_\L(N,\RHom_\L(M,\L)^\ast)\ar@{=}[uuu]&\\
	(M\lotimes_\L\RHom_\L(N,\L))^{\ast\ast}(-p_{\L/R})\ar@{=}[uu]&&\\ \\
	M\lotimes_\L\RHom_\L(N,\L)(-p_{\L/R})\ar[r]\ar[uu]^\lsimeq&\RHom_\L(N,M(-p_{\L/R}))\ar[r]\ar[uuu]_\rsimeq& \C(N,M)(-p_{\L/R})\ar[r]&} \]
when the objects are graded. Thus the assertion follows.
\end{proof}

Next we consider the version of \ref{iso}. Let $d=\dim^G\!R$ and suppose $R$ has Gorenstein parameter $p_R\in G$ (\ref{GP}). Note that a Gorenstein parameter exists as soon as $R_0$ is local (\ref{G-local}, \ref{resolution}).
We denote by $\C$ the $G$-graded dg singularity category of $\L$, and consider the full dg subcategory 
\begin{equation}\label{define C'}
\C^{\fl}:=\{X\in \C\mid H^0\C(X,X)\in\fl_0^G\!R\}.
\end{equation}

Let ${(-)^{\vee_G}}=\Hom_R(-,\bigoplus_{\m\in\Max^G\!R}E_R^G(R/\m))$ be the Matlis dual.

\begin{Thm}\label{griso}
Let $R$ be a $G$-graded commutative Gorenstein ring of $\dim^G\!R=d<\infty$ with Gorenstein parameter $p_R$, and $\L$ a $G$-graded symmetric $R$-order with Gorenstein parameter $p_\L$.
\begin{enumerate}
\item There exists an isomorphism
\[ {\C^{\fl}{}^{\vee_G}}\simeq\C^{\fl}(-p_\L)[d-1]\ \mbox{ in }\ \D^G((\C^{\fl})^\op\otimes_R\C^{\fl}).\]
\item If $\L$ satisfied {\rm(R$_{d-1}^G$)} condition, then there exists an isomorphism
\[ {\C^{\vee_G}}\simeq\C(-p_\L)[d-1]\ \mbox{ in }\ \D^G(\C^\op\otimes_R\C).\]
\end{enumerate}
Moreover, if $R_0$ is a finite dimensional algebra over a field $k$, then the Matlis dual $(-)^{\vee_G}$ in the isomorphisms above can be replaced by the graded $k$-dual $D$ sending $M=\bigoplus_{g\in G}M_g$ to $DM=\bigoplus_{g\in G}\Hom_k(M_{-g},k)$.
\end{Thm}

\begin{proof}
(1)  By \ref{sum} the symmetric order $\L$ has Gorenstein parameter $p_\L=p_R+p_{\L/R}$. By assumption, each cohomology of $\C^{\fl}$ lies in $\fl_0^G\!R$, so we have $\C^\ast[d]\simeq\C^{\vee_G}(p_R)$. Combining with \ref{greAR}, we deduce $\C^{\vee_G}=\C^\ast(-p_R)[d]=\C(-p_R-p_{\L/R})[d-1]=\C(-p_\L)[d-1]$.\\
(2)  If $\La$ satisfies (R$^G_{d-1}$) condition, then $\C^{\fl}=\C$.\\
The last statement follows from \ref{kdual}.
\end{proof}



As in the ungraded case, taking the $0$-th cohomology gives the classical graded Auslander-Reiten duality. As in the ungraded case, this implies the existence of almost split sequences.
\begin{Cor}\label{grAR}
For each $M,N\in\sCM^G_0\!\L$ we have a natural isomorphism
\[ D\sHom^G_\La(M,N)\simeq\sHom_\La^G(N,M(-p_\L)[d-1]), \]
that is, we have the following.
\begin{enumerate}
\item The triangulated category $\sCM^G_0\!\La$ has a Serre functor $(-p_\L)[d-1]$.
\item If $\Lambda$ satisfies {\rm(R$_{d-1}^G$)} condition, then $\sCM^G_0\!\La=\sCM^G\!\L$ has a Serre functor $(-p_\L)[d-1]$.
\end{enumerate}
\end{Cor}
We leave the graded analogues of \ref{spec} and \ref{sing1} to the reader.

\section{Cluster categories and singularity categories}\label{sec4}

\subsection{The equivalence}
%
%

Applying the results from the preceeding sections to the singularity categories of symmetric orders, we deduce that the equivalence of a graded singularity category and a derived category automatically implies the equivalence of ungraded singularity category and the cluster category. 

\begin{Setup}\label{setup main}
Our setting is the following.
\begin{enumerate}
\renewcommand\labelenumi{(\Roman{enumi})}
\renewcommand\theenumi{\Roman{enumi}}
\item $G$ is an abelian group and $R$ is a $G$-graded Gorenstein $k$-algebra over a field $k$ with $\dim_kR_0<\infty$.
\item $R$ is $G$-local with $G$-maximal ideal $\m$ such that $\m$ is maximal (as an ungraded ideal). Let $d:=\dim^G\!R$.
\item $\L$ is a $G$-graded $R$-algebra such that the structure morphism $R\to\La$ preserves the $G$-gradings. 
\item $\La$ is a symmetric $R$-order with Gorenstein parameter $p$. 
\end{enumerate}
\end{Setup}
Note that by \ref{dim's}(1) the assumption (II) shows that $d=\height^G\!\m=\height\m$.

We first state our main result for a special case when $\La$ has a small singular locus. We refer to \ref{sing} for basic properties of graded singular loci, in particular, $\Sing^G_R\!\La\subset\Sing^{G/H}_R\!\La$ holds for each torsion-free subgroup $H$ of $G$.
%
\begin{Thm}\label{CM2}
Under the setting \ref{setup main}, assume $\Sing^{G/(p)}_R\!\La\subset\{\m\}$ and that $p\in G$ is torsion-free. For each full dg subcategory $\A\subset\sg_{\dg}^G\La$ which generates $\sg^G\!\La$ as a thick subcategory, the following assertions hold.
\begin{enumerate}
\item The dg category $\A$ is component-wise proper and Gorenstein. 
\item 
There exists a commutative diagram of dg categories on the left below, whose vertical maps are isomorphisms in $\Hmo$. Therefore we have a commutative diagram of triangulated categories on the right below.
	\[
	\xymatrix@R=3mm{
		\A\ar@{-}[d]^-\rsimeq\ar[r]&\bG_{d-1}(\A)\ar@{-}[d]^-\rsimeq\\
		\sg_{\dg}^G\La\ar[r]&\sg_{\dg}^{G/(p)}\!\La}
	\qquad
	\xymatrix@R=3mm{
		\per\A\ar@{-}[d]^-\rsimeq\ar[r]&\C_{d-1}(\A)\ar@{-}[d]^-\rsimeq\\
		\sg^G\!\La\ar[r]&	\sg^{G/(p)}\!\La}
	\]
\item Assume $G=\Z$ and $\Sing_R\La\subset\{\m\}$. 
Then the object $V\in\D(\A^e)$ given by $V(A,B):=\sg_{\dg}^G(A,B(-1))$ for $A,B\in\A$ satisfies
$V^{\lotimes_\A p}\simeq D\A[1-d]$ in $\D(\A^e)$. Defining $\bG_{d-1}^{(1/p)}(\A)$ as $\A/V$,
the above commutative diagrams extend to the third columns below.
	\[
	\xymatrix@R=3mm{
		\A\ar[r]\ar@{-}[d]^-\rsimeq&\bG_{d-1}(\A)\ar[r]\ar@{-}[d]^-\rsimeq&\bG_{d-1}^{(1/p)}(\A)\ar@{-}[d]^-\rsimeq\\
		\sg_{\dg}^\Z\!\La\ar[r]&\sg_{\dg}^{\Z/p\Z}\!\La\ar[r]&\sg_{\dg}\!\La}
	\qquad
	\xymatrix@R=3mm{
		\per\A\ar[r]\ar@{-}[d]^-\rsimeq&\C_{d-1}(\A)\ar[r]\ar@{-}[d]^-\rsimeq&\C^{(1/p)}_{d-1}(\A)\ar@{-}[d]^-\rsimeq\\
		\sg^\Z\!\La\ar[r]&\sg^{\Z/p\Z}\!\La\ar[r]&\sg\La}
	\]
\end{enumerate}
\end{Thm}

There are two main differences in our main result for general case. First, we need to consider certain full subcategories $\sg_0^*\!\La$ of the singularity categories $\sg^*\!\La$, where $*$ is a group. Secondly, we need to consider certain localizations $\La_{\m,*}$  of $\La$. To state it explicitly, we need to prepare some notations.

Let $R$ be a commutative Noetherian $G$-graded ring $R$, and $\La$ a $G$-graded $R$-algebra  such that the structure morphism $R\to\La$ preserves the $G$-grading.
We denote by $\sg_{0,\dg}^G\!\La$ the full dg subcategory of $\sg_{\dg}^G\!\La$ corresponding to $\sg_0^G\!\La$ from (\ref{Dsg0}). 
For $\p\in\Spec^G\!R$, we denote by $R_{\p,G}$ the localization of $R$ at the set of $G$-homogeneous elements not belonging to $\p$, and 
let $\La_{\p,G}:=\La\otimes_RR_{\p,G}$.

\begin{Thm}\label{CM}
Under the setting \ref{setup main}, assume that $p\in G$ is torsion-free. For each full dg subcategory $\A\subset\sg_{0,\dg}^G\La$ which generates $\sg_0^G\!\La$ as a thick subcategory, the following assertions hold.
\begin{enumerate}
	\item The dg category $\A$ is component-wise proper and Gorenstein. 
	\item There exists a commutative diagram of dg categories on the left below, whose vertical maps are isomorphisms in $\Hmo$. Therefore we have a commutative diagram of triangulated categories on the right below.
	\[
	\xymatrix@R=3mm{
		\A\ar@{-}[d]^-\rsimeq\ar[r]&\bG_{d-1}(\A)\ar@{-}[d]^-\rsimeq\\
		\sg_{0,\dg}^G\La\ar[r]&\sg_{0,\dg}^{G/(p)}\!\La_{\m,G/(p)} }
	\qquad
	\xymatrix@R=3mm{
		\per\A\ar@{-}[d]^-\rsimeq\ar[r]&\C_{d-1}(\A)\ar@{-}[d]^-\rsimeq\\
		\sg_0^G\!\La\ar[r]&	\sg_0^{G/(p)}\!\La_{\m,G/(p)} }
	\]
	\item Assume $G=\Z$. Then the object $V\in\D(\A^e)$ given by $V(A,B):=\sg_{0,\dg}^G(A,B(-1))$ for $A,B\in\A$ satisfies
$V^{\lotimes_\A p}\simeq D\A[1-d]$ in $\D(\A^e)$. Defining $\bG_{d-1}^{(1/p)}(\A)$ as $\A/V$,
the above commutative diagrams extend to the third columns below.
	\[
	\xymatrix@R=3mm{
		\A\ar[r]\ar@{-}[d]^-\rsimeq&\bG_{d-1}(\A)\ar[r]\ar@{-}[d]^-\rsimeq&\bG_{d-1}^{(1/p)}(\A)\ar@{-}[d]^-\rsimeq\\
		\sg_{0,\dg}^\Z\!\La\ar[r]&\sg_{0,\dg}^{\Z/p\Z}\!\La_{\m,\Z/p\Z}\ar[r]&\sg_{0,\dg}\!\La_\m}
	\qquad
	\xymatrix@R=3mm{
		\per\A\ar[r]\ar@{-}[d]^-\rsimeq&\C_{d-1}(\A)\ar[r]\ar@{-}[d]^-\rsimeq&\C^{(1/p)}_{d-1}(\A)\ar@{-}[d]^-\rsimeq\\
		\sg_0^\Z\!\La\ar[r]&\sg_0^{\Z/p\Z}\!\La_{\m,\Z/p\Z}\ar[r]&\sg_0\!\La_\m}
	\]
\end{enumerate}
\end{Thm}

We refer to \ref{locex} which illustrates that we need to consider the localizations like $\La_{\m,G/(p)}$ in \ref{CM} above.

When $\sg_0^G\!\La$ has a tilting subcategory the above theorems can be stated as follows.

\begin{Cor}\label{tilt}
In the setting in \ref{CM}, suppose there is a tilting subcategory $\P\subset\sg_0^G\!\La$.
\begin{enumerate}
\item The $k$-linear category $\P$ is proper and Gorenstein.
\item There exists a commutative diagram
\[ \xymatrix@R=3mm{
	\per\P\ar[r]\ar@{-}[d]^-\rsimeq&\C_{d-1}(\P)\ar@{-}[d]^-\rsimeq\\
	\sg_0^G\!\La\ar[r]&\sg_0^{G/(p)}\!\La_{\m,G/(p)}.} \]
If $\Sing_R^{G/(p)}\!\L\subset\{\m\}$, then we can replace $\sg_0^G\!\L$ and $\sg_0^{G/(p)}\!\L_{\m,G/(p)}$ above by $\sg^G\!\L$ and $\sg^{G/(p)}\!\L$ respectively.
\item Suppose furthermore $G=\Z$.
Then there exists $V\in\D(\P^e)$ which satisfies $V^{\lotimes_\P p}\simeq D\P[-d]$ in $\D(\P^e)$ and yields a commutative diagram below.
\[ \xymatrix@R=3mm{
	\per\P\ar[r]\ar@{-}[d]^-\rsimeq&\C_{d-1}(\P)\ar@{-}[d]^-\rsimeq\ar[r]&\C^{(1/p)}_{d-1}(\P)\ar@{-}[d]^-\rsimeq\\
	\sg_0^\Z\!\La\ar[r]&\sg_0^{\Z/p\Z}\!\La_{\m,\Z/p\Z}\ar[r]&\sg_0\!\La_\m} \]
If $\Sing_R\L\subset\{\m\}$, then we can replace $\sg_0^\Z\!\L$, $\sg_0^{\Z/p\Z}\!\L_{\m,\Z/p\Z}$ and $\sg_0\!\L_\m$ above by $\sg^\Z\!\L$, $\sg^{\Z/p\Z}\!\L$ and $\sg\L$ respectively.
\end{enumerate}
\end{Cor}

As an application, we obtain the following result. We say that a subcategory $\C$ in a triangualted category $\T$ is {\it $n$-rigid} if $\Hom_\T(C,C'[i])=0$ for each $C,C'\in\C$ and $0<i<n$. We say $\C$ is {\it $n$-cluster tilting} if it is functorially finite and satisfies
\begin{align*}
	\C&=\{ T\in \T\mid \Hom_\C(T,C[i])=0 \text{ for all } C\in\C \text{ and } 0<i<n \}\\
	&=\{ T\in \T\mid \Hom_\C(C,T[i])=0 \text{ for all } C\in\C \text{ and } 0<i<n \}.
\end{align*}

\begin{Cor}
In the setting in \ref{tilt}(2), assume that $\P$ is equivalent to $\proj A$ for a finite dimensional $k$-algebra $A$.
\begin{enumerate}
\item If $\id_AA\le d-1$, then $\P$ is a $(d-1)$-rigid subcategory of  $\sg_0^{G/(p)}\!\L_{\m,G/(p)}$.
\item If $\pd_{A^e} A\le d-1$, then $\P$ is a  $(d-1)$-cluster tilting subcategory of $\sg_0^{G/(p)}\!\La_{\m,G/(p)}$.
\end{enumerate}
\end{Cor}

\begin{proof}
Since we have an equivalence $\sg_0^{G/(p)}\!\L_{\m,G/(p)}\simeq\C_{d-1}(A)$, these are basic properties of $(d-1)$-cluster categories \cite{Guo}.
\end{proof}

Now we explain the need of localization in our main theorem \ref{CM}.
\begin{Ex}\label{locex}
	Let $k$ be an algebraically closed field, $R=k[x,y]/(x^2)$ a $\Z$-graded Gorenstein ring with $\deg x=2$ and $\deg y=1$, which is $\Z$-local with $\Z$-maximal ideal $\m=(x,y)$ and has Gorenstein parameter $p=-1$.
	We know by \cite{BIY} there is a triangle equivalence $\sg_0^\Z\!R\simeq\per A$ for a finite dimensional algebra $A=\begin{bmatrix}k&0\\k[z]/(z^2)&k[z]/(z^2)\end{bmatrix}$, see (\ref{BIY equivalence}) below. Then one would expect a commutative diagram
	\[ \xymatrix@R=3mm{
		\per A\ar[r]\ar@{-}[d]^-\rsimeq&\C_{0}(A)\ar@{-}[d]^-\rsimeq\\
		\sg_0^\Z\!R\ar[r]&\sg_0\!R.} \]
	However this is {\it not} true.	We claim that the image of the forgetful functor $\sg^\Z_0\!R\to\sg_0\!R$ does {\it not} even generates $\sg_0\!R$. Indeed, for each $\alpha\in k$ consider the maximal ideal $\m_\alpha:=(x,y-\alpha)$. Then we have an equivalence
	\[\sg_0\!R\simeq\prod_{\alpha\in k}\sg_0\!R_{\m_\alpha}, \quad M\mapsto (M_{\m_\al})_{\al\in k}, \]
	and every $\sg_0\!R_{\m_\al}$ is equivalent to each other. The image of the forgetful functor $\sg^\Z_0\!R\to\sg_0\!R$ is the component for $\al=0$. Therefore the correct diagram is
	\[ \xymatrix@R=3mm{
		\per A\ar[r]\ar@{-}[d]^-\rsimeq&\C_{0}(A)\ar@{-}[d]^-\rsimeq\\
		\sg_0^\Z\!R\ar[r]&\sg_0\!R_\m.} \]
\end{Ex}

In Part 2, we will apply this result to various classes of Gorenstein rings and symmetric orders.
Let us end this subsection by posing a natural problem to study.
\begin{Pb}
Give a description of $\sg\La$ as a cluster-like category when $p$ is a torsion element in $G$.
\end{Pb}

\subsection{Proof of the main results}\label{correction}
Let $\La$ be a ring graded by an abelian group $G$, and let $H\subset G$ be a subgroup. Then we have a forgetful functor
\[ \Mod^G\!\La\to\Mod^{G/H}\!\La. \]
We collect some observations on the image of this forgetful functor at the level of singularity categories, in the setting of module-finite algebras. 
For an abelian group $A$ and $A$-graded ring $\Ga$, let $\Sim^A\!\Ga$ be the set of isomorphism classes of simple objects in $\mod^A\!\Ga$. 

In the rest of this subsection, we assume \ref{setup main}(I)(II)(III).
Let $H\subset G$ be a subgroup. By \ref{setup main}(II) the ideal $\m$ is also a $G/H$-maximal ideal.
Let 
\[
\Sim^{G/H}_{\m}\!\La:=\{S\in\Sim^{G/H}\!\La\mid\Supp_R^{G/H}\!S\subset\{\m\}\}.
\]
We first prove \ref{CM2} for isolated singularity case.

\begin{Prop}\label{isolated case}
In addition to \ref{setup main}{\rm(I)(II)(III)}, assume that $H$ is a torsion-free subgroup of $G$ and $\Sing^{G/H}_R\!\La\subset\{\m\}$. 
Then the image of the forgetful functor $\sg^G\!\La\to\sg^{G/H}\!\La$ generates $\sg^{G/H}\!\La$ as a thick subcategory.
\end{Prop}

\begin{proof}
(i) We prove $\sg^{G/H}\!\La=\thick(\Sim^{G/H}_{\m}\!\La)$ and $\sg^G\!\La=\thick(\Sim^G\!\La)$.
The first assertion is immediate from \ref{B16}(3). We prove the second one. For $X\in\sg^G\!\La$, let $E:=\End_{\sg\La}(X)\in\mod^G\!\La$. By \ref{Spec G G/H}(2) and the last assertion of \ref{B16 2}, we have $\Supp^G_RE\subset\Supp^{G/H}_RE\subset\Sing^{G/H}_R\!\La\subset\{\m\}$. By \ref{B16}(2), we have $X\in\thick(\Sim^G\!\La)$.

(ii) We prove the assertion.
By \ref{setup main}(II), $R$ is a $G$-local ring with $G$-maximal ideal $\m$, so we have $\Sim^G\!\La=\Sim^G\!(\La/\m\La)$. Since $H$ is torsion-free, we have $\Sim^G\!(\La/\m\La)/H\simeq\Sim^{G/H}\!(\La/\m\La)$ by \ref{forget 2}(3). Thus we have a bijection $(\Sim^G\!\La)/H\xsimeq\Sim^{G/H}_{\m}\!\La$.
Thus the image of the composition $\sg^G\!\La=\thick(\Sim^G\!\La)\xrightarrow{{\rm forget}}\sg^{G/H}\!\La=\thick(\Sim^{G/H}_{\m}\!\La)$ generates $\sg^{G/H}\!\La$ as a thick subcategory.
\end{proof}

\begin{proof}[Proof of \ref{CM2}]
Since $\La$ is a $d$-dimensional symmetric order, the graded Enhanced Auslander-Reiten duality \ref{griso} gives an isomorphism $D\C\simeq\C(-p)[d-1]$.

(1) By \ref{root}, the degree $0$ part $\C_0$ is component-wise proper and Gorenstein. Therefore so is $\A$ since it is Morita equivalent to $\C_0$.
	
(2) By \ref{isolated case}, we have an isomorphism $\thick_\dg(\sg_{\dg}^G\!\La)\simeq\sg^{G/(p)}_{\dg}\!\La$ in $\Hmo$. Thus \eqref{sg_dg and sg} gives the commutative square in the right, and Morita Theorem \ref{futatsu}(1) gives the commutative square in the left since $\C$ is $p$-shifted $(d-1)$-Calabi-Yau.
\begin{align*}
\xymatrix@R=3mm{
	\A\ar@{-}[rr]^-\simeq\ar[d]&\ar@{}[d]|(.45){\rm\ref{futatsu}(1)}&\per^G_\dg\!\C\ar@{-}[rr]^-\simeq\ar[d]&\ar@{}[d]|(.45){\eqref{sg_dg and sg}}&\sg_{\dg}^G\!\La\ar[d]\\
	\bG_{d-1}(\A)\ar@{-}[rr]^-\simeq&&\per^{G/(p)}_\dg\!\C\ar@{-}[rr]^-\simeq&&\thick_\dg(\sg_{\dg}^G\!\La)\ar@{-}[r]^-\simeq&\sg_{\dg}^{G/(p)}\!\La.}
\end{align*}
Thus we obtain the desired diagram.

(3) This is similar to (2); this time we use Morita Theorem \ref{futatsu}(2) instead of (1) above. 
\end{proof}

Now we prove our general result \ref{CM}. We need the following variation of \ref{isolated case}, where
\[\sg^{G/H}_{\m}\!\La:=\thick(\Sim^{G/H}_{\m}\!\La)\subset\sg^{G/H}\!\La.\]

\begin{Prop}\label{localization}
In addition to \ref{setup main}{\rm(I)(II)(III)}, assume that $H$ is a torsion-free subgroup of $G$.
\begin{enumerate}
\item\label{simples} The functors $\mod^G\!\La\xrightarrow{{\rm forget}}\mod^{G/H}\!\La\xrightarrow{(-)_{\m,G/H}}\mod^{G/H}\!\La_{\m,G/H}$ induce bijections
\begin{align*}
(\Sim^G\!\La)/H\xsimeq\Sim^{G/H}_{\m}\!\La\xsimeq\Sim^{G/H}\!\La_{\m,G/H}.
\end{align*}
\item\label{on sg} The image of the composition $\sg_0^G\!\La\subset\sg^G\!\La\xrightarrow{{\rm forget}}\sg^{G/H}\!\La$ generates $\sg^{G/H}_{\m}\!\La$ as a thick subcategory.
\item We have an equivalence $-\otimes_RR_{\m,G/H}:\sg^{G/H}_{\m}\!\La\xsimeq\sg^{G/H}_0\!\La_{\m,G/H}$.
\end{enumerate}
Thus we have equivalences $\thick_{\sg^{G/H}\!\La}(\sg_{0}^G\!\La)\simeq\sg_{\m}^{G/H}\!\La\simeq\sg^{G/H}_{0}\!\La_{\m,G/H}$.
\end{Prop}

\begin{proof}
For brevity, we denote $(-)':=-\otimes_RR_{\m,G/H}$.

(1) By \ref{setup main}(II), $R$ is a $G$-local ring with $G$-maximal ideal $\m$, so we have $\Sim^G\!\La=\Sim^G\!(\La/\m\La)$. Since $H$ is torsion-free, we have $\Sim^G\!(\La/\m\La)/H\simeq\Sim^{G/H}\!(\La/\m\La)$ by \ref{forget 2}(3). Thus we have the left bijection
\begin{align*}
(\Sim^G\!\La)/H=\Sim^G\!(\La/\m\La)/H\xsimeq\Sim^{G/H}\!(\La/\m\La)=\Sim^{G/H}_{\m}\!\La.
\end{align*}
Since $R'$ is a $G/H$-local ring with $G/H$-maximal ideal $\m$, we have $\Sim^{G/H}\!\La'=\Sim^{G/H}\!(\La'/\m\La')$. Since any ($G/H$-homogeneous) element outside $\m$ is a unit in $R/\m$, we have $R/\m=R'/\m R'$ and $\La/\m\La=\La'/\m\La'$. Thus we have the right bijection
\[\Sim^{G/H}_{\m}\!\La=\Sim^{G/H}\!(\La/\m\La)\xrightarrow{(-)'}\Sim^{G/H}\!(\La'/\m\La')=\Sim^{G/H}\!\La'.\]

(2) By \ref{setup main}(II), the category $\sg_0^G\!\La$ is generated by $\Sim^G\!\La$. Thus the left bijection in (\ref{simples}) shows that the image generates $\sg^{G/H}_{\m}\!\La$.

(3) First we prove that the functor $\sg_{\m}^{G/H}\!\La\to\sg_0^{G/H}\!\La'$ is fully faithful.
Let $X,Y\in\sg^{G/H}_{\m}\!\La$ and put $M:=\Hom_{\sg\La}(X,Y)$. Then we have an isomorphism $M'\xsimeq\Hom_{\sg\La'}(X',Y')$ of $G/H$-graded $R'$-modules.
Since $\Supp_R^G\!M\subset\Supp_R^{G/H}M\subset\{\m\}$ we have $\Supp_RM\subset\{\m\}$ by \ref{tsujitsuma}.
Then we obtain $M=M'=M_\m$ and hence
\begin{align}\label{La and La'}
M=\Hom_{\sg\La}(X,Y)\xsimeq\Hom_{\sg\La'}(X',Y')=M',
\end{align}
which shows that $(-)'=-\otimes_RR_{\m,G/H}$ is fully faithful. 

It remains to prove that the functor is dense.
By \ref{setup main}(II), $\sg_{\m}^{G/H}\!\La$ and $\sg_0^{G/H}\!\La_{\m,G/H}$ are generated by $\Sim_{\m}^{G/H}\!\La$ and $\Sim^{G/H}\!\La_{\m,G/H}$ respectively.
Thus the assertion follows from the right bijection in (\ref{simples}). 
\end{proof}

	


Although the proof of \ref{CM} is just a generalization of that of \ref{CM2} using \ref{localization}, we include the details for completeness.
We denote by $\C^{\fl}$ the $G$-graded dg category given in \eqref{define C'} and $\C^{\fl}_0$ is its (Adams) degree $0$ part.

\begin{proof}[Proof of \ref{CM}]
Since $\La$ is a $d$-dimensional symmetric order, the graded Enhanced Auslander-Reiten duality \ref{griso} gives an isomorphism $D\C^{\fl}\simeq\C^{\fl}(-p)[d-1]$.

(1)  By \ref{root}, the degree $0$ part $\C^{\fl}_0$ is component-wise proper and Gorenstein. Therefore so is $\A$ since it is Morita equivalent to $\C^{\fl}_0$.
	
(2) As subcategories of \eqref{sg_dg and sg}, we have the commutative diagram
\begin{align}\label{thick_dg}
\xymatrix@R=3mm{
	\per^G_\dg\!\C^{\fl}\ar@{-}[r]^-\simeq\ar[d]&\sg^G_{0,\dg}\!\La\ar[d]\\
	\per^{G/(p)}_\dg\!\C^{\fl}\ar@{-}[r]^-\simeq& \thick_\dg(\sg^G_{0,\dg}\!\La)\ar@{^(->}[r]&\sg^{G/(p)}_{0,\dg}\!\La.}
\end{align}
By \ref{localization}, we have an isomorphism $\thick_\dg(\sg_{0,\dg}^G\!\La)\simeq\sg^{G/(p)}_{0,\dg}\!\La_{\m,G/(p)}$ in $\Hmo$.
On the other hand, since $\C^{\fl}$ is $p$-shifted $(d-1)$-Calabi-Yau, Morita Theorem \ref{futatsu}(1) gives the commutative square in the left.\begin{align*}
\xymatrix@R=3mm{
	\A\ar@{-}[rr]^-\simeq\ar[d]&\ar@{}[d]|(.45){\rm\ref{futatsu}(1)}&\per^G_\dg\!\C^{\fl}\ar@{-}[rr]^-\simeq\ar[d]&\ar@{}[d]|(.45){\eqref{thick_dg}}&\sg_{0,\dg}^G\!\La\ar[d]\\
	\bG_{d-1}(\A)\ar@{-}[rr]^-\simeq&&\per^{G/(p)}_\dg\!\C^{\fl}\ar@{-}[rr]^-\simeq&&\sg_{0,\dg}^{G/(p)}\!\La_{\m,G/(p)}.}
\end{align*}
Thus we obtain the desired diagram.

(3) This is similar to (2); this time we use Morita Theorem \ref{futatsu}(2) instead of (1) above. 
\end{proof}

%
%

\subsection{Hypersurface case}
Let $S$ be a regular ring, $0\neq f\in S$, and $R=S/(f)$ a hypersurface singularity. Then by Eisenbud's matrix factorization theorem \cite{Ei}\cite[Chapter 7]{Yo90}, the singularity category of $R$ is $2$-periodic, which allows us to change the CY dimension flexibly. Let us note a straightforward observation on a derived version of this periodicity, in the graded setting.

Let $S=k[x_0,\ldots,x_d]$ be a polynomial ring, graded by an abelian group $G$ with $\deg x_i=p_i$. Let $0\neq f\in S$ be a homogeneous element of degree $c$, and $R=S/(f)$ the corresponding hypersurface singularity. Then $R$ is a Gorenstein ring with Gorenstein parameter $p=\sum_{i=0}^dp_i-c$.
As in Section \ref{grenh}, we denote by $\C$ the $G$-graded dg singularity category \ref{grdgsg} of $R$, and by $\C^{\fl}$ its dg subcategory given in \eqref{define C'}.

\begin{Prop}\label{mf}
In the above setting, the following assertions hold.
\begin{enumerate}
\item We have $\C[2]=\C(c)$ as dg $(\C^\op\otimes_R\C)$-modules.
\item If $R_0$ is finite dimensional over $k$, we have $D\C^{\fl}\simeq\C^{\fl}(-p-\ell c)[d+2\ell-1]$ in $\D^G((\C^{\fl})^\op\otimes_R\C^{\fl})$. Therefore $\C^{\fl}$ is $(p+lc)$-shifted $(d+2l-1)$-Calabi-Yau. 
\end{enumerate}
\end{Prop}

\begin{proof}
	It is enough to prove (1). By \ref{IGdg} we may take the dg singularity category $\C$ as $\C_\ac(\proj^G\!R)_\dg$, the category of acyclic complexes of projective modules. We know that any CM $R$-module has a $2$-periodic complete resolution given by the corresponding matrix factorization of $f$ (see \cite[Chapter 7]{Yo90}). Taking the grading into account, we see that the degree of two successive differentials has to sum to $\deg f=c$. It follows that we may take a complete resolution $X\in\C_\ac(\proj^G\!R)_\dg$ of any graded CM module so that it satisfies a strict equality $X[2]=X(c)$, hence the conclusion.
\end{proof}

As an application of \ref{futatsu} and \ref{mf}(2), we obtain the following more general version of \ref{tilt} for hypersurface singularities in which we can modify the CY dimension.
\begin{Thm}\label{hypersurface}
In the above setting, assume that $R_0$ is finite dimensional over $k$, and $\m:=(x_0,\ldots,x_d)\subset R$ is the unique $G$-maximal ideal of $R$. Let $\A\subset\sg_{0,\dg}^G\!R$ be a full dg subcategory which generates $\sg_{0}^G\!R$ as a thick subcategory.
Let $\ell\in\Z$, and assume that $p+\ell c\in G$ is torsion-free.
\begin{enumerate}
\item There exists a commutative diagram
\[ \xymatrix@R=3mm{
	\per\A\ar@{-}[d]^-\rsimeq\ar[r]&\C_{d+2\ell-1}(\A)\ar@{-}[d]^-\rsimeq\\
	\sg_0^G\!R\ar[r]&\sg_0^{G/(p+\ell c)}\!R_{\m,G/(p+\ell c)}. } \]
If $\Sing^{G/(p+\ell c)}\!R\subset\{\m\}$, then we can replace $\sg_0^G\!R$ and $\sg_0^{G/(p+\ell c)}\!R_{\m,G/(p+\ell c)}$ above by $\sg^G\!R$ and $\sg^{G/(p+\ell c)}\!R$ respectively.
\item Suppose furthermore $G=\Z$. Then there exists $V\in\D(\A^e)$ which satisfies $V^{\lotimes_\A(p+\ell c)}\simeq D\A[1-d-2\ell]$ in $\D(\A^e)$ and yields a commutative diagram below.
\[ \xymatrix@R=3mm{
	\per\A\ar@{-}[d]^-\rsimeq\ar[r]&\C_{d+2\ell-1}(\A)\ar@{-}[d]^-\rsimeq\ar[r]&\C_{d+2\ell-1}^{(1/p+\ell c)}(\A)\ar@{-}[d]^-\rsimeq\\
	\sg_0^\Z\!R\ar[r]&\sg_0^{\Z/(p+\ell c)\Z}\!R_{\m,\Z/(p+\ell c)\Z}\ar[r]&\sg_0\!R_\m } \]
If $\Sing R\subset\{\m\}$, then we can replace $\sg_0^\Z\!R$, $\sg_0^{\Z/(p+\ell c)\Z}\!R_{\m,\Z/(p+\ell c)\Z}$ and $\sg_0\!R_\m$ above by $\sg^\Z\!R$, $\sg^{\Z/(p+\ell c)\Z}\!R$ and $\sg R$ respectively.
\end{enumerate}
\end{Thm}
\begin{proof}
	Using \ref{mf}(2), the proof is analogous to \ref{CM}.
\end{proof}

We will often use the above modification for $\ell=1$.
\begin{Ex}\label{case l=1}
In \ref{hypersurface}(2), let $\ell:=1$. Then for $p_S:=\sum_{i=0}^dp_i$, we have the following commutative diagram.
\[ \xymatrix@R=3mm{
	\per\A\ar@{-}[d]^-\rsimeq\ar[r]&\C_{d+1}(\A)\ar@{-}[d]^-\rsimeq\ar[r]&\C_{d+1}^{(1/p_S)}(\A)\ar@{-}[d]^-\rsimeq\\
	\sg_0^\Z\!R\ar[r]&\sg_0^{\Z/p_S\Z}\!R_{\m,\Z/p_S\Z}\ar[r]&\sg_0\!R_\m } \]
\end{Ex}

\part{Tilting theory for singularity categories and realizations as cluster categories}\label{part2}

The aim of this part is to apply the theoretical results from the previous part to some classes of Gorenstein rings and also more generally to symmetric orders over Gorenstein rings. We obtain various triangle equivalences between their singularity categories and cluster categories of finite dimension algebras.

\section{Rings of dimension $0$ and $1$}

\subsection{Finite dimensional symmetric algebras}
In this section, we study the singularity categories of finite dimensional algebras over a field $k$.
Let $\L=\bigoplus_{i\ge0}\L_i$ be a $\Z$-graded finite dimensional self-injective algebra such that $\L_0$ has finite global dimension. It is shown in \cite{Ya} that
\[T=\bigoplus_{i\ge1}\Lambda(i)_{\ge0}\]
is a tilting object in $\sg^{\Z}\!\L$, and therefore we have a triangle equivalence
\begin{equation}\label{yamaura equivalence}
\sg^\Z\!\L\simeq\D^b(\mod A)\ \mbox{ for }\ A:=\End_{\sg\L}^\Z(T).
\end{equation}
Now we assume that $\L$ is a symmetric $k$-algebra with Gorenstein parameter $p$, that is, the socle of the $\L$-module $\L$ is contained in $\L_{-p}$. Since $\gd\L_0$ is assumed to be finite, the inequality $p<0$ holds unless $\L$ is semisimple.
Moreover, we have an isomorphism of $k$-algebras:
\[A\simeq \begin{bmatrix}
	\Lambda_0 & 0&\cdots& 0 \\
	\Lambda_1&\Lambda_0&\cdots&0\\
	\vdots& \vdots&\ddots&\vdots \\
	\Lambda_{-p-1}&\Lambda_{-p-2}&\cdots&\Lambda_0
\end{bmatrix}\]
We can apply our result \ref{tilt} to realize the singularity category of $\Lambda$ as the $(-1)$-cluster category of $A$.

\begin{Thm}\label{dim0}
Let $\L=\bigoplus_{i\ge0}\L_i$ be a finite dimensional non-semisimple symmetric algebra over a field $k$ with $\gd\La_0<\infty$ and with Gorenstein parameter $p$. 
There exists a commutative diagram of equivalences
\[ \xymatrix@R=3mm{
	\D^b(\mod A)\ar@{-}[d]^-\rsimeq\ar[r]&\C_{-1}(A)\ar@{-}[d]^-\rsimeq\ar[r]&\C_{-1}^{(1/p)}(A)\ar@{-}[d]^-\rsimeq\\
	\sg^\Z\!\La\ar[r]&\sg^{\Z/p\Z}\!\La\ar[r]&\sg\La. } \]
\end{Thm}

Note that the left commutative square was given in \cite[1.6]{Ya}.

Let us look at a few concrete examples.

\begin{Ex}
Let $A$ be a finite dimensional algebra of finite global dimension, and let
\[ \L=A\oplus DA \]
be the trivial extension algebra, on which we give a grading by $\deg A=0$ and $\deg DA=1$. Then $\L$ is a symmetric algebra over $k$, and has Gorenstein parameter $-1$. 
Then $\sg^\Z\!\L$ has a tilting object $A$ such that $A\simeq\End_{\sg\L}^\Z(A)$. By \eqref{yamaura equivalence} we have a triangle equivalence $\sg^\Z\!\L\simeq\D^b(\mod A)$ \cite{Hap}. We deduce from \ref{CM} that there is a commutative diagram of equivalences given in \cite[1.7]{Ya}.
\[ \xymatrix@R=3mm{
	\D^b(\mod A)\ar@{-}[d]^-\rsimeq\ar[r]&\C_{-1}(A)\ar@{-}[d]^-\rsimeq\\
	\sg^\Z\!\La\ar[r]&\sg\La. } \]
This is a special case of Keller's equivalence \cite[Theorem 2]{Ke05} for trivial extension dg algebras.
\end{Ex}

\begin{Ex}
Let $R=k[x]/(x^{n+1})$ with $\deg x=1$. This is a $\Z$-graded artinian hypersurface singularity with Gorenstein parameter $-n$ and $R_0=k$. By \eqref{yamaura equivalence} and \ref{hypersurface}, for each $\ell\in\Z$, we obtain a commutative diagram
\[ \xymatrix@R=3mm{
	\D^b(\mod A)\ar@{-}[d]^-\rsimeq\ar[r]&\C_{2\ell-1}(A)\ar@{-}[d]^-\rsimeq\ar[r]&\C_{2\ell-1}^{(1/\ell(n+1)-n)}(A)\ar@{-}[d]^-\rsimeq\\
	\sg^\Z\!R\ar[r]&\sg^{\Z/(\ell(n+1)-n)\Z}\!R\ar[r]&\sg R. } \]
with $A=kA_n$, the path algebra of linearly oriented type $A_n$.
For $\ell=1$ we obtain a commutative diagram
\[ \xymatrix@R=3mm{
	\D^b(\mod A)\ar@{-}[d]^-\rsimeq\ar[r]&\C_{1}(A)\ar@{-}[d]^-\rsimeq\\
	\sg^\Z\!R\ar[r]&\sg R. } \]
In particular, we have triangle equivalences
\[\cdots\simeq\C_{-3}^{(1/2n+1)}(A)\simeq\C_{-1}^{(1/n)}(A)\simeq\C_1(A)\simeq\C_{3}^{(1/n+2)}(A)\simeq\C_{5}^{(1/2n+3)}(A)\simeq\cdots.\]
\end{Ex}

\begin{Ex}
For $n\ge1$, 
let $R:=k[x_1,\ldots,x_n]/(x_ix_j,x_i^2-x_j^2\mid 1\le i<j\le n)$ with $\deg x_i=1$. This is a $\Z$-graded artinian Gorenstein ring with Hilbert function $(1,n,1)$. Notice that if $k$ is an algebraically closed, then this is the unique ring satisfying these conditions up to isomorphisms of $\Z$-graded algebras.
Since $R$ has Gorenstein parameter $-2$ and $R_0=k$, by \eqref{yamaura equivalence}, we obtain a commutative diagram
\[ \xymatrix@R=3mm{
	\D^b(\mod A)\ar@{-}[d]^-\rsimeq\ar[r]&\C_{-1}(A)\ar@{-}[d]^-\rsimeq\ar[r]&\C_{-1}^{(1/2)}(A)\ar@{-}[d]^-\rsimeq\\
	\sg^\Z\!R\ar[r]&\sg^{\Z/2\Z}\!R\ar[r]&\sg R. } \]
where $A=kQ$ is the path algebra of the $n$-Kronecker quiver 
\newcommand{\hd}{\rotatebox{90}{$\cdot\hspace{-1mm}\cdot\hspace{-1mm}\cdot$}}
$Q=[\xymatrix@C=2em{\bullet\ar@<6pt>[r]^{x_1}\ar@<-4pt>[r]_{x_n}\ar@<1pt>@{}[r]|-{\hd}&\bullet}]$.

\end{Ex}


\subsection{Gorenstein rings in dimension $1$}
Our next application is to the singularity categories of commutative Gorenstein rings of dimension one. We start with recalling a fundamental result given in \cite{BIY}. In the rest, we assume that
\begin{itemize}
\item $R=\bigoplus_{i\ge0}R_i$ is a $\Z$-graded commutative Gorenstein ring with dimension one and Gorenstein parameter $p$ such that $R_0$ is a field.
\end{itemize}
Clearly, $R$ is a $G$-local ring with $G$-maximal ideal $\m:=\bigoplus_{i>0}R_i$.
Also, we denote by $K$ the localization of $R$ at the set of all homogeneous non-zero divisors so that $K$ is the $\Z$-graded total quotient ring. One can take the minimum integer $q>0$ such that $K\simeq K(q)$ in $\Mod^\Z\!R$ \cite[4.11]{BIY}.
It is shown in \cite[1.4]{BIY} that, if $p\le0$, then
\[T=\bigoplus_{i=1}^{q-p}R(i)_{\ge0}\]
is a tilting object in $\sg_0^{\Z}\!R$, and therefore we have a triangle equivalence
\begin{align}\label{BIY equivalence}\notag
\sg_0^{\Z}\!R&\simeq\per A\ \mbox{ for }\\
A&:=\End_{\sg R}^\Z(T)
\simeq \begin{bmatrix}
	R_0 & 0&\cdots& 0 &0&0&\cdots&0\\
	R_1&R_0&\cdots&0 &0&0&\cdots&0\\
	\vdots& \vdots&\ddots&\vdots  &0&0&\cdots&0\\
	R_{-p-1}&R_{-p-2}&\cdots&R_0 &0&0&\cdots&0\\
	K_{-p} & K_{-p-1}&\cdots& K_1 &K_0&K_{-1}&\cdots&K_{1-q}\\
	K_{1-p} & K_{-p}&\cdots& K_2 &K_1&K_{0}&\cdots&K_{2-q}\\
	\vdots& \vdots&\ddots&\vdots  &\vdots& \vdots&\ddots&\vdots\\
	K_{q-p-1} & K_{q-p-2}&\cdots& K_{q} &K_{q-1}&K_{q-2}&\cdots&K_{0}
	\end{bmatrix}.
\end{align}
We can apply our result \ref{tilt} to realize the singularity category of $R$ as the $0$-cluster category of $A$.

\begin{Thm}\label{1jigen}
If $p<0$, then there exists a commutative diagram of equivalences
\[ \xymatrix@R=3mm{
	\per A\ar@{-}[d]^-\rsimeq\ar[r]&\C_{0}(A)\ar@{-}[d]^-\rsimeq\ar[r]&\C_{0}^{(1/p)}(A)\ar@{-}[d]^-\rsimeq\\
	\sg_0^\Z\!R\ar[r]&\sg_0^{\Z/p\Z}\!R_{\m,\Z/p\Z}\ar[r]&\sg_0\!R_\m. } \]
If $\Sing R\subset\{\m\}$, then one has $\sg_0^\Z\!R=\sg^\Z\!R$, $\sg_0^{\Z/p\Z}\!R_{\m,\Z/p\Z}=\sg^{\Z/p\Z}\!R$, and $\sg_0\!R_\m=\sg R$.
\end{Thm}

In the rest of this section, we apply \ref{1jigen} to some important classes of Gorenstein rings of dimension one.

We study certain families of hypersurface singularities studied in \cite{BIY}.

\begin{Ex}[Standard grading]\label{std}
Let $n>2$ be an integer, $\al_1,\ldots,\al_m\in k$ be distinct scalars, and
\[ R=k[x,y]/(f),\ \ \ f(x,y)=\textstyle{\prod}_{i=1}^m(x-\al_iy)^{n_i},\ \ \ \deg x=\deg y=1,\ \ \ n=\deg f(x,y)=\sum_{i=1}^mn_i. \]
This is a plane curve singularity with unique graded maximal ideal $\m=(x,y)$ and of Gorenstein parameter $p=2-n<0$. The algebra $A$ in \eqref{BIY equivalence} is given by the following quiver with relations \cite[Section 2.1]{BIY}:
\[ \xymatrix@!R=2mm{
	&&&&&\bullet\ar@(dr,ur)_{w_1}\\
	&&&&&\bullet\ar@(dr,ur)_{w_2}\\
	1\ar@2[r]^-x_-y&2\ar@2[r]^-x_-y&\cdots\ar@2[r]^-x_-y&n-2\ar[uurr]|-{z_1}\ar[urr]|-{z_2}\ar[drr]|-{z_{m-1}}\ar[ddrr]|-{z_{m}}&&\vdots & xy-yx,\ \ w_i^{n_i},\ \ z_i(x-\al_iy)-w_iz_iy\\
	&&&&&\bullet\ar@(dr,ur)_{w_{m-1}}\\
	&&&&&\bullet\ar@(dr,ur)_{w_m}} \]
Notice that $A$ is an Iwanaga-Gorenstein algebra with self-injective dimension at most $2$ \cite[2.1(a)]{BIY}. If $R$ is reduced (that is, $n_i=1$ for each $i$), then $A$ has global dimension at most $2$. In this case, if $n=3$, then $A$ is derived equivalent to the path algebra $kQ$ of type $D_4$, and if $n=4$, then $A$ is derived equivalent to the canonical algebra of type $(2,2,2,2)$, see \ref{T_pq singularities}.

Applying \ref{1jigen} and \ref{hypersurface}, we obtain a commutative diagram of equivalences for each $\ell\in\Z$
\[ \xymatrix@R=3mm{
	\per A\ar@{-}[d]^-\rsimeq\ar[r]&\C_{2\ell}(A)\ar@{-}[d]^-\rsimeq\ar[r]&\C_{2\ell}^{(1/(\ell n-n+2))}(A)\ar@{-}[d]^-\rsimeq\\
	\sg_0^\Z\!R\ar[r]&\sg_0^{\Z/(\ell n-n+2)\Z}\!R_{\m,\Z/(\ell n-n+2)\Z}\ar[r]&\sg_0\!R_\m.}\]
In particular, for $\ell=1$, we obtain a commutative diagram, which is closely related to \cite{HI}:
\[	\xymatrix@R=3mm{
	\per A\ar@{-}[d]^-\rsimeq\ar[r]&\C_{2}(A)\ar@{-}[d]^-\rsimeq\ar[r]&\C_{2}^{(1/2)}(A)\ar@{-}[d]^-\rsimeq\\
	\sg_0^\Z\!R\ar[r]&\sg_0^{\Z/2\Z}\!R_{\m,\Z/2\Z}\ar[r]&\sg_0\!R_\m.} \]
\end{Ex}

Now we consider hypersurface singularities of finite representation type.

\begin{Ex}[Simple curve singularities]\label{simple singularity}
Let $S=k[x,y]$ be a polynomial ring over an arbitrary field $k$, and let $R=S/(f)$ be an ADE singularity given by the table below. 
We assume that $R$ is reduced (that is, the characteristic of $k$ is not equal to $2$ for type $A_{2n-1}$ and $D_{2n}$), and consider the following minimal $\Z$-grading making $f$ homogeneous.
\[\begin{array}{c||c|c|c|c|c}
R&A_{n}&D_{n}&E_6&E_7&E_8\\ 
f&x^{n+1}-y^2&x^{n-1}-xy^2&x^4-y^3&x^3y-y^3&x^5-y^3\\ \hline\hline
(\deg x,\deg y)&\begin{array}{cc}(1,\frac{n+1}{2})&\mbox{$n$ is odd}\\ (2,n+1)&\mbox{$n$ is even}\end{array}&\begin{array}{cc}(2,n-2)&\mbox{$n$ is odd}\\ (1,\frac{n}{2}-1)&\mbox{$n$ is even}\end{array}&(3,4)&(2,3)&(3,5)\\ \hline
-p:=-p_R&\begin{array}{cc}\frac{n-1}{2}&\mbox{$n$ is odd}\\ n-1&\mbox{$n$ is even}\end{array}&\begin{array}{cc}n-2&\mbox{$n$ is odd}\\ \frac{n}{2}-1&\mbox{$n$ is even}\end{array}&5&4&7\\ \hline
-p_S&\begin{array}{cc}\frac{n+3}{2}&\mbox{$n$ is odd}\\ n+3&\mbox{$n$ is even}\end{array}&\begin{array}{cc}n&\mbox{$n$ is odd}\\ \frac{n}{2}&\mbox{$n$ is even}\end{array}&7&5&8\\ \hline
c:=\deg f&\begin{array}{cc}n+1&\mbox{$n$ is odd}\\ 2(n+1)&\mbox{$n$ is even}\end{array}&\begin{array}{cc}2(n-1)&\mbox{$n$ is odd}\\ n-1&\mbox{$n$ is even}\end{array}&12&9&15
\end{array}\]
These simple curve singularities satisfies $\Sing R\subset\{(x,y)\}$, and hence \eqref{BIY equivalence} is a triangle equivalence
\[\sg^{\Z}\!R\simeq\per kQ,\]
where $Q$ is a Dynkin quiver given by the following table, see \cite[Section 2.2]{BIY}:
\[\begin{array}{c||c|c|c|c|c}
R&A_{n}&D_{n}&E_6&E_7&E_8\\ \hline\hline
Q&\begin{array}{cc}D_{\frac{n+3}{2}}&\mbox{$n$ is odd}\\ A_{n}&\mbox{$n$ is even}\end{array}
&\begin{array}{cc}A_{2n-3}&\mbox{$n$ is odd}\\ D_{n}&\mbox{$n$ is even}\end{array}&E_6&E_7&E_8
\end{array}\]
Applying \ref{1jigen} and \ref{case l=1}, we obtain commutative diagrams of equivalences for each $\ell\in\Z$
\[ \xymatrix@R=3mm{
	\per kQ\ar@{-}[d]^-\rsimeq\ar[r]&\C_{2\ell}(kQ)\ar@{-}[d]^-\rsimeq\ar[r]&\C_{2\ell}^{(1/(p+\ell c))}(kQ)\ar@{-}[d]^-\rsimeq\\
	\sg^\Z\!R\ar[r]&\sg^{\Z/(p+\ell c)\Z}\!R\ar[r]&\sg R,}
\qquad
\xymatrix@R=3mm{
	\per kQ\ar@{-}[d]^-\rsimeq\ar[r]&\C_{2}(kQ)\ar@{-}[d]^-\rsimeq\ar[r]&\C_{2}^{(1/p_S)}(kQ)\ar@{-}[d]^-\rsimeq\\
	\sg^\Z\!R\ar[r]&\sg^{\Z/p_S\Z}\!R\ar[r]&\sg R.} \]

\end{Ex}

Next we consider hypersurface singularities of tame representation type \cite{DG}.

\begin{Ex}[$T_{pq}$ singularities]\label{T_pq singularities}
We consider a $\Z$-graded reduced hypersurface singularity
\[R=k[x,y]/(f)\ \mbox{ for }\ f=\left\{
\begin{array}{ll}
{\displaystyle \prod_{1\le i\le 4}(x-\alpha_iy)}&(\deg x,\deg y)=(1,1),\\
{\displaystyle \prod_{1\le i\le 3}(x-\alpha_iy^2)}&(\deg x,\deg y)=(2,1).
\end{array}
\right.\]
over an arbitrary field $k$.
This is a class of $T_{pq}$ singularities $R=k[x,y]/(f)$ with $f=x^p+\gamma x^2y^2+y^q$, $\gamma\in k\backslash\{0,1\}$, where $(p,q,\deg x,\deg y)=(4,4,1,1)$ or $(3,6,2,1)$.

Then the algebra $A$ in \eqref{BIY equivalence} is derived equivalent to a \emph{canonical algebra of type $(2,2,2,2)$}, which is given by the following quiver with relations for some $\lambda\in k\backslash\{0,1\}$, see \cite[Section 2,3]{BIY}:
\[{\small\xymatrix@C=4em@R=.1em{
 & \bullet \ar[rdd]^{b_1}& \\
 & \bullet \ar[rd]|{b_2} &&b_1a_1+b_2a_2+b_3a_3=0\\
 \bullet \ar[ruu]^{a_1} \ar[ru]|{a_2} \ar[rd]|{a_3} \ar[rdd]_{a_4} & & \bullet &b_1a_1+\lambda b_2a_2+b_4a_4=0.\\
 & \bullet \ar[ru]|{b_3}&\\
  & \bullet \ar[ruu]_{b_4}&
}}\]
More explicitly, $\lambda$ is given by $\lambda=(\alpha_1-\alpha_4)(\alpha_2-\alpha_3)(\alpha_1-\alpha_3)^{-1}(\alpha_2-\alpha_4)^{-1}$ for the first case and 
$\lambda=(\alpha_2-\alpha_3)(\alpha_1-\alpha_3)^{-1}$ for the second case.
We refer to section \ref{results for GL} below for a background of general canonical algebras.
\end{Ex}

\subsection{Numerical semigroup rings}
Let $\NN=\{0,1,2,\ldots\}$ be the set of non-negative integers. A {\it numerical semigroup} is a submonoid $S\subset\NN$ (containing the unit $0$) whose complement is finite. Let $k$ be an arbitrary field and let
\[ R=k[S] \]
be the corresponding semigroup ring. Such a commutative ring is called a {\it numerical semigroup ring}. We regard $k[S]$ as a subalgebra of $k[\NN]=k[t]$ with the usual identification $n\leftrightarrow t^n$. Since $R$ is a $1$-dimensional domain it is certainly Cohen-Macaulay. In what follows we consider the grading on $R$ induced from the standard grading $\deg t=1$ on $k[t]$. The graded total quotient ring $K$ of $R$ is $k[\Z]=k[t^{\pm1}]$. 

Let us first prepare some combinatorial notion on numerical semigroups.
\begin{Def}
Let $S\subset\NN$ be a numerical semigroup.
\begin{enumerate}
	\item The {\it Frobenius number} of $S$ is $a_S:=\max(\Z\setminus S)$.
	\item We say $S$ is {\it symmetric} if $\Z\setminus S=\{a_S-n\mid n\in S\}$.
\end{enumerate}
\end{Def}

The above notions are related to the well-known structure of the canonical module and characterization of Gorensteinness.
We denote by $D\colon M=\bigoplus_{i\in\Z}M_i\mapsto\bigoplus_{i\in\Z}\Hom_k(M_{-i},k)$ the graded dual.
\begin{Prop}[{\cite[4.4.8]{BH}}]\label{Gor}
\begin{enumerate}
	\item $\om:=D(K/R)$ is the canonical module for $R$.
	\item $R$ is Gorenstein if and only if $S$ is symmetric.
	\item\label{Frob} When $R$ is Gorenstein, its Gorenstein parameter is equal to the minus of the Frobenius number of $S$.
\end{enumerate}
\end{Prop}

In particular, numerical semigroup rings give a class of positively graded $1$-dimensional commutative Gorenstein ring with negative Gorenstein parameter. By \ref{1jigen} we obtain the following.
\begin{Thm}\label{num}
Let $S\subsetneq\NN$ be a symmetric numerical semigroup with Frobenius number $a$, and $R=k[S]$ the semigroup ring.
\begin{enumerate}
	\item The object $T:=\bigoplus_{i=1}^{a+1}R(i)_{\geq0}$ is a tilting object in $\sg^{\Z}\!R$.
	\item\label{end} The endomorphism algebra $A:=\End_{{\sg R}}^\Z(T)$ is given by the following matrix algebra.
	\[ \left( 
	\begin{array}{ccccc}
		R_0 & 0&\cdots& 0&0 \\
		R_1&R_0&\cdots&0&0\\
		\vdots& \vdots&\ddots&\vdots&\vdots \\
		R_{a-1}&R_{a-2}&\cdots&R_0&0\\
		k&k&\cdots&k&k
	\end{array}
	\right) \]
	\item There exists a commutative diagram of equivalences
	\[ \xymatrix@!R=3mm{
		\per A\ar[r]\ar@{-}[d]^-\rsimeq&\C_0(A)\ar[r]\ar@{-}[d]^-\rsimeq&\C_0^{(1/a)}(A)\ar@{-}[d]^-\rsimeq\\
		\sg^\Z\!R\ar[r]&\sg^{\Z/a\Z}\!R\ar[r]&\sg R. } \]
\end{enumerate}
\end{Thm}
\begin{proof}
	We know that $R$ is Gorenstein since $S$ is symmetric, that $R_0=k$ by the connectedness assumption, and that the Gorenstein parameter is $-a<0$ by \ref{Gor}(\ref{Frob}) and $S\neq\NN$. By $K=k[t^{\pm1}]$ we see $K(1)\simeq K$. Then the assertion (1) follows from the result from \cite{BIY} stated at the beginning of this section. We have (2) by \cite[1.4(c)]{BIY}, and (3) by \ref{tilt}.
\end{proof}

We can give a more combinatorial description of the endomorphism algebra $A$. For a numerical semigroup $S\subset\NN$, we associate a poset $\P(S)$ as follows.
\begin{Def}\label{poset}
	Let $S\subset\NN$ be a numerical semigroup with Frobenius number $a$. The poset $\P(S)$ is defined on the set $\{1,\ldots,a\}\sqcup\{a+1\}$ with the following order:
	\begin{itemize}
		\item $i\leq j$ if and only if $j-i\in S$ for $1\leq i,j\leq a$.
		\item $i< a+1$ for all $1\leq i\leq a$.
	\end{itemize}
\end{Def}

\begin{Cor}\label{incidence}
In the setting of \ref{num}, the endomorphism algebra $A=\End_{{\sg R}}^\Z(T)$ is isomorphic to the incidence algebra of $\P(S)$.
\end{Cor}
\begin{proof}
	This follows easily from the description of $A$ given in \ref{num}(\ref{end}). 
\end{proof}

Let us first look at some examples where $S$ is generated by $2$ elements.
\begin{Ex}\label{S_{p,q}}
Let $0<p<q$ be relatively prime integers, and $S_{p,q}$ the numerical semigroup ring generated by $p$ and $q$.
The semigroup algebra of $S$ is isomorphic to the hypersurface singularity
	\[ R=k[S_{p,q}]=k[x,y]/(x^p-y^q) \]
with $\deg x=q$ and $\deg y=p$ by comparing their Hilbert series. The Gorenstein parameter of $R$ is $p+q-pq$. The poset $\P(S_{p,q})$ defined in \ref{poset} is described, for example, as follows. For $(p,q)=(2,2n+1)$ we have $a=2n-1$ and the Hasse diagram is given by
	\[ \xymatrix@R=1.5em{
		1\ar[r]&3\ar[r]&\cdots\ar[r]&2n-3\ar[r]&2n-1\ar[d]\\
		2\ar[r]&4\ar[r]&\cdots\ar[r]&2n-2\ar[r]&2n,} \]
	so that the incidence algebra is the path algebra of type $A_{2n}$ (see \ref{simple singularity}). For $(p,q)=(3,3n+1)$ we have $a=6n-1$ and the Hasse diagram of $\P(S_{p,q})$ looks as below.
{\small
\[
\xymatrix@!C=2.5em@R=1.5em{
&1\ar[r]\ar[rrrrd]&\cdots\ar[r]&3n-5\ar[r]\ar[rrrrd]&3n-2\ar[r]\ar[rrrrd]&3n+1\ar[r]&\cdots\ar[r]&6n-5\ar[r]&6n-2\ar[rdd]\\
&2\ar[r]\ar[rrrrd]&\cdots\ar[r]&3n-4\ar[r]\ar[rrrrd]&3n-1\ar[r]&3n+2\ar[r]&\cdots\ar[r]&6n-4\ar[r]&6n-1\ar[rd]\\
&3\ar[r]\ar[rrrrd]&\cdots\ar[r]&3n-3\ar[r]\ar[rrrrd]&3n\ar[r]&3n+3\ar[r]&\cdots\ar[r]&6n-3\ar[rr]&&6n\\ 
1\ar[r]&4\ar[r]&\cdots\ar[r]&3n-2\ar[r]&3n+1\ar[r]&3n+4\ar[r]&\cdots\ar[r]&6n-2\ar[rru]}
\]
}
\end{Ex}

\section{Quotient singularities}
Let $k$ be a field and $G\subset\SL_d(k)$ a finite subgroup such that $|G|\neq0$ in $k$. It naturally acts on the polynomial ring and let $R$ be its invariant subring:
\begin{equation}\label{S and R}
S=k[x_1,\ldots,x_d], \quad R=S^G.
\end{equation}
Throughout this section we use a common notation: $\frac{1}{n}(a_1,\ldots,a_d)$ denotes the element $\diag(\zeta^{a_1},\ldots,\zeta^{a_d})\in\GL_d(k)$ for a primitive $n$-th root of unity $\zeta$.

We give a $\Z$-grading on $S$ by $\deg x_i=1$ so that the invariant subring $R$ inherits a grading from $S$. We call this grading on $R$ the {\it standard grading}. This makes $R$ into a Gorenstein ring of Gorenstein parameter $d$ \cite[6.5]{IT}. The following result gives a tilting object in the $\Z$-graded singularity category of $R$ for this $\Z$-grading. We write $\Om_S$ the syzygy of a graded $S$-module, the kernel of the projective cover in $\md^\Z\!S$. For $X\in\md^\Z\!R$, we denote by $[X]_{\CM}$ the maximal direct summand of $X$ which is Cohen-Macaulay. It is unique up to isomorphism since $\md^\Z\!R$ is a Krull-Schmidt category.

\begin{Prop}[{\cite[2.7]{IT}}]\label{TA}
Let $R=S^G$ be the $\Z$-graded quotient singularity with the standard grading. Assume $R$ has an isolated singularity.
\begin{enumerate}
\item $T=\bigoplus_{p=1}^d[\Om_S^pk(p)]_{\CM}$ is a tilting object for $\sg^\Z\!R$.
\item The endomorphism algebra $A^{\st}:=\End_{\sg R}^\Z(T)$ has finite global dimension.
\end{enumerate}
Consequently, there exists a triangle equivalence $\sg^\Z\!R\simeq\D^b(\mod A^\st)$.
\end{Prop}

Applying \ref{TA} and our result \ref{tilt}, we obtain the following description of the ungraded singularity categories of quotient singularities.

\begin{Thm}\label{quot}
Let $R=S^G$ be the quotient singularity with the standard grading which is an isolated singularity, and $A^\st$ as in \ref{TA}. Then there exists a commutative diagram of equivalences
\[ \xymatrix@R=3mm{
	\D^b(\mod A^\st)\ar[r]\ar@{-}[d]^-\rsimeq&\C_{d-1}(A^\st)\ar[r]\ar@{-}[d]^-\rsimeq&\C_{d-1}^{(1/d)}(A^\st)\ar@{-}[d]^-\rsimeq\\
	\sg^\Z\!R\ar[r]&\sg^{\Z/d\Z}\!R\ar[r]&\sg R.} \]
\end{Thm}

Sometimes the invariant ring $R$ concentrates in degrees multiple some integer. In this case we can divide the grading by this integer, making the Gorenstein parameter of $R$ smaller.
	
Let $n$ be an integer dividing $d$, and here we discuss the case where $R$ concentrates in degrees $n\Z$.
In this case, we can define the \emph{divided grading} of $R$ by setting the degree $i$ part as $R_{ni}$ for each $i\in\Z$.

\begin{Lem}
Let $R=S^G$ be the invariant ring with the standard grading. Then $R$ concentrates in degrees multiple of $n$ if and only if $G$ contains the cyclic group generated by $\frac{1}{n}(1,\ldots,1)$.
\end{Lem}

\begin{proof}
If $G$ contains $\frac{1}{n}(1,\ldots,1)$, then $R\subset S^{\frac{1}{n}(1,\ldots,1)}=S^{(n)}$, the $n$-th Veronese subring. Conversely, if $R$ is concentrated in degree $n\Z$, then its quotient field is contained in that of $S^{(n)}$, thus $\langle\frac{1}{n}(1,\ldots,1)\rangle\subset G$ by Galois theory.
\end{proof}

If $R$ is $n\Z$-graded, the category of $\Z$-graded modules $\Mod^\Z\!R$ canonically breaks into $n$ mutually isomorphic categories. For $M\in\Mod^\Z\!R$ we denote by $M=\bigoplus_{i\in\Z/n\Z}M^{i}$ the decomposition of $M$ along this splitting of categories.
\begin{Lem}\label{copy}
Let $G\subset\SL_d(k)$ subgroup containing $\frac{1}{n}(1,\ldots,1)$. Then the algebra $A^\st=\End_{\sg R}^\Z(T)$ in \ref{TA}(2) is the direct product of algebras $A_i:=\End_{\sg R}^\Z(T^i)$ along $1\leq i\leq n$, which are mutually derived equivalent.
\end{Lem}

\begin{proof}
	Since the tilting object $T$ gives an equivalence $\sg^\Z\!R\simeq\D^b(\mod A)$, the pieces $T^i$ yield equivalences between the factors of $\sg^\Z\!R$ and $\D^b(\mod A^\st)=\D^b(\mod A_1)\times\cdots\times\D^b(\mod A_n)$. Also, each factor of $\sg^\Z\!R$ are mutually isomorphic by the degree shift functor, thus $A_i$ are mutually derived equivalent. 
\end{proof}
We will refer to $\End_{\sg R}^\Z(T^0)$ as $A^{\div}$.
\begin{Thm}\label{div}
Let $n$ be an integer dividing $d$ and $p:=d/n$. Let $G\subset\SL_d(k)$ be a finite subgroup, containing $\frac{1}{n}(1,\ldots,1)$ such that $R=S^G$ is an isolated singularity. We give a divided grading on $R$. Then there exists a commutative diagram of equivalences
	\[ \xymatrix@R=3mm{
		\D^b(\mod A^{\div})\ar[r]\ar@{-}[d]^-\rsimeq&\C_{d-1}(A^{\div})\ar@{-}[d]^-\rsimeq\ar[r]&\C_{d-1}^{(1/p)}(A^{\div})\ar@{-}[d]^-\rsimeq\\
		\sg^\Z\!R\ar[r]&\sg^{\Z/p\Z}\!R\ar[r]&\sg R. } \]
\end{Thm}

\begin{proof}
	We have by \ref{TA} an equivalence $\sg^\Z\!R\simeq\D^b(\mod A)$, where the left-hand-side is the standard (non-divided) grading, and the right-hand-side is equivalent to the direct product of $n$ copies of $\D^b(\mod A^{\div})$. Dividing the grading of $R$ by $n$ we deduce an equivalence $\sg^\Z\!R\simeq\D^b(\mod A^{\div})$. Noting that the Gorstein parameter is $d/n=p$, we obtain the conclusion by \ref{tilt}.
\end{proof}

As an application of \ref{div}, we immediately obtain a result of Keller-Reiten \cite{KRac} and Keller-Murfet-Van den Bergh \cite{KMV}.
\begin{Ex}\label{KR}
Let $d=3$ so that $S=k[x,y,z]$. Let $G$ be the subgroup of $\SL_3(k)$ generated by $\frac{1}{3}(1,1,1)$, put $R=S^G$, on which we give the standard grading. It is easy to compute the endomorphism ring $A^\st$ in \ref{TA} to be the disjoint union of three $3$-Kronecker quivers $Q_3\colon\xymatrix{\bullet\ar[r]^-3&\bullet}$ (see \cite[8.16]{IT}).

Since $R$ is concentrated in degrees multiples of $3$, the algebra $A^\st$ is the direct product of $3$ derived equivalent copies of $A^{\div}$ which is necessarily $kQ_3$. By \ref{div} we deduce equivalences
\[ \xymatrix@R=3mm{
		\D^b(\mod kQ_3)\ar[r]\ar@{-}[d]^-\rsimeq&\C_2(kQ_3)\ar@{-}[d]^-\rsimeq\\
		\sg^\Z\!R\ar[r]&\sg R} \]
for the divided grading on $R$. Now the right column recovers \cite{KRac}.
\end{Ex}

\begin{Ex}\label{KMV}
Let $d=4$ so that $S=k[x,y,z,w]$. Let $G$ be the subgroup of $\SL_4(k)$ generated by $\frac{1}{2}(1,1,1,1)$, and put $R=S^G$.
Consider the standard grading on $R$. Then the endomorphism algebra $A^\st$ of the tilting object $T$ given in \ref{TA} can be computed to be the product of $2$ copies of the $6$-Kronecker quiver:
\[ A^\st=kQ_6\times kQ_6, \quad Q_6=\xymatrix{\bullet\ar[r]^-6&\bullet}. \]
Since $R$ is concentrated in even degrees, we can divide the grading by $2$, and in view of \ref{copy} we must have $A^{\div}=kQ_6$. Then we see by \ref{div} that there is a commutative diagram of equivalences
\[ \xymatrix@R=3mm{
	\D^b(\mod kQ_6)\ar@{-}[d]^-\rsimeq\ar[r]&\C_3(kQ_6)\ar[r]\ar@{-}[d]^-\rsimeq&\C_3^{(1/2)}(kQ_6)\ar@{-}[d]^-\rsimeq\\
	\sg^\Z\!R\ar[r]&\sg^{\Z/2\Z}\!R\ar[r]&\sg R,} \]
whose rightmost column as proved by Keller-Murfet-Van den Bergh \cite{KMV}.
\end{Ex}

\begin{Ex}\label{dim4}
	Let $d=4$ so that $S=k[x,y,z,w]$. Let $G$ be the subgroup of $\SL_4(k)$ generated by $\frac{1}{4}(1,1,3,3)$. Since $G$ contains $\frac{1}{2}(1,1,1,1)$ we can give $R=S^G$ the divided grading. Let us describe the algebra $A^{\div}$ for this case. By \cite[2.9]{IT} the endomorphism ring $A^\st$ in \ref{TA} is isomorphic to $eBe$ for the algebra $B$ and its idempotent $e$ defined as follows: Consider the quiver below on the left. The arrows which go one place to the right or three places to the left are labelled by $x,y$, and which go one place to the left or three places to the right are labelled by $z,w$. Then define $B$ as its path algebra modulo the exterior relations, that is, $v^2=0$ for any linear combination $v$ of $x,y,z,w$. Now let $e$ be the idempotent of $B$ corresponding to the $12$ vertices on the right three columns.
	
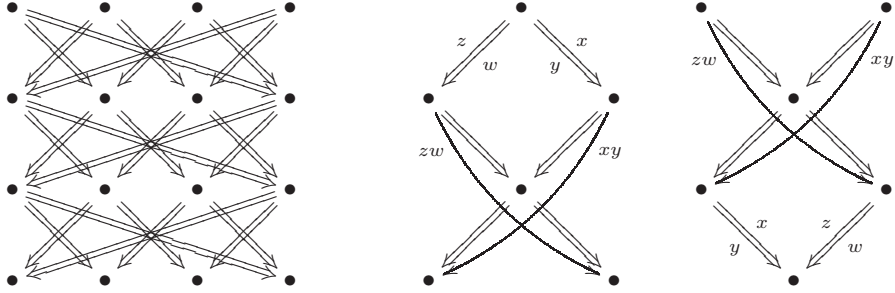
\[ \xymatrix{
		\bullet\ar@2[dr]\ar@2[drrr]&\bullet\ar@2[dr]\ar@2[dl]&\bullet\ar@2[dr]\ar@2[dl]&\bullet\ar@2[dl]\ar@2[dlll]\\
		\bullet\ar@2[dr]\ar@2[drrr]&\bullet\ar@2[dr]\ar@2[dl]&\bullet\ar@2[dr]\ar@2[dl]&\bullet\ar@2[dl]\ar@2[dlll]\\
		\bullet\ar@2[dr]\ar@2[drrr]&\bullet\ar@2[dr]\ar@2[dl]&\bullet\ar@2[dr]\ar@2[dl]&\bullet\ar@2[dl]\ar@2[dlll]\\
		\bullet&\bullet&\bullet&\bullet}
	\qquad\qquad
	\xymatrix{
		&\bullet\ar@2[dr]^-x_-y\ar@2[dl]_-z^-w&\\
		\bullet\ar@2[dr]\ar@/_10pt/[ddrr]_<<<<<{zw}&&\bullet\ar@2[dl]\ar@/^10pt/[ddll]^<<<<<{xy}\\
		&\bullet\ar@2[dr]\ar@2[dl]&\\
		\bullet&&\bullet}
	\qquad
	\xymatrix{
		\bullet\ar@2[dr]\ar@/_10pt/[ddrr]_<<<<<{zw}&&\bullet\ar@2[dl]\ar@/^10pt/[ddll]^<<<<<{xy}\\
		&\bullet\ar@2[dr]\ar@2[dl]&\\
		\bullet\ar@2[dr]^-x_-y&&\bullet\ar@2[dl]_-z^-w\\
		&\bullet&} \]
	Note that $A^\st=eBe$ has two connected components, corresponding to the fact that $R$ is concentrated in even degrees. These two algebras are not isomorphic (in fact opposite) to each other, but they are derived equivalent by \ref{copy}. Letting $A^{\div}$ be one of them, which is presented by one of the quivers on the right above. The arrows are labelled similarly as those in $B$, with two additional arrows going two places down labelled by $xy$ and $zw$, respectively.
	
	We conclude by \ref{div} that there is a commutative diagram of equivalences
	\[ \xymatrix@R=3mm{
		\D^b(\mod A^{\div})\ar[r]\ar@{-}[d]^-\rsimeq&\C_{3}(A^{\div})\ar@{-}[d]^-\rsimeq\ar[r]&\C_{3}^{(1/2)}(A^{\div})\ar@{-}[d]^-\rsimeq\\
		\sg^\Z\!R\ar[r]&\sg^{\Z/2\Z}\!R\ar[r]&\sg R. } \]
\end{Ex}

\begin{Ex}[Simple surface singularities]
Let $G\subset\SL_2(k)$ be a finite subgroup. They are classified by the ADE Dynkin diagrams as follows, where $\zeta_\ell$ denotes a primitive $\ell$-th root of $1$.
\[
\begin{array}{c||c||c|c|c}
	\text{type}&\text{generator(s)}&\text{equation}&\text{degree}&h=\deg f\\ \hline\hline
	A_n&\begin{pmatrix}\zeta_{n+1}&0\\0&\zeta_{n+1}^{-1}\end{pmatrix}&x^{n+1}+yz&(1,p,n+1-p)&n+1\\ \hline
	D_n&\begin{pmatrix}\zeta_{2n-4}&0\\0&\zeta_{2n-4}^{-1}\end{pmatrix}, \begin{pmatrix}0&\zeta_4\\ \zeta_4&0\end{pmatrix}&x^{n-1}+xy^2+z^2&(2,n-2,n-1)&2(n-1)\\ \hline
	E_6&\Frac{1}{\sqrt{2}}\begin{pmatrix}\zeta_8&\zeta_8^3\\ \zeta_8&\zeta_8^7\end{pmatrix}, D_4&x^4+y^3+z^2&(3,4,6)&12\\ \hline
	E_7&\begin{pmatrix}\zeta_8^3&0\\ 0&\zeta_8^5\end{pmatrix}, E_6&x^3y+y^3+z^2&(4,6,9)&18\\ \hline
	E_8&\Frac{1}{\sqrt{5}}\begin{pmatrix}\zeta_5^4-\zeta_5&\zeta_5^2-\zeta_5^3\\ \zeta_5^2-\zeta_5^3&\zeta_5-\zeta_5^4\end{pmatrix}, \Frac{1}{\sqrt{5}}\begin{pmatrix}\zeta_5^2-\zeta_5^4&\zeta_5^4-1\\ 1-\zeta_5&\zeta_5^3-\zeta_5\end{pmatrix}&x^5+y^3+z^2&(6,10,15)&30
\end{array}
\]
The group $G$ contains an element $\frac{1}{2}(1,1)$ in all cases except for type $A_{2n}$. In these cases, with the divided grading on $R$, we obtain the following commutative diagram of equivalences for the corresponding Dynkin quivers $Q$. 
\[ \xymatrix@R=3mm{
	\D^b(\mod kQ)\ar[r]\ar@{-}[d]^-\rsimeq&\C_1(kQ)\ar@{-}[d]^-\rsimeq\\
	\sg^\Z\!R\ar[r]&\sg R } \]
Indeed, we know that the algebra $A^\st$ is a direct product of two copies of Dynkin quiver $Q$ of the corresponding type \cite[8.14]{IT}, while by \ref{copy} it is also the direct product of two derived equivalent copies of $A^{\div}$. It follows that $A^{\div}$ must be derived equivalent to $kQ$, hence the desired diagram by \ref{div}.

Let us take another point of view. It is well-known that $R$ is a hypersurface singularity given as in the right half of the above table. We consider the $\Z$-grading therein, which is the divided grading except for type $A_n$, and $p$ in type $A_n$ is an arbitrary integer.
Applying \ref{hypersurface}, we obtain commutative diagrams of equivalences for each $\ell\in\Z$
\[ \xymatrix@R=3mm{
	\D^b(\mod kQ)\ar@{-}[d]^-\rsimeq\ar[r]&\C_{2\ell+1}(kQ)\ar@{-}[d]^-\rsimeq\ar[r]&\C_{2\ell+1}^{(1/(1+\ell h))}(kQ)\ar@{-}[d]^-\rsimeq\\
	\sg^\Z\!R\ar[r]&\sg^{\Z/(1+\ell h)\Z}\!R\ar[r]&\sg R.}
\]
Thus the category $\C_{2\ell+1}^{(1/(1+\ell h))}(kQ)$ is independent of $\ell\in\Z$. This can be explained also by the fractionally Calabi-Yau property of dimension $(h-2)/h$ of $\D^b(\mod kQ)$, where $h=\deg f$ is the Coxeter number of $Q$. 
\end{Ex}

\section{Geigle-Lenzing complete intersections and Grassmannian cluster categories}

\subsection{Geigle-Lenzing complete intersections}\label{results for GL}
Following \cite{GL,HIMO}, we recall the basic setup for Geigle-Lenzing theory. Let $k$ be a field, and let $d\geq0$ and $n\geq1$ be integers. The {\it Geigle-Lenzing complete intersection} is the commutative ring
\[ R=k[T_0,\ldots,T_d,x_1,\ldots,x_n]/(x_1^{p_1}-l_1,\ldots,x_n^{p_n}-l_n), \]
where $l_1,\ldots,l_n$ are linear forms in $T_0,\ldots,T_d$ {as linear independent as possible}, and $p_1,\ldots,p_n\ge2$ are integers called weights.
It is graded by the abelian group
\[ \LL=\left(\Z\vec{c}\oplus\textstyle{\bigoplus}_{i=1}^n\Z\vec{x_i}\right)/\langle p_i\vec{x_i}-\vec{c}\mid 1\leq i\leq n\rangle \]
with $\deg T_i=\vec{c}$ and $\deg x_i=\vec{x_i}$.
It is an $\LL$-local ring with $\LL$-maximal ideal $\m=(T_0,\ldots,T_d,x_1,\ldots,x_n)$. We have $\dim R=\dim^\LL\!R=d+1$, and $R$ has Gorenstein parameter
\[ -\vec{\om}:=\textstyle{\sum_{i=1}^n}\vec{x_i}-(n-d-1)\vec{c}. \]
Let $\LL_+$ be the submonoid of $\LL$ generated by $\vec{x_1},\ldots,\vec{x_n}$, and define a partial order $\le$ on $\LL$ by $\vec{x}\le\vec{y}\Longleftrightarrow\vec{y}-\vec{x}\in\LL_+$. We set
\[\vec{\delta}:=d\vec{c}+2\vec{\om}\in\LL.\]
Then $R$ is
\begin{itemize}
	\item regular if and only if $n\leq d+1$ if and only if $0\not\le\vec{\delta}$, and
	\item a complete intersection of codimension $n-d-1$ if $n>d+1$.
\end{itemize}
Consider the interval $I:=[0,\vec{\delta}]$ in $\LL$. For each $X\in\mod^{\LL}\!R$, we define $X_I:=\bigoplus_{\vec{x}\in I}X_{\vec{x}}\in\mod^{\LL}\!R$. Let
\[T:=\bigoplus_{\vec{x}\in I}R(\vec{x})_I\in\mod^{\LL}\!R\]
The {\it CM canonical algebra} is a subring
\[A^{\CM}=(R_{\x-\y})_{\x,\y\in I}\]
of the full matrix ring $M_I(R)$ \cite{HIMO} (above Theorem 3.20), and can be identified with the endomorphism algebra of $T$:
\[A^{\CM}=\End_R^{\LL}(T).\]
From our point of view, the following tilting result on Geigle-Lenzing complete intersection is a ``graded'' assumption part of \ref{CM}.

\begin{Thm}[{\cite[3.20]{HIMO}}]\label{himo}
$\sg^{\LL}\!R$ has a tilting object $T$ with $\End_{\sg R}^\LL(T)=A^{\CM}$, and there exists a triangle equivalence
\[ \xymatrix{\sg^\LL\!R\ar@{-}[r]^-\simeq&\D^b(\mod A^{\CM}).}\]
\end{Thm}

Applying \ref{CM}(1) we obtain the following ``ungraded'' part.
\begin{Thm}\label{GL}
Let $R$ be a Geigle-Lenzing complete intersection. Assume that $\vec{\om}$ is not torsion and that $\Sing^{\LL/(\vec{\om})}\!R\subset\{\m\}$. Then there exists a commutative diagram of equivalences
\[ 	\xymatrix@R=3mm{
	\D^b(\mod A^{\CM})\ar[r]\ar@{-}[d]^-\rsimeq&\C_d(A^{\CM})\ar@{-}[d]^-\rsimeq\\
	\sg^\LL\!R\ar[r]&\sg^{\LL/(\vec{\om})}\!R. } \]
\end{Thm}

Note that $\Sing^{\LL}\!R\subset\{\m\}$ always holds \cite[3.32]{HIMO}. We expect that $\Sing^{\LL/(\vec{\om})}\!R\subset\{\m\}$ also holds if $\vec{\om}$ is not torsion.

\begin{Ex}
Let us discuss the hypersurface case $n=d+2$ in more detail. After a linear change of variables 
we may assume $R$ has the form 
\[ R=k[x_1,\ldots,x_{d+2}]/(x_1^{p_1}+\cdots+x_{d+2}^{p_{d+2}}) \]
with $\LL=(\Z\vec{c}\oplus{\bigoplus}_{i=1}^{d+2}\Z\vec{x_i})/\langle p_i\vec{x_i}-\vec{c}\mid 1\leq i\leq n\rangle$ and
with Gorenstein parameter $-\vec{\om}=\sum_{i=1}^{d+2}\vec{x_i}-\vec{c}$.

In this case we have an even more explicit description of $A^{\CM}$ \cite[3.23(b)]{HIMO}. Let $\AA_n$ be a linear oriented quiver of type $A_n$:
\[\xymatrix@!R=2mm@!C=2mm{
\AA_n\ar@{}[r]|-=&\bullet\ar[r]&\bullet\ar[r]&\bullet\ar[r]&\cdots\ar[r]&\bullet} \]
Then $A^{\CM}$ is isomorphic to the tensor product of path algebras of linearly oriented quivers of type $A$:
\[ A^{\CM}\simeq\bigotimes_{i=1}^{d+2}k\AA_{p_i-1}. \]
By \ref{hypersurface} we have a commutative diagram below of equivalences for each $\ell\in\Z$ such that $-\vec{\om}+\ell\vec{c}$ is not torsion and $\Sing^{\LL/(-\vec{\om}+l\vec{c})}\!R\subset\{\m\}$.
\[ \xymatrix@R=3mm{
	\D^b(\mod A^{\CM})\ar@{-}[d]^-\rsimeq\ar[r]&\C_{d+2\ell}(A^{\CM})\ar@{-}[d]^-\rsimeq\\
	\sg^\LL\!R\ar[r]&\sg^{\LL/(-\vec{\om}+\ell \vec{c})}\!R. } \]
In particular for $\ell=1$ we have
\[ \xymatrix@R=3mm{
	\D^b(\mod A^{\CM})\ar@{-}[d]^-\rsimeq\ar[r]&\C_{d+2}(A^{\CM})\ar@{-}[d]^-\rsimeq\\
	\sg^\LL\!R\ar[r]&\sg^{\LL/(\vec{x})}\!R } \]
for $\vec{x}=\sum_{i=1}^{d+2}\vec{x_i}$.
The case $d=1$ will be used in the next subsection.
\end{Ex}

\begin{Ex}
We consider the case $d\ge1$, $n=d+3$ and $p_i=2$ for all $1\le i\le n$. In this case, $\vec{\delta}=\vec{c}$ holds, and the algebra $A^{\CM}$ is given by the following quiver with two relations, where without loss of generality, we assume $\ell_i=T_{i-1}$ for each $1\le i\le d+1$ and $\ell_i=\sum_{j=1}^{d+1}\lambda_{i,j-1}T_{j-1}$ for $i=d+2,d+3$.
\[\xymatrix@R=.5em@C=4em{
&\vec{x_1}\ar[rdd]|{x_1}\\
&\vec{x_2}\ar[rd]|{x_2}\\
0\ar[ruu]|{x_1}\ar[ru]|{x_2}\ar[rd]|{x_{d+2}}\ar[rdd]|{x_{d+3}}&\vdots&\vec{c}&x_{d+2}^2=\sum_{j=1}^{d+1}\lambda_{d+2,j-1}x_j^2\\
&\vec{x_{d+2}}\ar[ru]|{x_{d+2}}&&x_{d+3}^2=\sum_{j=1}^{d+1}\lambda_{d+3,j-1}x_j^2\\
&\vec{x_{d+3}}\ar[ruu]|{x_{d+3}}
}\]
\end{Ex}


\subsection{Grassmannian cluster categories}
We give an application of general results above to Grassmannian cluster categories of Jensen-King-Su.
After recalling its basic definitions, we observe that this is a special case of Geigle-Lenzing hypersurface, and give an equivalence between the cluster category of an algebra of global dimension $\leq2$.

Let $k$ be an arbitrary field and let $0<\ell<n$ be integers. Consider the algebra $\L$ presented by the following quiver with relations.
\[ \xymatrix@R=3mm{
	&1\ar@<2pt>[dr]^-x\ar@<2pt>[dl]^-y&\\
	n\ar@<2pt>[ur]^-x\ar@<2pt>[dd]^-y&&2\ar@<2pt>[ul]^-y\ar@<2pt>[dd]^-x && xy=yx,\quad x^\ell=y^{n-\ell}\\
	\\
	\cdots\ar@<2pt>[uu]^-x&&\cdots\ar@<2pt>[uu]^-y } \]
Then the {\it Grassmannian cluster category} \cite{JKS} is the Frobenius category $\CM\L$.
It can be defined in a slightly different way: Let
\[ R=k[x,y]/(x^\ell-y^{n-\ell}) \]
be a one-dimensional hypersurface singularity. Give a $\Z/n\Z$-grading on $R$ by $\deg x=1$ and $\deg y=-1$. 
Then the following elementary property shows that the Grassmannian category can be equivalently defined as $\CM^{\Z/n\Z}\!R$.

\begin{Lem}\label{Z/nZ}
There is an equivalence $\Mod\L\simeq\Mod^{\Z/n\Z}\!R$, which induces an equivalence $\CM\L\simeq\CM^{\Z/n\Z}\!R$ and a triangle equivalence $\sg \L\simeq\sg^{\Z/n\Z}\!R$.
\end{Lem}

\begin{proof}
This algebra $\L$ is isomorphic to the smash product
\[R\#(\Z/n\Z)=(R_{j-i})_{i,j\in\Z/n\Z}.\]
The assertion is immediate from \cite[3.1]{IL}.
\end{proof}

Now note that our hypersurface $R=k[x,y]/(x^{\ell}-y^{n-\ell})$ is a Geigle-Lenzing hypersurface (with $d=0$ and $n=2$) graded by an abelian group
\[\LL=(\Z\vec{x}\oplus\Z\vec{y})/(\ell\vec{x}-(n-\ell)\vec{y}),\] 
where $\deg x=\vec{x}$ and $\deg y=\vec{y}$. The $\Z/n\Z$-grading on $R$ given above is specialization of $\LL$-grading via the isomorphism of groups
\begin{equation}\label{L to Z/nZ}
\LL/(\vec{x}+\vec{y})\simeq\Z/n\Z,\ \vec{x}\mapsto1,\ \vec{y}\mapsto-1.
\end{equation}
Fix a subset $I\subset[1,n]$ with $|I|=l$ and define a $\Z$-grading of $\Lambda$ by 
\[\deg(x:i\to i+1):=\left\{\begin{array}{ll}0&i\in I\\ 1&i\notin 
I,\end{array}\right.\ \deg(y:i+1\to i):=\left\{\begin{array}{ll}1&i\in I\\ 0&i\notin I.\end{array}\right.\]
As in \ref{Z/nZ}, there is an equivalence $\Mod^\Z\!\L\simeq\Mod^{\LL}\!R$, which induces equivalences
\begin{equation}\label{graded Morita for L and Z}
\CM^\Z\!\L\simeq\CM^{\LL}\!R\ \text{ and }\  \sg^\Z\!\L\simeq\sg^{\LL}\!R.
\end{equation}
These observations lead to the following main result of this section, giving an equivalence between the Grassmannian cluster category and the $2$-cluster category of an algebra of global dimension $\leq2$.

\begin{Thm}\label{gra}
Let $k$ be an arbitrary field, $0<\ell<n$ be integers, and let $A:=k\AA_{\ell-1}\otimes_kk\AA_{n-\ell-1}$ which is presented by the following grid quiver with commutativity relations at each square.
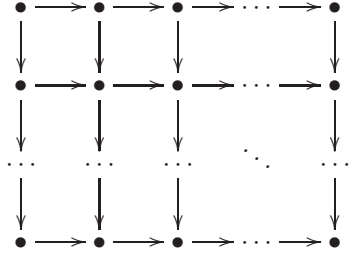
\[ \xymatrix@!R=2mm@!C=2mm{
	\bullet\ar[r]\ar[d]&\bullet\ar[r]\ar[d]&\bullet\ar[r]\ar[d]&\cdots\ar[r]&\bullet\ar[d]\\
	\bullet\ar[r]\ar[d]&\bullet\ar[r]\ar[d]&\bullet\ar[r]\ar[d]&\cdots\ar[r]&\bullet\ar[d]\\
	\cdots\ar[d]&\cdots\ar[d]&\cdots\ar[d]&\ddots&\cdots\ar[d]\\
	\bullet\ar[r]&\bullet\ar[r]&\bullet\ar[r]&\cdots\ar[r]&\bullet } \]
Then $\Sing^{\Z/n\Z}\!R\subset\{\m\}$ holds, and there exists a commutative diagram of equivalences
	\[ \xymatrix@R=3mm{
		\D^b(\mod A)\ar[r]\ar@{-}[d]^-\rsimeq&\C_2(A)\ar@{-}[d]^-\rsimeq\\
		\sg^\LL\!R\ar[r]^(.45){\eqref{L to Z/nZ}}\ar@{-}[d]^\rsimeq&\sg^{\Z/n\Z}\!R\ar@{-}[d]^\rsimeq\\
		\sg^\Z\!\L\ar[r]&\sg\L. } \]
\end{Thm}

\begin{proof}
The lower part of the diagram is immediate from \ref{Z/nZ} and \eqref{graded Morita for L and Z}.

To prove the upper part, thanks to \ref{himo} and \ref{case l=1}, it suffices to show $\Sing^{\Z/n\Z}\!R\subset\{\m\}$. Note that $Z:=R_0$ is a polynomial algebra $k[t]$ for $t:=xy$, and the natural map $Z\to e_1\La e_1$ is an isomorphism.
For each $\p\in\Spec^{\Z/\n\Z}\!R\setminus\{\m\}$, we have an isomorphism $R_{\p,\Z/n\Z}\#(\Z/n\Z)\simeq\La_{Z\cap\p}$ and hence an equivalence
\[\mod^{\Z/n\Z}R_{\p,\Z/n\Z}\simeq\mod\La_{Z\cap\p}.\]
Let $(-)_t$ be a localization with respect to the multiplicative set $\{t^i\mid i\ge0\}$. 
Then it suffices to prove that $\Lambda_t$ has finite global dimension since $\La_{Z\cap\p}$ is a localization of $\Lambda_t$.

For each $1\le i\le n$, $x^{i-1}:e_1\La_t\to e_i\La_t$ is an isomorphism in $\mod\La_t$ with inverse $y^{i-1}t^{1-i}$. Thus $\Lambda_t$ is Morita equivalent to $e_1\La_te_1\simeq k[t^{\pm1}]$, and hence $\Lambda_t$ has finite global dimension, as desired.
\end{proof}

Combining \ref{gra} and \ref{S_{p,q}}, we obtain the following result. We write $\AA_n\otimes\AA_m$ for the poset whose Hasse diagram is given by the $(n\times m)$-grid quiver as in \ref{gra}.

\begin{Cor}
If $p$ and $q$ are relatively prime, the incidence algebras of $\P(S_{p,q})$ and of $\AA_{p-1}\otimes\AA_{q-1}$ are derived equivalent.
\end{Cor}

For example, $k\AA_2\otimes_kk\AA_{3n}$ is derived equivalent to the incidence algebra of $\P(S_{3,3n+1})$, which looks as below, where each low contains $2n$ vertices and the top low is identified with the bottom one. 
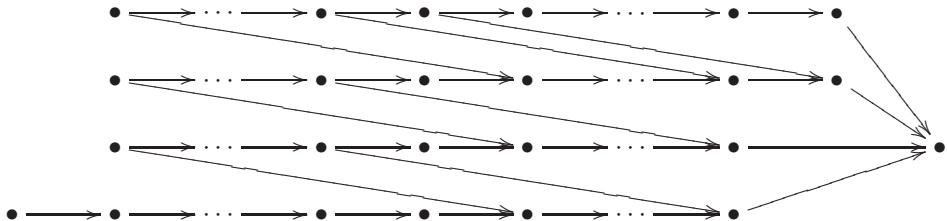
\[\xymatrix@!C=1.5em@R=1.5em{
&\bullet\ar[r]\ar[rrrrd]&\cdots\ar[r]&\bullet\ar[r]\ar[rrrrd]&\bullet\ar[r]\ar[rrrrd]&\bullet\ar[r]&\cdots\ar[r]&\bullet\ar[r]&\bullet\ar[rdd]\\
&\bullet\ar[r]\ar[rrrrd]&\cdots\ar[r]&\bullet\ar[r]\ar[rrrrd]&\bullet\ar[r]&\bullet\ar[r]&\cdots\ar[r]&\bullet\ar[r]&\bullet\ar[rd]\\
&\bullet\ar[r]\ar[rrrrd]&\cdots\ar[r]&\bullet\ar[r]\ar[rrrrd]&\bullet\ar[r]&\bullet\ar[r]&\cdots\ar[r]&\bullet\ar[rr]&&\bullet\\ 
\bullet\ar[r]&\bullet\ar[r]&\cdots\ar[r]&\bullet\ar[r]&\bullet\ar[r]&\bullet\ar[r]&\cdots\ar[r]&\bullet\ar[rru]}
\]
\subsection{Infinite Grassmannian cluster categories}

Let $\ell\geq1$ and consider a $\Z$-graded non-reduced one-dimensional hypersurface singularity
\[R=k[x,y]/(x^\ell),\quad \deg x=1, \deg y=-1. \]
The \emph{Grassmannian cluster category of infinite rank} \cite{ACFGS} is defined as the category $\CM^\Z\!R$ of $\Z$-graded Cohen-Macaulay $R$-modules. By virtue of this grading, the {\it graded} singularity category $\sg_0^\Z\!R$ is already $2$-CY. Since $R$ is not positively graded, however, we cannot directly apply (\ref{BIY equivalence}) to obtain a description of the graded singularity category. We will nevertheless realize it as the $2$-cluster category of an additive category.

Let $\A$ the additive category given by the quiver with relations below.
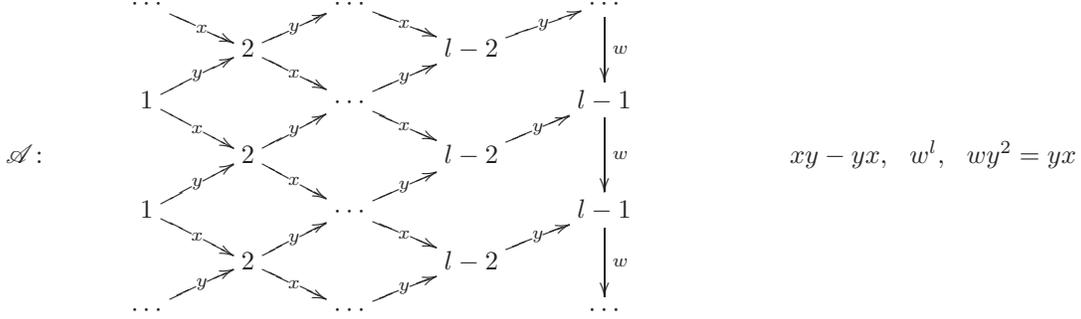
\[ \xymatrix@R=2mm{
	&\cdots\ar[dr]|-x&&\cdots\ar[dr]|-x&&\cdots\ar[dd]^-w\\
	&&2\ar[ur]|-y\ar[dr]|-x&&\ell-2\ar[ur]|-y&\\
	&1\ar[ur]|-y\ar[dr]|-x&&\cdots\ar[ur]|-y\ar[dr]|-x&&\ell-1\ar[dd]^-w\\
	\A\colon&	&2\ar[ur]|-y\ar[dr]|-x&&\ell-2\ar[ur]|-y& &&xy-yx,\ \ w^\ell,\ \ wy^2=yx\\
	&1\ar[ur]|-y\ar[dr]|-x&&\cdots\ar[ur]|-y\ar[dr]|-x&&\ell-1\ar[dd]^-w\\
	&&2\ar[ur]|-y\ar[dr]|-x&&\ell-2\ar[ur]|-y&\\
	&\cdots\ar[ur]|-y&&\cdots\ar[ur]|-y&&\cdots} \]
It is a covering of the algebra
\[ \xymatrix@!R=2mm{ A\colon&1\ar@2[r]^-x_-y&2\ar@2[r]^-x_-y&\cdots\ar@2[r]^-x_-y&\ell-2\ar[r]_y&\ell-1\ar@(ur,dr)^{w}&& xy-yx,\ \ w^\ell,\ \ wy^2=yx}. \]
Define the $\Z$-grading of $A$ by $\deg x=1$, $\deg y=0$ and $\deg w=1$. Then we have an equivalence
\[\A\simeq\proj^{\Z}\!A.\]
The main result of this section is existence of a triangle equivalence
\[ \xymatrix{\sg_0^{\Z}\!R\ar@{-}[r]^-\simeq&\C_2(\A)}. \]
This is in fact the ``ungraded'' part of the commutative diagram in our general result \ref{CM}.
More precisely, we consider the $\Z^2$-grading on $R$ given by $\deg x=(1,0)$ and $\deg y=(0,1)$ which lifts the original $\Z$-grading via a surjection $\Z^2\twoheadrightarrow\Z$, $(i,j)\mapsto i-j$.
\begin{Thm}\label{gra2}
There exists a commutative diagram of equivalences
\[ \xymatrix@R=3mm{
	\per\A\ar[r]\ar@{-}[d]^-\rsimeq&\C_2(\A)\ar@{-}[d]^-\rsimeq\\
	\sg_0^{\Z^2}\!R\ar[r]&\sg_0^{\Z}\!R.} \]
\end{Thm}
\begin{proof}
(1) First we prove $R_{\m,\Z}=R$.
Indeed, let $r=r(x,y)\in R\setminus\m$ be a homogeneous element. Then 
the constant term $c:=r(0,0)$ of $r$ is non-zero and hence $r\in R_0$. On the other hand, the subalgebra $R_0$ of $R$ is generated by $xy$. Thus $r-c$ is nilpotent by $x^\ell=0$, and hence $r$ is a unit of $R$.

(2) Next we prove that there is an equivalence
	\[ \sg_0^{\Z^2}\!R\simeq\per\A. \]
We give (still another) $\Z$-grading on $R=k[x,y]/(x^\ell)$ by $\deg x=\deg y=1$. For distinction we denote this grading group by $\Z^\prime$. Since this is a positive grading we can apply (\ref{BIY equivalence}): $T=\bigoplus_{i=1}^{\ell-1}R(i)_{\geq0}\in\sg^{\Z^\prime}\!R$ is a tilting object and there is a triangle equivalence
	\[ \sg_0^{\Z^\prime}\!R\simeq\per A \]
	for $A=\sEnd_R^\Z(T)$, which is described as above by \ref{std}.
	
	We next give a $\Z^2$-grading on $R$ by $\deg x=(1,0)$ and $\deg y=(0,1)$. There is a surjection $\Z^2\to\Z^\prime$ taking $(1,0)\mapsto 1$ and $(0,1)\mapsto1$, which yields a covering functor $\pi\colon\sg^{\Z^2}\!R\to\sg^{\Z^\prime}\!R$. Each direct summand $R(i)_{\geq0}=(x,y)^i(i)$ of $T$ is generated by monomials of $x,y$ and hence a $\Z^2$-graded submodule of $R(i)$. Thus the tilting object $T$ belongs to the image of $\pi$. By \ref{covering} the inverse image $\pi^{-1}(T)$ is a tilting subcategory, which is equivalent to the additive category $\A$ defined above, so that we have a desired equivalence.

(3) Finally we apply \ref{hypersurface} with $G=\Z^2$ to prove the claim. The Gorenstein parameter of $\Z^2$-graded $R$ is $a=(1-\ell,1)$ and the degree of the defining equation $x^\ell$ is $c=(\ell,0)$, thus \ref{hypersurface}(1) for $\ell=1$ yields a commutative diagram of equivalences
	\[ \xymatrix@R=3mm{
		\per\A\ar[r]\ar@{-}[d]^-\rsimeq&\C_2(\A)\ar@{-}[d]^-\rsimeq\\
		\sg_0^{\Z^2}\!R\ar[r]&\sg_0^{\Z^2/(1,1)}\!R'} \] 
where $R':=R_{\m,\Z^2/(1,1)}$ for $\m:=(x,y)$. It remains to prove $\sg_0^{\Z^2/(1,1)}\!R'=\sg_0^\Z\!R$.
The $\Z^2/(1,1)$-grading on $R$ coincides with our original $\Z$-grading on $R$ via the isomorphism
$\Z^2/(1,1)\simeq\Z$ given by
$(1,0)\mapsto1$ and $(0,1)\mapsto-1$. Thus $R'=R$ holds by (1), and the claim follows.
\end{proof}

\begin{appendix}
\renewcommand{\theequation}{\Alph{section}.\arabic{equation}}
\section{Keller's dg quotients}\label{dgquotient}
\subsection{Existence and uniqueness of dg quotients}
We collect some fundamental facts around dg quotients, with an emphasis on the construction and how to relate their structures with the original categories. Among some well-established ways to construct and describe the dg quotients, we follow the original one due to Keller \cite{Ke99} based on localization theory of triangulated categories.

Throughout this section we let $K$ be a commutative ring (not necessarily a field), and everything will be $K$-linear. 
Let $\B$ be a dg category and $\A\subset\B$ a full dg subcategory. Then the functor $-\lotimes_\A\B\colon\D(\A)\to\D(\B)$ is fully faithful, by which we will identify $\D(\A)$ as a thick subcategory of $\D(\B)$.
\begin{Def}\label{define dgq}
Let $\B$ be a dg category and $\A\subset\B$ a full dg subcategory. The {\it dg quotient} of $\B$ by $\A$ is a dg category $\C$ together with a dg functor $\B\to\C$ such that there exists an equivalence $\D(\B)/\D(\A)\simeq\D(\C)$ making the following diagram commutative, where the horizontal functors are canonical ones.
\[\xymatrix@R=5mm{
\D(\B)\ar[r]\ar@{=}[d]&\D(\B)/\D(\A)\ar[d]\\
\D(\B)\ar[r]&\D(\C)
}\]
\end{Def}

The fundamental theorem is the existence and uniqueness of dg quotients.
\begin{Thm}[{\cite[4.6]{Ke99}}]
For any small dg category $\B$ and its full subcategory $\A$, the dg quotient $\B\to\B/\A$ exists, which is unique up to unique isomorphism in $\Hmo$.
\end{Thm}

The aim of this section is to give a detailed account of the construction of $\B/\A$ above. It is based on localization theory of triangulated categories, which we recall first.

A sequence
\[\N\xrightarrow{F}\T\xrightarrow{G}\U\]
of triangulated categories is called {\it exact} if $F$ is fully faithful, $G\circ F=0$, and $G$ induces a triangle equivalence $\T/F(\N)\xrightarrow{\simeq}\U$. Recall that a thick subcategory of a triangulated category with coproducts is called a \emph{localizing subcategory} if it is closed under coproducts.
We start with the following results which is a consequence of Brown representability theorem (see \cite[5.2]{Ke94}\cite[8.3.3]{Ne01}).
\begin{Prop}\label{locthy}
	\renewcommand{\labelenumi}{(\alph{enumi})}
	Let $\T$ and $\U$ be compactly generated triangulated categories and $G\colon\T\to\U$ a triangle functor. Then the following are equivalent.
	\begin{enumerate}
		\item There is a localizing subcategory $\N\subset\T$ such that $\N\xrightarrow{}\T\xrightarrow{G}\U$ is exact.
		\item $G$ has a fully faithful right adjoint $H$.
		\item there is a stable $t$-structure $\T=\N\perp\N'$ such that $G$ gives an equivalence $\N'\to\U$.
	\end{enumerate}
	Moreover, the triangle associated to $X\in\T$ is given by extending the unit map $X\to HG(X)$.
\end{Prop}

This theorem allows us to realize the Verdier quotient $\U$ of $\T$ as a thick subcategory, which leads to the following construction.
\begin{Con}\label{con}
Replacing $\B$ by its cofibrant resolution if necessary, we assume that each morphism complex of $\B$ is cofibrant over $K$ (see \cite[Appendix B]{Dr}, \cite[2.3(3)]{To}).
Consider the multiplication map $\B\lotimes_\A\B\to\B$ and complete it to the triangle
\[ \xymatrix{ \B\lotimes_\A\B\ar[r]&\B\ar[r]& M } \]
in $\D(\B^e)$, where $\B^e=\B^\op\otimes_K\B$. We define
\[ \B/\A:=\REnd_\B(M), \]
which is equipped with the canonical dg functor $\B\to\B/\A$.
More precisely, the dg category $\B/\A$ has
\begin{itemize}
	\item objects: same as $\B$,
	\item morphisms: $\cHom_{\B}(P(-,B),P(-B^\prime))$ for $B,B^\prime\in\B$, where $P\to M$ is a cofibrant resolution of $M$ over $\B^\op\otimes_K\B$,
\end{itemize}
and the left $\B$-module structure on $P$ yields a natural functor $\B\to\B/\A$. Note that $P$ being cofibrant over $\B^\op\otimes_K\B$ implies it is cofibrant as a right $\B$-module since $\B$ is cofibrant over $K$. It follows that $\cEnd_\B(P)$ computes $\REnd_\B(M)$.
\end{Con}

Recall that we identify $\D(\A)$ as a localizing subcategory of $\D(\B)$ via the fully faithful functor $-\lotimes_\A\B\colon\D(\A)\to\D(\B)$. Then there exists a stable $t$-structure
\[ \D(\B)=\D(\A)\perp\U \]
by \ref{locthy}, where $\U:=\D(\A)^\perp:=\{ X\in\D(\B)\mid \Hom_{\D(\B)}(Y,X)=0 \text{ for all } Y\in\D(\A)\}$.
\begin{Lem}\label{truncation}
For each $X\in\D(\B)$, the truncation of $X$ with respect to the stable $t$-structure $\D(\B)=\D(\A)\perp\U$ is given by the triangle
\[ \xymatrix{ X\lotimes_\A\B\ar[r]&X\ar[r]&X\lotimes_\B M } \]
obtained by applying $X\lotimes_\B-$ to \ref{con}. 
\end{Lem}
\begin{proof}
	It is enough to show that for each $Y\in\D(\A)$ (or the identified $Y\lotimes_\A\B\in\D(\B)$), the induced map $\Hom_{\D(\B)}(Y\lotimes_\A\B,X\lotimes_\A\B)\to\Hom_{\D(\B)}(Y\lotimes_\A\B,X)$ is an isomorphism. This follows from the commutativity of the diagram
	\[ \xymatrix@R=5mm{
		\RHom_\B(Y\lotimes_\A\B,X\lotimes_\A\B)\ar[r]&\RHom_\B(Y\lotimes_\A\B,X)\\
		\RHom_\A(Y,X)\ar@{=}[r]\ar[u]^-\lsimeq&\RHom_\A(Y,X)\ar@{-}[u]_-\rsimeq, } \]
	where the left vertical map is an isomorphism via the fully faithful functor $-\lotimes_\A\B\colon\D(\A)\to\D(\B)$, and the right vertical map is an adjunction.
\end{proof}

Let us now prove that the constructed functor $\B\to\B/\A$ is the dg quotient.
\begin{Thm}[\cite{Ke99}]\label{existence}
Let $\B$ be a dg category and $\A\subset\B$ a full dg subcategory. 
\begin{enumerate}
\item The induction functor $-\lotimes_\B\B/\A\colon\D(\B)\to\D(\B/\A)$ identifies with the Verdier quotient $\D(\B)\to\D(\B)/\D(\A)$.
\item Let $\B\to\C$ be a dg functor such that the composition $\D(\A)\to\D(\B)\to\D(\C)$ is zero. Then $\B\to\C$ uniquely factors through $\B\to\B/\A$ in $\Hmo$.
\item Let $\B\to\C$ be another dg quotient of $\B$ by $\A$. Then there exists a unique isomorphism $\B/\A\to\C$ in $\Hmo$ making the following diagram commutative.
\[\xymatrix@R=5mm{
\B\ar[r]\ar@{=}[d]&\B/\A\ar[d]\\
\B\ar[r]&\C
}\]
\end{enumerate}
\end{Thm}
\begin{proof}
	Recall the stable $t$-structure $\D(\B)=\D(\A)\perp\U$ for $\U=\D(\A)^\perp=\{ X\in\D(\B)\mid \Hom_\B(Y,X)=0 \text{ for all } Y\in\D(\A)\}$.
	
	(1)  It is enough to show that $\U$ is compactly generated by $\{M(-,B)\mid B\in\B\}$, which is done in the following four steps.
	
	\Step{For each $B\in\B$, the $\B$-module $M(-,B)$ lies in $\U$.}\label{MU}
	
	Applying \ref{truncation} to $X=\B(-,B)$ we get a triangle
	\[ \xymatrix{ \B(-,B)\lotimes_\A\B\ar[r]&\B(-,B)\ar[r]&M(-,B) }, \]
	which shows that the last term is in $\U$.
		 
	\Step{$\U$ has coproducts and the inclusion $\U\hookrightarrow\D(\B)$ preserves coproducts.}\label{cont}
	
	Since $\D(\A)\subset\D(\B)$ is compactly generated by compact objects in $\D(\B)$, we see that $\U=\D(\A)^\perp$ is closed under coproducts in $\D(\B)$, so the assertion follows.
		
	\Step{The object $M(-,B)$ is compact in $\U$ for each $B\in\B$.}\label{compact}
	
	Let $\{U_i\}_i$ be a set of objects in $\U$, and we have to show that the canonical map $\bigoplus_i\Hom_\U(M(-,B),U_i)\to\Hom_\U(M(-,B),\bigoplus_iU_i)$ is an isomorphism. By Step \ref{cont} above, this is equivalent to saying that the morphism $\bigoplus_i\RHom_\B(M(-,B),U_i)\to\RHom_\B(M(-,B),\bigoplus_iU_i)$ is a quasi-isomorphism. 
	Applying $\bigoplus_i\RHom_\B(-,U_i)$ and $\RHom_\B(-,\bigoplus_iU_i)$ to the triangle in Step \ref{MU} and comparing them by the natural transformation, we get a commutative diagram
	\[ \xymatrix{
		\bigoplus_i\RHom_\B(M(-,B),U_i)\ar[r]\ar[d]&\bigoplus_i\RHom_\B(\B(-,B),U_i)\ar[r]\ar[d]&\bigoplus_i\RHom_\B(\B(-,B)\lotimes_\A\B,U_i)\ar[d]\\
		\RHom_\B(M(-,B),\bigoplus_iU_i)\ar[r]&\RHom_\B(\B(-,B),\bigoplus_iU_i)\ar[r]&\RHom_\B(\B(-,B)\otimes_\A\B,\bigoplus_iU_i). } \]
	Now, the two rightmost terms are $0$ since $\B(-,B)\lotimes_\A\B$ lies in $\D(\A)$ and $U_i\in\U$, and the middle vertical map is an isomorphism. We conclude that the left vertical map is also an isomorphism.
	
	\Step{The set $\{M(-,B)\mid B\in\B\}$ generates $\U$.}
	
	By Step \ref{compact}, we only have to show that $\Hom_\U(M(-,B),U[i])=0$ for all $i\in\Z$ implies $U=0$ for each $U\in\U$. Applying $\RHom_\B(-,U)$ to the triangle in Step \ref{MU}, we have a triangle
	\[ \xymatrix{ \RHom_\B(M(-,B),U)\ar[r]&\RHom_\B(\B(-,B),U)\ar[r]&\RHom_\B(\B(-,B)\lotimes_\A\B,U) }. \]
	The first term is $0$ by assumption, and so is the last term by the stable $t$-structure. It follows that the middle term $U(B)$ is $0$ for all $B\in\B$, that is, $U=0$. 
	
	(2)  Let $\B\to\C$ be an arbitrary dg functor such that $\D(\A)\to\D(\B)\to\D(\C)$ is $0$. We claim that the $(\B/\A,\C)$-bimodule $Z\colon M\lotimes_\B\C$ gives a unique factorization $\B\to\B/\A\xrightarrow{Z}\C$ in $\Hmo$ of $\B\to\C$.
	Applying $-\lotimes_\B\C$ to the triangle $\B\lotimes_\A\B\to\B\to M$ and noting that $\C$ restricted to $\A$ is acyclic, we get an isomorphism $\C\xsimeq M\lotimes_\B\C=Z$ in $\D(\B^\op\otimes\C)$. This shows that $Z$ is perfect over $\C$ on the right and thus gives a morphism $\B/\A\to \C$ in $\Hmo$, and also the commutativitiy of the triangle
	\[ \xymatrix@R=5mm{
		\B\ar[r]\ar[dr]&\B/\A\ar[d]^-Z\\
		&\C. } \]
	To show uniqueness, it is enough to see that the dg quotient $\B\to\B/\A$ is an epimorphism in $\Hmo$. Let $\C$ be a dg category and $X,Y\in\Hom_{\Hmo}(\B/\A,\C)$ which coincide in $\Hom_{\Hmo}(\B,\C)$, in other words, isomorphic in $\D(\B^\op\otimes\C)$. Since $\B\to\B/\A$ is a localization so is $\B^\op\otimes\C\to(\B/\A)^\op\otimes\C$, which means the restriction functor $\D((\B/\A)^\op\otimes\C)\to\D(\B^\op\otimes\C)$ is fully faithful. It follows that $X\simeq Y$ already in $\D((\B/\A)^\op\otimes\C)$.
	
	(3)  This follows from (2) as it says that the dg quotient $\B\to\B/\A$ is the cokernel of $\A\to\B$ in $\Hmo$.
\end{proof}

In fact, the proof tell us that the morphism complex of the dg quotient is described as follows. Recall that an {\it exact sequence} of dg categories is a sequence $\A\to\B\to\C$ of dg categories and dg functors which induces an exact sequence of triangulated categories $\D(\A)\to\D(\B)\to\D(\C)$. Note that this amounts to saying that $\C$ is the dg quotient of $\B$ by $\A$.
\begin{Prop}
	Let $\A\xrightarrow{}\B\xrightarrow{}\C$ be an exact sequence of dg categories.
	Then there exists a triangle in $\D(\B^e)$:
	\[ \xymatrix{ \B\lotimes_\A\B\ar[r]&\B\ar[r]&\C\ar[r]&\B\lotimes_\A\B[1]. } \]
\end{Prop}
\begin{proof}
	It is enough to show that $\B/\A\simeq M$ in $\D(\B^e)$. This follows from applying $\RHom_\B(-,M)$ to the triangle $\B\lotimes_\A\B\to\B\to M$. Indeed, noting that $\RHom_\B(\B\lotimes_\A\B,M)=0$ since $\B\lotimes_\A\B\in\D(\A)$ and $M\in\U$, we get an isomorphism $\REnd_\B(M)\xsimeq M$.
\end{proof}

\subsection{Reformulation in the Morita homotopy category}
We give a reformulation of the results from the previous subsection using the language of the homotopy category $\Hmo$.
Recall from \ref{TabTo} that a morphism in $\Hmo$ corresponds bijectively to bimodules which are perfect on the right. If $X\in\D(\A^\op\otimes\B)$ is a bimodule which is perfect over $\B$, we denote by $X\colon\A\to\B$ or $\A\xrightarrow{X}\B$ when we regard the bimodule $X$ as a morphism in $\Hmo$.

Let us record the notion of exact sequences of dg categories {in $\Hmo$}, analogous to \ref{short}.
\begin{Def}\label{define exact}
We say that a morphism $X\colon\A\to\B$ in $\Hmo$ is {\it fully faithful} if the functor $-\lotimes_\A X\colon\D(\A)\to\D(\B)$ is fully faithful.
A sequence $\A\xrightarrow{X}\B\xrightarrow{Y}\C$ in $\Hmo$ is {\it exact} if the induction functors give an exact sequence of triangulated categories:
\[ \xymatrix{\D(\A)\ar[r]^-{-\lotimes_\A\B}&\D(\B)\ar[r]^-{-\lotimes_\B\C}&\D(\C). } \]
\end{Def}

The definition of dg quotients given in \ref{define dgq} can be generalized as follows.

\begin{Def}
Let $X\colon\A\to\B$ be a fully faithful morphism. The {\it dg quotient} of $\B$ by $\A$ is a morphism $Y\colon\B\to\C$ in $\Hmo$ which fits into an exact sequence $\A\xrightarrow{X}\B\xrightarrow{Y}\C$.
\end{Def}

We have the following reformulation of \ref{existence} since any morphism in $\Hmo$ comes from a dg functor up to isomorphism.

\begin{Thm}
Let $\A\to\B$ be a fully faithful morphism in $\Hmo$. Then the dg quotient $\B\to\C$ exists uniquely up to unique isomorphism in $\Hmo$.
\end{Thm} 
In view of \ref{existence}(2), the dg quotient is nothing but the cokernel of the inclusion in $\Hmo$. It is easy to see that $X$ is the kernel of $Y$.
%


\section{Group graded commutative rings}\label{G-graded rings}
We give a background on commutative Noetherian rings graded by an arbitrary abelian group. There are some references on this subject; \cite{BH,BS,GW1,GW2} when the grading group is torsion-free, and \cite{Ka,DGL} in the arbitrary case.
In this paper we need some fundamental results which are not covered by these references, so here we give the graded version of prime ideals and of the theory of injective modules, including the structure theorem of minimal injective resolutions, and Matlis duality. 
We will omit some proofs which are easy or parallel to the classical ones.
\subsection{$G$-prime ideals}
Let $G$ be an arbitrary abelian group, and let $R$ be a commutative Noetherian ring which is $G$-graded. Let us start with the fundamental notion of prime ideals in the graded setting.
\begin{Def}\label{GSpec}
Let $R$ be a $G$-graded ring.
\begin{enumerate}
	\item We say that a homogeneous ideal $\p\subsetneq R$ is {\it $G$-prime} if $xy\in\p$ implies $x\in\p$ or $y\in\p$ whenever $x$ and $y$ are homogeneous.
	\item A {\it $G$-maximal ideal} is an ideal of $R$ which is maximal among proper homogeneous ideals.
	\item A {\it $G$-local ring} is a $G$-graded ring with a unique $G$-maximal ideal.
\end{enumerate}
We denote by $\Spec^G\!R$ the set of $G$-prime ideals of $R$, and by $\Max^G\!R$ the set of $G$-maximal ideals.
\end{Def}
When $\p\subset R$ is a $G$-prime, then the $G$-graded ring $R/\p$ is a {\it $G$-domain} in the sense that the product of two non-zero homogeneous elements is non-zero. Similarly if $\m\subset R$ is $G$-maximal, then $R/\m$ is a {\it $G$-field}; any non-zero homogeneous element is invertible. It is easy to see that if $G$ is torsion-free, then any $G$-prime is a prime (see \ref{Spec G G/H}(1)). However, the following example shows this is not the case when $G$ has torsion.
\begin{Ex}
\begin{enumerate}
	\item Let $G=\Z/2\Z$ and $R=k[x]/(x^2-1)$ with $\deg x=1$. Then $R$ is a $G$-domain, but is not a domain.
	\item More generally, let $G$ be a finitely generated abelian group, and let $R=kG$ be the group algebra. This can be seen as a $G$-graded ring by $\deg g=g$ for each $g\in G$. Then $R$ is a $G$-field.
\end{enumerate}
\end{Ex}

We first give the following simple characterization of $G$-local rings.
\begin{Prop}\label{G-local}
	Let $R$ be a $G$-graded ring.
	\begin{enumerate}
		\item There exist a map $\Spec^G\!R\to\Spec R_0$ given by $\p\to\p_0$ and a map $\Spec R_0\to\Spec^G\!R$ given by $\q\mapsto\widetilde{\q}:=\bigoplus_{g\in G}\{r\in R_g\mid rR_{-g}\subset\q\}$, satisfying $(\widetilde{\q})_0=\q$.
		\item They induce mutually inverse bijections $\Max^G\!R\to\Max  R_0$ and $\Max  R_0\to\Max^G\!R$.
		\item $R$ is $G$-local if and only if $R_0$ is local.
	\end{enumerate}
\end{Prop}
\begin{proof}
	(1)  If $\p$ is a $G$-prime of $R$ then $R/\p$ is a $G$-domain, hence its degree $0$ part $R_0/\p_0$ is a domain. This shows $\p\mapsto\p_0$ is well-defined. Suppose conversely that $\q\in\Spec R_0$. Let $x, y\in R\setminus\widetilde{\q}$ be homogeneous elements Then there are homogeneous $x^\prime, y^\prime\in R$ such that $xx^\prime\in R_0\setminus \q$ and $yy^\prime\in R_0\setminus\q$. Then $xx^\prime yy^\prime\in R_0\setminus\q$, so that $xy\not\in\widetilde{\q}$. This shows $\q\mapsto\widetilde{\q}$ is well-defined. Finally it is immediate that $\widetilde{\q}_0=\q$. \\
	(2)  The similar reasons as above show that the given maps restrict to ones between ($G$-)maximal ideals. It remains to verify $\widetilde{\m_0}=\m$ for all $\m\in\Max^G\!R$. Clearly we have an inclusion $\m\subset\widetilde{\m_0}$, so $\m$ being $G$-maximal gives the equality.\\
	(3)  This is a direct consequence of (2).
\end{proof}

Let us next discuss the notion of {\it graded dimension}.
\begin{Def}\label{Gdim}
	Let $R$ be a commutative $G$-graded ring.
	\begin{enumerate}
		\item The {\it $G$-height} of a $G$-prime ideal $\p$, denoted $\height^G\!\p$, is the largest length integer $n$ such that there is a chain $\p_0\subsetneq\p_1\subsetneq\cdots\subsetneq\p_n$ of $G$-prime ideals.
		\item The {\it $G$-dimension} of $R$ is defined by $\dim^G\!R=\sup\{\height^G\!\p\mid \p\in\Spec^G\!R\}$.
	\end{enumerate}
\end{Def}
The following example shows $\dim^G\!R$ is not necessarily equal to the (ungraded) $\dim R$.
\begin{Ex}
	Let $G=\Z^n$ and $R$ the group algebra of $G$, thus $R=k[x_1^{\pm1},\ldots,x_n^{\pm1}]$ with $\deg x_i=(0,\ldots,1,\ldots0)$. Then $R$ is a $G$-field so $\dim^G\!R=0$, but $\dim R=n$.
\end{Ex}

We have the following relationship between the graded and ungraded dimensions.
\begin{Thm}\label{dim's}
Let $R$ be a commutative Noetherian $G$-graded ring.
\begin{enumerate}
\item For any $\p\in\Spec^G\!R$ and $\q\in\Min R/\p$, we have $\height^G\!\p=\height\q$.
\item We have $\dim^G\!R\leq\dim R$.
\end{enumerate}
\end{Thm}

We need several preparations for the proof. The first one is easy, where we denote by $\Min_R\!M$ the set of minimal primes of an $R$-module $M$.
\begin{Lem}\label{min}
	Let $R$ be an arbitrary commutative Noetherian ring and $I\subset J$ be ideals of $R$. Then for any $\p\in\Min_R\!R/J$, there is $\q\in\Min_R\!R/I$ such that $\q\subset \p$.
\end{Lem}
\begin{proof}
	Localizing at $\p$, we have that $\p R_\p$ is a minimal prime of $R_\p/JR_\p$. Then any element of $\Min_{R_\p}\!R_\p/IR_\p$ is contained in $\p R_\p$, and thus its preimage $\q$ in $R$ is a minimal prime of $R/I$ contained in $\p$.
\end{proof}

For $\p\in\Spec R$, let $\p^\ast$ be the ideal of $R$ generated by all homogeneous elements in $\p$. Then $\p^\ast\in\Spec^G\!R$, thus we get a map $(-)^\ast\colon\Spec R\to\Spec^G\!R$.
\begin{Prop}\label{ass}
Let $R$ be a commutative Noetherian $G$-graded ring.
\begin{enumerate}
\item For each $\p\in\Spec^G\!R$ and any $\q\in\Ass R/\p$, we have $\q^\ast=\p$.
\item For each $\p\in\Spec^G\!R$ and any $\q\in\Supp R/\p$, we have $\q^\ast\supset\p$.
\item The map $(-)^\ast\colon\Spec R\to \Spec^G\!R$ is surjective.
\end{enumerate}
\end{Prop}
\begin{proof}
	(1)  Since $\p\subset\q$ and $\p$ is homogeneous, we clearly have $\p\subset\q^\ast$. Consider the injection $f\colon R/\q\hookrightarrow R/\p$. Suppose that there is an element $x\in\q^\ast\setminus\p$, which we may take homogeneous. Since $x$ is a homogeneous element outside $\p$ and $\p$ is a $G$-prime, the multiplication by $x$ is injective on $R/\p$, hence also on $\Im f\simeq R/\q$. This contradicts $x\in\q^\ast\subset\q$, and therefore $\p=\q^\ast$.
	
	(2)  This follows from (1) since any minimal element of $\Supp R/\p$ is in $\Ass R/\p$.
	
	(3)  By (1) taking an associated prime gives a section to the map $(-)^\ast$.
\end{proof}

Now we can prove \ref{dim's}.
\begin{proof}[Proof of \ref{dim's}]
	(1)  We first prove the inequality $\height^G\!\p\leq\height\q$. Let $\p=\p_0\supsetneq\p_1\supsetneq\cdots\supsetneq\p_d$ be a chain of $G$-prime ideals. 
	For an inclusion $\p_i\supset\p_{i+1}$ in $\Spec^G\!R$, by applying \ref{min} to $(I,J)=(\p_{i+1},\p_{i})$ and $\q_i\in\Min_R\!R/\p_i$, we find $\q_{i+1}\in\Min_R\!R/\p_{i+1}$ such that $\q_{i+1}\subset\q_{i}$.
	Since $\q_i^\ast=\p_i$ by \ref{ass}(1), so each inclusion $\q_i\supset\q_{i+1}$ must be strict.
	
	We next prove the converse inequality. If $\height^G\!\p=r$ then there is a sequence $x_1,\ldots,x_r$ of homogeneous elements in $\p$ such that there is no $G$-prime ideal $\p^\prime$ such that $(x_1,\ldots,x_r)\subset\p^\prime\subsetneq\p$. Let $\q\in\Min R/\p$. We claim that there is no prime ideal $\r$ such that $(x_1,\ldots,x_r)\subset\r\subsetneq\q$. If there is such an $\r$, then we have $(x_1,\ldots,x_r)\subset\r^\ast\subset\q^\ast$, and $\q^\ast=\p$ by \ref{ass}(1). By our choice of $x_1,\ldots,x_r$, this forces $\r^\ast=\q^\ast=\p$, hence $\p\subset\r\subsetneq\q$. This contradicts the minimality of $\q$. We deduce that $(x_1,\ldots,x_r)$ is a $\q R_\q$-primary ideal of $R_\q$, hence $\height\q\leq r$.
	
	(2)  This is a straightforward consequence of (1).
\end{proof}
%

We now turn to graded localizations and graded supports.
For $\p\in\Spec^G\!R$ we denote by $R_{\p,G}$ the localization of $R$ with respect to the homogeneous elements not in $\p$:
\[ R_{\p,G}=\{a/s \mid a\in R,\, s\not\in\p\text{: homogeneous} \}. \]
Then $R_{\p,G}$ is naturally a $G$-graded ring. It is $G$-local with $G$-maximal ideal $\p R_{\p,G}$.
For $M\in\Mod^G\!R$ we define the $G$-graded $R_{\p,G}$-module $M_{\p,G}$ by $M\otimes_RR_{\p,G}$. Then the {\it $G$-support} of $M\in\Mod^G\!R$ is
\[ \Supp_R^G\!M:=\{\p\in\Spec^G\!R\mid M_{\p,G}\neq0\}. \]

Under a suitable assumption, the ordinary support and the $G$-support is related as follows.
\begin{Lem}\label{tsujitsuma}
Let $R$ be a commutative $G$-graded Noetherian ring, and let $\m$ be a $G$-maximal ideal which is maximal as an ungraded ideal. Then the following are equivalent for $0\neq X\in\mod^G\!R$.
	\begin{enumerate}
		\renewcommand\labelenumi{(\roman{enumi})}
		\renewcommand\theenumi{\roman{enumi}}
		\item $\Supp^G_R\!X\subset\{\m\}$.
		\item $\Supp_R\!X\subset\{\m\}$.
	\end{enumerate} 
\end{Lem}
\begin{proof}
	(i)$\Rightarrow$(ii)  By (i), $X$ has a finite filtration by $R/\m$ in $\mod^G\!R$. Since $R/\m$ is simple in $\mod R$ by assumption, we obtain (ii).
	
	(ii)$\Rightarrow$(i)  It is enough to show $\Ass^G_R\!X\subset\{\m\}$. Let $\p\in\Ass^G_R\!X$. Then we have an inclusion $R/\p\hookrightarrow X$, thus $\emptyset\neq\Ass R/\p\subset\Ass X\subset\{\m\}$, and therefore $\m\in\Ass R/\p$. By \ref{ass}(1) we must have $\p=\m^\ast=\m$.
\end{proof}

\begin{Prop}\label{Spec G G/H}
Let $H$ be a torsion-free subgroup of $G$. 
\begin{enumerate}
\item Any $G$-prime is a $G/H$-prime, by which we view $\Spec^G\!R\subset\Spec^{G/H}\!R$.
\item For $X\in\mod^G\!R$, we have $\Supp^G_R\!X=\Supp^{G/H}_R\!X\cap\Spec^G\!R$.
\end{enumerate}
\end{Prop}
\begin{proof}
(1) By induction we may assume $H$ has rank $1$, thus $H\simeq\Z$. Then each $G/H$-homogeneous element $x\in R$ can be written as $x=\sum_{i\in\Z}x_i$ with each $x_i$ being $G$-homogeneous.
Now let $\p\in\Spec^G\!R$ and we have to show that $\p$ is a $G/H$-prime. Let $x,y\in R$ be $G/H$-homogeneous elements such that $xy\in\p$. Writing $x=\sum_{i\in\Z}x_i$ and $y=\sum_{j\in\Z}y_j$ as a sum of $G$-homogeneous elements, we see by induction on $i$ and $j$ that either $x$ or $y$ must lie in $\p$.

(2) Let $X\in\mod^G\!R$ and $I:=\Ann_RX$. 
For $\p\in\Spec^G\!R$, it suffices to show that $X_{\p,G}\neq0$ if and only if $X_{\p,G/H}\neq0$. Let $R^{h,G}$ be the set of $G$-homogeneous elements of $R$.
For $*=G,G/H$, $X_{\p,*}\neq0$ is equivalent to $I\cap(R^{h,*}\setminus\p)=\emptyset$. This is equivalent to $I\subset\p$ since $I$ is $G$-homogeneous. Thus the assertion follows.
\end{proof}

\subsection{Graded Matlis theory}
To simplify notation, in this subsection we denote the $G$-homogeneous localization $R_{\p,G}$ by $R_{[\p]}$, and the {\it $G$-residue field} by $k[\p]$:
\[ R_{[\p]}:=R_{\p,G}, \quad k[\p]:=R_{[\p]}/\p R_{[\p]}. \]

We discuss the injective objects in $\Mod^G\!R$. Since $\Mod^G\!R$ is a Grothendieck abelian category, every $M\in\Mod^G\!R$ has an injective hull, which we denote by $E_R^G(M)$. For each $\p\in\Spec^G\!R$ define the subgroup
\[ G_{R/\p}=\{p \in G\mid R/\p\simeq R/\p(p) \in \Mod^G\!R\}. \]
The following observation gives some alternative descriptions of $G_{R/\p}$.

\begin{Lem}\label{G_R/p}
	The following are equivalent for $\p\in\Spec^G\!R$ and $p\in G$.
	\begin{enumerate}
		\renewcommand{\labelenumi}{(\roman{enumi})}
		\renewcommand{\theenumi}{\roman{enumi}}
		\item\label{1} $R/\p\simeq R/\p(p)$ in $\Mod^G\!R$.
		\item\label{2} $E^G_R(R/\p)\simeq E^G_R(R/\p)(p)$ in $\Mod^G\!R$.
		\item\label{3} $k[\p]\simeq k[\p](p)$ in $\Mod^G\!R_{[\p]}$.
		\item\label{4} $E^G_{R_{[\p]}}(k[\p])\simeq E^G_{R_{[\p]}}(k[\p])(p)$ in $\Mod^G\!R_{[\p]}$.
	\end{enumerate}
\end{Lem}
\begin{proof}
	(\ref{1})$\Rightarrow$(\ref{2}) and (\ref{3})$\Rightarrow$(\ref{4}):  Take the injective hull. Note that $E^G_R(M(p))=E_R^G(M)(p)$.\\
	(\ref{2})$\Rightarrow$(\ref{1}): For the inclusion $R/\p\hookrightarrow E_R^G(R/\p)$, we have $R/\p=\{x\in E_R^G(R/\p)_0\mid \p x=0\}$, so any isomorphism $E_R^G(R/\p)\xsimeq E_R^G(R/\p)(p)$ restricts to $R/\p\xsimeq R/\p(p)$.\\
	(\ref{4})$\Rightarrow$(\ref{3}): Apply (\ref{2})$\Rightarrow$(\ref{1}) to $(R_{[\p]},\p R_{[\p]})$ in place of $(R,\p)$.\\
	(\ref{2})$\Leftrightarrow$(\ref{4}): It is easy to see (as in \cite[18.4(vi)]{Ma}) that $E_R^G(R/\p)$ can be seen as an $R_{[\p]}$-module, and there is an isomorphism $E_R^G(R/\p)\simeq E_{R_{[\p]}}^G(R_{[\p]}/\p R_{[\p]})$ as graded $R_{[\p]}$-modules. This immediately yields the assertion.
\end{proof}

As in the ungraded case, the injective objects in $\Mod^G\!R$ are classified by $G$-prime ideals. The following results are more explicit version of \cite[2.8]{Ka}.

\begin{Thm}\label{inj}
	Let $R$ be a commutative Noetherian ring which is $G$-graded.
	\begin{enumerate}
		\item For each $\p\in\Spec^G\!R$ the injective module $E_R^G(R/\p)$ is indecomposable.
		\item Every injective object in $\Mod^G\!R$ is a direct sum of indecomposable injective objects.
		\item There is a bijection between the pairs $(\p,p)$ of $\p\in\Spec^G\!R$ and $p\in G/G_{R/\p}$ and the isomorphism classes of indecomposable injective objects in $\Mod^G\!R$ given by $(\p,p)\mapsto E^G_R(R/\p)(p)$.
	\end{enumerate}
\end{Thm}

We next consider the minimal injective resolution of a module $M\in\Mod^G\!R$. 
By \ref{G_R/p}, we can state the following structure theorem of minimal injective resolutions.

\begin{Thm}[{\cite[2.10]{Ka}}]
	Let $M\in\Mod^G\!R$ and let $\xymatrix{0\ar[r]&M\ar[r]&I^0\ar[r]&I^1\ar[r]&I^2\ar[r]&\cdots }$ be the minimal injective resolution of $M$. Then we have
	\[ I^i=\bigoplus_{\substack{\p\in\Spec^G\!R\\ p\in G/G_{R/\p}}}E^G_R(R/\p)(-p)^{\mu^G_i(\p,p,M)} \text{ with } \mu^G_i(\p,p,M)=\dim_{k(\p)}\Ext_{R_{[\p]}}^i(k[\p],M_{[\p]})_{p}. \]
\end{Thm}


Now we focus on Gorenstein rings, which are of central interest in this paper. We start with its graded version.
\begin{Def}\label{GGor}
	We say that a Noetherian $G$-graded ring is {\it $G$-Gorenstein} if for each $G$-maximal ideal $\m$, the free module $R_{[\m]}$ has finite injective dimension in $\Mod^G\!R_{[\m]}$.
\end{Def}

If a $G$-graded ring $R$ has finite self-injective dimension in $\Mod R$, then it is $G$-Gorenstein by graded Baer criterion.
We have the following structure of minimal injective resolution of $R$ in $\Mod^G\!R$.

\begin{Thm}[{\cite[2.15]{Ka}}]\label{resolution}
	Let $R$ be a $G$-Gorenstein ring and let ${0\to R\to I^0\to I^1\to I^2\to\cdots }$ be the minimal injective resolution of $R$. Then we have
	\[ I^i=\bigoplus_{\height^G\!\p=i}E_R^G(R/\p)(p_\p), \]
	for some $p_\p\in G/G_{R/\p}$.
\end{Thm}


It follows from the above resolution that for a $G$-local $G$-Gorenstein ring $(R,\m)$ with $\dim^G\!R=d$, we have $\Ext_R^d(R/\m,R)\simeq R/\m(p_\m)$ in $\Mod^G\!R$. We call this value $p_\m\in G/G_{R/\m}$ (or its representative in $G$) the {\it Gorenstein parameter} of $R$, which uniquely exists in $G/G_{R/\m}$.
This local notion naturally leads to the following  global version.
\begin{Def}\label{GP}
	Let $R$ be a $G$-graded $G$-Gorenstein ring. We say $R$ has {\it Gorenstein parameter $p\in G$} if its image in $G/G_{R/\m}$ is the Gorenstein parameter of the $G$-local ring $R_{[\m]}$ for each $G$-maximal ideal $\m$.
\end{Def}

Define the {\it graded dual} functor $D\colon\Mod^G\!R\to\Mod^G\!R$ by $M=\bigoplus_{g\in G}M_g\mapsto\bigoplus_{g\in G}\Hom_k(M_{-g},k)$. Clearly $D$ induces a duality on the subcategory of $\Mod^G\!R$ consisting of graded modules $M$ such that each $M_g$ is finite dimensional. It is also easy to verify the ``adjunction'' isomorphism $\Hom_R(X,DY)\simeq D(X\otimes_RY)$ of graded $R$-modules for each $X,Y\in\Mod^G\!R$. 
Let us note the following explicit structure of the injective hull of the simples. 

\begin{Prop}\label{DR}
	Assume that $R_0$ is a finite dimensional $k$-algebra. Then we have an isomorphism $\bigoplus_{\m\in\Max^G\!R}E_R^G(R/\m)\simeq DR$.
\end{Prop}

\begin{proof}
	We have an isomorphism $D(R/\m)\simeq R/\m$ in $\Mod^G\!R$ for all $\m\in\Max^G\!R$. Indeed, both are simple objects in $\Mod^G\!R$ corresponding to the same $G$-maximal ideal, so they should be isomorphic up to a degree shift. Since $R/\m$ is a $G$-field, its support $\{g\in G\mid (R/\m)_g\neq0\}$ is a subgroup of $G$, hence coincides with that of $D(R/\m)$. We see $R/\m$ and $D(R/\m)$ must be isomorphic by \ref{supp} below.
	
	Next we consider the projection $R\to R/\m$ for each $\m\in\Max^G\!R$. Dualizing and composing with the isomorphism claimed above, we obtain a map $R/\m\xsimeq D(R/\m)\to DR$, which defines a morphism
	\[ \xymatrix{\disoplus_{\m\in\Max^G\!R}R/\m\ar[r]^-f&DR }. \]
	We prove $f$ is a monomorphism and is essential. Since $DR$ is injective by adjunction, we shall conclude $DR\simeq E_R^G(\bigoplus_{\m\in\Max^G\!R}R/\m)\simeq\bigoplus_{\m\in\Max^G\!R}E_R^G(R/\m)$.
	Note that $\bigoplus_{\m\in\Max^G\!R}R/\m$ is a direct sum of mutually non-isomorphic simple objects, so its submodule, in particular $\Ker f$, is of the form $\bigoplus_{\m\in I} R/\m$ for a subset $I\subset\Max^G\!R$. Since each component $R/\m\to DR$ of $f$ is non-zero, we see $\Ker f$ has to be $0$. It remains to prove $f$ is an essential extension. Let $X\subset DR$ be a graded submodule such that $X\cap(\bigoplus_{\m\in\Max^G\!R}R/\m)=0$. This means the composite $X\subset DR\twoheadrightarrow\Coker f$ is injective. Dualizing, we have that the map $\bigcap_{\m\in\Max^G\!R}\m\subset R\twoheadrightarrow DX$ is surjective. We now conclude by Nakayama's lemma that $DX=0$.
\end{proof}

\begin{Lem}\label{supp}
	Let $\m$ be a $G$-maximal ideal. Then for all $g\in G$ with $(R/\m)_g\neq0$ we have $R/\m\simeq R/\m(g)$.
\end{Lem}
\begin{proof}
	Pick $x\in R_g\setminus\m_g$. Then multiplication by $x$ gives a non-zero map $R/\m\to R/\m(g)$ between simple objects, hence an isomorphism.
\end{proof}

To conclude this subsection, let us note some observations which is used in this article.
The first one is the following special case of Matlis duality.
We put
\[ \fl_0^G\!R=\{ M\in\Mod^G\!R\mid \text{any }\p\in\Supp_R^G\!M \text{ satisfies }\height^G\!\p=d\}. \]
\begin{Prop}\label{Matlis}
	Let $R$ be a $G$-graded ring with $\dim^G\!R=d$, and let $E=\bigoplus_{\height^G\!\m=d}E_R^G(R/\m)$. Then $\Hom_R(-,E)$ give a duality
	\[ \xymatrix{ \fl_0^G\!R\ar@{<->}[r]&\fl_0^G\!R }. \]
\end{Prop}
\begin{proof}
	Note that the evaluation map gives a natural isomorphism $1\to \Hom_R(\Hom_R(-,E),E)$ on $\fl_0^G\!R$. Indeed, both are exact functors and yield isomorphisms on the simples. It follows by induction that these are naturally isomorphic on $\fl_0^G\!R$.
\end{proof}

Next we relate the above Matlis dual with the graded dual defined above. If $R_0$ is a finite dimensional algebra over a field $k$, then each $R_g$ is finite dimensional. It follows that each object of $\fl_0^G\!R$ is componentwise finite dimensional so that the graded dual $D$ is a duality.
\begin{Cor}\label{kdual}
	Let $R$ be a $G$-graded ring with $\dim^G\!R=d$, and suppose that $R_0$ is a finite dimensional $k$-algebra. Then there is a natural isomorphism on $\fl_0^G\!R$:
	\[ \Hom_R(-,\textoplus_{\height^G\!\m=d}E_R^G(R/\m))\simeq D. \]
\end{Cor}
\begin{proof}
	We know that $D=\Hom_R(-,DR)=\Hom_R(-,\bigoplus_{\m\in\Max^G\!R}E_R^G(R/\m))$ by adjunction and \ref{DR}. Now, since every object in $\fl_0^G\!R$ is supported at $G$-maximal ideals of $G$-height $d$, we have $\Hom_R(X,E_R^G(R/\m))=0$ for all $X\in\fl^G_0\!R$ and $\m\in\Max^G\!R$ with $\height^G\!\m<d$. It follows that we have $\Hom_R(-,\bigoplus_{\m\in\Max^G\!R}E_R^G(R/\m))=\Hom_R(-,\bigoplus_{\height^G\!\m=d}E_R^G(R/\m))$ on $\fl^G_0\!R$.
\end{proof}

\section{Group graded algebras}
We next discuss algebras graded by an arbitrary abelian group. In particular, we discuss gradablity of simples over finite dimensional algebras, and the graded singularity categories of module-finite algebras. 
\subsection{Gradability of simples}
Let $G$ be a finitely generated abelian group, and $\L$ a $G$-graded finite dimensional $k$-algebra. For a torsion-free subgroup $H$ of $G$, consider the forgetful functor
\[F:\mod^G\!\L\to\mod^{G/H}\!\L.\]
We call $X\in\mod^{G/H}\!\La$ \emph{$G$-gradable} if there exists $Y\in\mod^G\!\La$ such that $X\simeq FY$ in $\mod^{G/H}\!\La$.

\begin{Lem}\label{forget 1}
	Under the setting above, the following assertions hold.
	\begin{enumerate}
		\item $\La_0$ is local if and only if $\La^{(H)}:=\bigoplus_{h\in H}\Lambda_h$ is local.
		\item $X\in\mod^G\!\La$ is indecomposable if and only if $FX\in\mod^{G/H}\!\La$ is indecomposable.
		\item Each direct summand of $G$-gradable modules in $\mod^{G/H}\!\La$ is also $G$-gradable.
		\item Let $X,Y\in\mod^G\!\La$ be indecomposable. Then $FX\simeq FY$ in $\mod^{G/H}\!\La$ if and only if there exists $h\in H$ such that $X\simeq Y(h)$ in $\mod^G\!\La$.
	\end{enumerate}
\end{Lem}

\begin{proof}
	Since $H\simeq\Z^n$ for some $n\ge0$, the assertion (1) is an iterated application of \cite[3.1]{GG}.
	Other assertions are immediate, see \cite[3.2, 3.3, 4.1]{GG}.
\end{proof}

\begin{Prop}\label{forget 2}
	Under the setting above, the following assertions hold.
	\begin{enumerate}
		\item For each $S\in\Sim^G\La$, we have $FS\in\Sim^{G/H}\La$.
		\item $\rad^G\!\La=\rad^{G/H}\!\La$.
		\item $F$ induces a bijection
		\[ (\Sim^G\!\La)/H\xsimeq\Sim^{G/H}\!\La. \]
	\end{enumerate}
\end{Prop}

\begin{proof}
	By induction, it suffices to consider the case $H=(h)$ for a torsion-free element.
	
	(1) We show that, 
	We have to show that each $G/H$-homogeneous element $x\neq0$ generates $S$. We may assume $x\in\bigoplus_{i\in\Z}S_{ih}$, and moreover $x=\sum_{i\ge0}x_i$ with $x_i\in S_{ih}$ such that $x_0\neq0$. Since $S\in\Sim^G\La$, we have $x_0\La=S$.
	
	It suffices to prove $x\La\supset S_{g+H}$ for each $g\in G$.  Since $\dim_kS$ is finite, there exists $n\in\Z$ such that $S_{g+\Z_{>n}h}=0$.
	Now fix $m\le n$ and assume $x\La\supset S_{g+\Z_{>m}h}$. Since $x_0\La=S$, we have
	\[S_{g+mh}=x_0\La_{g+mh}=(x-x_{\geq1})\La_{g+mh}\subset x\La-S_{g+\Z_{>m}h}\subset x\La.\]
	Thus $x\La\supset S_{g+\Z_{\ge m}h}$. Inductively $x\La\supset S_{g+H}$ holds, as desired.
	
	(2) Recall that, for $*=G$ and $G/H$, $\rad^*\La$ is the intersection of the annihilators of elements in $\Sim^*\!\La$. Thus (i) implies $\rad^G\!\La\supset\rad^{G/H}\!\La$.
	
	To prove the reverse inclusion, it suffices to show that, for each $S\in\Sim^{G/H}\!\La$, $S(\rad^G\!\La)=0$.
	Otherwise, since $S(\rad^G\!\La)\in\mod^{G/H}\!\La$, we have $S(\rad^G\!\La)=S$.
	This is a contradiction since $\La\in\mod^G\!\La$ has a finite filtration by elements in $\Sim^G\!\La$ and hence $\rad^G\!\La$ is a nilpotent ideal.
	
	(3) The map is well-defined by (1), surjective by (2) and \ref{forget 1}(3), and injective by \ref{forget 1}(4).
\end{proof}

\subsection{Graded singular loci and graded singularity categories}
Let $G$ be an abelian group, $R$ a commutative Noetherian $G$-graded ring, and $\Lambda$ a $G$-graded module-finite $R$-algebra such that the structure morphism $R\to\Lambda$ preserves the $G$-grading.

We first discuss singular loci of module-finite algebras.
For a $G$-graded ring $A$ and $X\in\Mod^G\!A$ we denote by $\pd^G_A\!X$ the projective dimension of $X$ in $\Mod^G\!A$. Recall that the {\it singular locus} of a $G$-graded module-finite $R$-algebra $\La$ is
\[ \Sing^G_R\!\La=\{ \p\in\Spec^G\!R \mid \sg^G\!\La_{\p,G}\neq0\}, \]
and $\sg^G\!\La=0$ if and only if every object in $\mod^G\!\La$ has finite projective dimension.
\begin{Lem}\label{sing}
Let $H\subset G$ be a torsion-free subgroup. For $\p\in\Spec^G\!R$ consider the localization $\La_{\p.G}\to\La_{\p,G/H}$.
\begin{enumerate}
\item For $X\in\mod^G\!\La$, we have $\pd^G_{\La_{\p,G}}\!X_{\p,G}=\pd^{G/H}_{\La_{\p,G/H}}\!X_{\p,G/H}$.
\item We have $\Sing_R^G\!\La\subset\Sing^{G/H}_R\!\La\cap\Spec^{G}\!R$.
\item Suppose furthermore that $G=H=\Z$ and $\La=R$. Then $\Sing^\Z\!R=\Sing R\cap\Spec^\Z\!R$.
\end{enumerate}
\end{Lem}
We do not know if in (2) the equality holds in general.
\begin{proof}
	(1)  For $X\in\Mod^G\!\La$ we know that $\pd^G_\La\!X=\sup\{i\geq0\mid E^i\neq0\}$ for $E^i=\Ext^i_\La(X,\Om^iX)$. Then for $X\in\mod^G\!\La$ we have $\pd^G_{\La_{\p,G}}\!X_{\p,G}=\sup\{i\geq0\mid (E^i)_{\p,G}\neq0\}$. By \ref{Spec G G/H}(2) we see that $(E^i)_{\p,G}\neq0$ if and only if its $G/H$-homogeneous localization $(E^i)_{\p,G/H}=\Ext^i_{\La_{\p,G/H}}(X_{\p,G/H},\Om^iX_{\p,G/H})$ is non-zero, which yields the desired equality.
	
	(2)  Let $\p\in\Sing^G_R\!\La$ so that $\mod^G\!\La_{\p,G}$ has infinite global dimension. Then it contains a module $X$ of infinite projective dimension. By (1), its localization $X_{\p,G/H}$ has also infinite projective dimension in $\mod^{G/H}\!\La_{\p,G/H}$, thus $\p\in\Sing^{G/H}_R\!\La$.
	
	(3)  We have the inclusion $\subset$ by (2), so we prove the converse. Let $\p\in\Spec^\Z\!R$, and we have to show that if $\p\not\in\Sing^\Z\!R$ then $\p\not\in\Sing R$. By assumption we have that $\mod^\Z\!R_{\p,\Z}$ has finite global dimension, thus applying \ref{reg} to the $\Z$-graded ring $R_{\p,\Z}$ yields that this is regular (as an ungraded ring). It follows that $R_\p$, being a localization of $R_{\p,\Z}$, is also regular, proving $\p\not\in\Sing R$. 
\end{proof}

We state the relationship for graded and ungraded regularity in the commutative $\Z$-graded case.
\begin{Lem}\label{reg}
	Let $R$ be a commutative Noetherian $\Z$-graded ring. Then every object of $\mod^\Z\!R$ has finite projective dimension if and only if $R$ is regular (as an ungraded ring).
\end{Lem}
\begin{proof}
	The `if' part is clear, so we prove the `only if' part. By \cite[2.2.24(a)]{BH} we know that $\Sing R$ is defined by a homogeneous ideal, so it suffices to show that the localization $R_\m$ is regular for every graded maximal ideal $\m$. By assumption $R/\m\in\mod^\Z\!R$ has finite projective dimension, thus so does $R_\m/\m R_\m\in\mod R_\m$ by localizing, hence $R_\m$ is regular.
\end{proof}

Let us turn to $G$-graded singularity categories of module-finite algebras.
For a subset $\Phi$ of $\Spec^G\!R$, let
\[\mod^G_\Phi\!\La:=\{X\in\mod^G\!\La\mid \Supp^G_RX\subset\Phi\}\ \mbox{ and }\ \sg^G_\Phi\!\La:=\thick(\mod^G_\Phi\!\La)\subset\sg^G\!\La.\]
Clearly $\Supp^G_R\End_{\sg\La}(X)\subset\Supp^G_RX$ holds. On the other hand, the following observation \cite{Sch,Tak} shows that, in the singularity category, $X$ is generated by $\mod_\Phi^G\!\La$ for $\Phi:=\Supp^G_R\End_{\sg\La}(X)$.

\begin{Prop}\label{B16}
Let $\La$ be a $G$-graded module-finite $R$-algebra.
\begin{enumerate}
\item For each $X\in\mod^G\!\La$, we have
\[ X\in\sg_\Phi^G\!\La\ \mbox{ for }\ \Phi:=\Supp^G_R\End_{\sg\La}(X). \]
\item For each subset $\Phi$ of $\Spec^G\!R$, we have
\[\sg_\Phi^G\!\La=\{X\in\sg^G\!\La\mid\Supp_R^G\End_{\sg\La}(X)\subset\Phi\}.\]
\item For $\Phi:=\Sing^G_R\La$, we have
\[\sg_\Phi^G\!\La=\sg^G\!\La.\]
\end{enumerate}
\end{Prop}

To prove this, we need the following preparation.

\begin{Lem}\label{B16 2}
Under the setting of \ref{B16}, for each $X\in\mod^G\!\La$, we have
\begin{align*}
\Supp_R^G\underline{\End}_\La(X)&=\{\p\in\Spec^G\!R\mid X_{[\p]}\notin\proj\La_{[\p]}\}\subset\Supp_R^GX,\\
\Supp_R^G\End_{\sg\La}(X)&=\Supp_R^G\underline{\End}_\La(\Om^nX)\ \mbox{ for }\ n\gg0,\\
\Supp_R^G\End_{\sg\La}(X)&\subset\Sing^G_R\!\La.
\end{align*}
\end{Lem}

\begin{proof}
Since $\underline{\End}_\La(X)_{[\p]}=\underline{\End}_{\La_{[\p]}}(X_{[\p]})$, the first equality follows.
We prove the second one. Recall \cite{KV} that the morphism space in the singularity category $\sg^G\!\La$ is given as the stabilization of the stable category $\smod^G\!\La$, in particular we have the following for each $X,Y\in\mod^G\!\La$:
	\begin{equation}\label{Hom_sg colimit} \Hom_{\sg\La}(X,Y)=\colim\left( \xymatrix{ \sHom_\La(X,Y)\ar[r]^-{\Om}&\sHom_\La(\Om X,\Om Y)\ar[r]^-{\Om}&\cdots} \right).
	\end{equation}
Thus \[\Supp_R^G\End_{\sg\La}(X)\subset\Supp_R^G\sEnd_\La(\Om^nX)\ \mbox{ for }\ n\gg0.\]
Since $\Omega_{\La_{[\p]}}$ sends $\proj\La_{[\p]}$ to itself, the first equality implies $\Supp^G_R\sEnd_\La(X)\supset\Supp^G_R\sEnd_\La(\Om X)\supset\cdots$. 
The sequence stabilizes since $\Spec^G\!R$ is a Noetherian space. Thus we have 
\[\Supp_R^G\End_{\sg\La}(X)\subset\bigcap_{n\ge0}\Supp_R^G\sEnd_\La(\Om^nX)=\Supp_R^G\sEnd_\La(\Om^nX)\ \mbox{ for }\ n\gg0.\]
To show the converse, fix $\p\in\Spec^G\!R$ which does note belong to the left-hand-side. Then $\End_{\sg\La}(X)_{[\p]}=0$ holds, and hence $X_{[\p]}\in\mod\La_{[\p]}$ has finite projective dimension, say $n$. Then $\sEnd_{\La}(\Om^nX)_{[\p]}=0$ for $n\gg0$. Thus $\p$ does not belong to the right-hand-side.

We prove the third inclusion.
For each $\p\in\Spec^G\!R$, by \eqref{Hom_sg colimit}, we have $\End_{\sg\La}(X)_{[\p]}=\End_{\sg\La_{[\p]}}(X_{[\p]})=\bigoplus_{g\in G}\Hom_{\sg^G\!\La_{[\p]}}(X_{[\p]},X_{[\p]}(g))$. If $\p\notin\Sing^G_R\!\La$, then the right-hand-side vanish, and the assertion holds.
\end{proof}

Now we prove \ref{B16}.

\begin{proof}[Proof of \ref{B16}]
(1) Let $I$ be the $G$-graded defining ideal of $\Phi$ and take its homogeneous generators $r_1,\ldots,r_m$. Let $K_i:=[\cdots0\to R\xrightarrow{r_i}R\to 0\cdots]$ and $X_i:=X\lotimes_RK_1\lotimes_R\cdots\lotimes_RK_i$ the Koszul complex. Then $X_0=X$, and
\[X_m\in\thick(\mod_\Phi^G\La)\subset\D^b(\mod^G\!\La)\]
since each cohomology of $X_m$ is annihilated by $I$ and hence belongs to $\mod_\Phi^G\La$.
For each $1\le i\le m$, there is a triangle in $\D^b(\mod^G\!\La)$ and also in $\sg^G\!\La$:
\[X_{i-1}\xrightarrow{r_i}X_{i-1}\to X_i\to X_{i-1}[1].\]
Inductively, one can show that in $\sg^G\!\La$ the triangle above splits, $X_i\simeq X_{i-1}\oplus X_{i-1}[1]$ holds, and the $R$-module $\End_{\sg^G\!\La}(X_i)$ is annihilated by $I$.
Consequently, $X$ is a direct summand of $X_m$ in $\sg^G\!\La$, and the assertion follows.
	
(2) The inclusion $\subset$ follows from $\Supp_R^G\End_{\sg\La}(X)\subset\Supp^G_RX$ in \ref{B16 2}, and $\supset$ from (1).

(3) For each $X\in\sg^G\!\La$, we have $\Supp^G_R\!\End_{\sg\La}(X)\subset\Sing^G_R\!\La$ by \ref{B16 2}. Thus $X\in\sg_\Phi^G\!\La$ holds by (2), and the assertion follows.
\end{proof}

\begin{Lem}\label{B16 3}
Let $\La$ be a $G$-graded module-finite $R$-algebra. Then
\begin{align*}
\Sing^G_R\!\La&=\bigcup_{X\in\sg^G\!\La}\Supp_R^G\End_{\sg\La}(X).
\end{align*}
\end{Lem}

\begin{proof}
Let $\p\in\Spec^G\!R$. Since the localization functor $\sg^G\!\La\to\sg^G\!\La_{\p,G}$ is dense, and $\End_{\sg\La_{\p,G}}(X_{\p,G})=\End_{\sg\La}(X)_{\p,G}$, it follows that $\p\in\Sing_R^G\!\La$ if and only if there exists $X\in\sg^G\!\La$ such that $\End_{\sg\La_{\p,G}}(X_{\p,G})\neq0$ if and only if there exists $X\in\sg^G\!\La$ such that $\p\in\Supp_R^G\!\End_{\sg\La}(X)$. Thus the assertion holds.
\end{proof}

One can relate the graded and ungraded singular loci of $G$-graded module-finite algebras.
\begin{Lem}\label{tsujitsuma2}
Let $\m$ be a $G$-maximal ideal which is maximal as an ungraded ideal, and let $\La$ be a $G$-graded module-finite $R$-algebra. 
If $\Sing_R\!\La\subset\{\m\}$, then $\Sing_R^G\!\La\subset\{\m\}$.
\end{Lem}
\begin{proof}
	We have $\bigcup_{X\in\sg\La}\Supp_R\!\End_{\sg\La}(X)\stackrel{\ref{B16 3}}{=}\Sing_R\!\La\subset\{\m\}$. 
	For each $X\in\sg\La$, applying \ref{tsujitsuma} to $\End_{\sg\La}(X)$, we have $\Supp_R^G\!\End_{\sg\La}(X)\subset\{\m\}$. Thus $\bigcup_{X\in\sg^G\!\La}\Supp_R^G\!\End_{\sg\La}(X)\stackrel{\ref{B16 3}}{=}\Sing_R^G\!\La\subset\{\m\}$.
\end{proof}


\section{Covering functors and tilting subcategories}

A functor $\pi\colon\tC\to\C$ between categories is called a {\it covering functor} if the canonical maps
\[ \Coprod_{\pi Y^\prime=\pi Y}\tC(X,Y^\prime)\to\C(\pi X,\pi Y),\qquad \Coprod_{\pi X^\prime=\pi X}\tC(X^\prime,Y)\to\C(\pi X,\pi Y) \]
are isomorphisms for all $X,Y\in\tC$. When the categories are additive (resp. triangulated) we require the functor to be additive (resp. triangulated).
Let us give a useful result which allows us to lift tilting subcategories to its covering.
A special case can be found in existing literatures, e.g. \cite[Theorem 3.5]{As}.

\begin{Thm}\label{covering}
Let $\pi\colon\T\to\U$ be a covering functor between idempotent-complete triangulated categories such that the image of $\pi$ generates $\U$ as a thick subcategory, and $\B\subset\U$ a subcategory contained in the image of $\pi$.
\begin{enumerate}
\item $\B\subset\U$ is silting if and only if $\pi^{-1}\B\subset\T$ is silting.
\item $\B\subset\U$ is tilting if and only if $\pi^{-1}\B\subset\T$ is tilting.
\end{enumerate}
\end{Thm}
\begin{proof}
	Since $\pi$ is a triangulated covering functor and $\pi\pi^{-1}\B=\B$, it is clear that $\T(A,A^\prime[i])=0$ for all $A,A^\prime\in\pi^{-1}\B$ if and only if $\U(B,B^\prime[i])=0$ for all $B,B^\prime\in\B$. It follows that $\B\subset\U$ is presilting (resp. pretilting) if and only if $\pi^{-1}\B\subset\T$ is presilting (resp. pretilting). It remains to show that a presilting subcategory $\B$ generates $\U$ if and only if $\pi^{-1}\B$ generates $\T$. The assertion follows from \ref{pi^-1} below.	
\end{proof}

For subcategories $\V$ and $\W$ of a triangulated category $\C$ we denote the category of extensions $\V\ast\W=\{ C\in\C\mid \text{there is a triangle } V\to C\to W\to V[1] \text{ in } \C \}$. With this notation, for a presilting subcategory $\V\subset\C$ we have $\thick\V=\bigcup_{l\geq0}\add(\V[-l]\ast\cdots\ast\V[l])$ \cite[2.15]{AI}.
\begin{Lem}\label{pi^-1}
Let $\B\subset\U$ be a idempotent-complete presilting subcategory. Then $\B$ generates $\U$ if and only if $\pi^{-1}\B$ generates $\T$.
\end{Lem}
\begin{proof}	
It suffices to prove ``only if '' part.
	Pick an object $A\in\T$. Up to adding a direct summand and suitably shifting, we may assume $\pi A\in\B\ast\cdots\ast\B[\ell]\ast\B[\ell+1]$ for some $\ell\geq0$. We claim by induction that one can take the triangles $A_{i+1}\to \widetilde{B_i}\xrightarrow{a_i} A_i\to A_{i+1}[1]$ in $\T$ for $0\leq i\leq \ell$ with $A_0=A$, satisfying the following:
	\begin{itemize}
		\item $\pi a_i$ is a right $\B$-approximation in $\U$,
		\item $\pi A_{i+1}\in \B\ast\cdots\ast\B[\ell-i]$. 
	\end{itemize}
	Suppose that we have constructed such triangles 
	up to $i-1$ ($i\geq0$). Since $\pi$ is a covering functor there is a morphism $a_i\colon \widetilde{B_i}\to A_i$ in $\T$ whose image under $\pi$ is a right $\B$-approximation. Indeed, for a right $\B$-approximation $b_i^0\colon B_i^0\to\pi A_i$, there exist finitely many $\widetilde{B_i}^1,\ldots,\widetilde{B_i}^n\in\pi^{-1}B_i^0$ and morphisms $a_i^j\colon\widetilde{B_i}^j\to A_i$ such that the induced map $B_i^0\xrightarrow{\diag}\bigoplus_jB_i^0\xrightarrow{\pi a_i^j} \pi A_i$ coincides with $b_i^0$. Then the second map is a right $\B$-approximation, thus we can take $\widetilde{B_i}:=\bigoplus_j\widetilde{B_i}^j\to A_i$ as ${a_i}$. Now we extend this map to a triangle $A_{i+1}\to \widetilde{B_i}\xrightarrow{a_i} A_i\to A_{i+1}[1]$ in $\T$. By induction hypothesis we know that $\pi A_{i}\in\B\ast\cdots\ast\B[\ell-i+1]$, and since $\pi a_i$ is a right $\B$-approximation, we deduce $\pi A_{i+1}\in\B\ast\cdots\ast\B[\ell-i]$, as desired.
	
	Now in the last triangle (for $i=\ell$) we have $\pi A_{\ell+1}\in\B$, thus $A_{\ell+1}\in\pi^{-1}\B$. It follows from the triangles we constructed that $A$ lies in the thick subcategory generated by $\pi^{-1}\B$.
\end{proof}

Even if $\B\subset\U$ is not presilting, we can argue as follows for the ``algebraic'' case. For an additive category $\A$ with arbitrary (set-indexed) coproducts, we denote by $\A^c$ the full subcategory of compact objects.
\begin{Rem}
Suppose that there exist a triangle functor $\widehat{\pi}\colon\widehat{\T}\to\widehat{\U}$ (which is not necessarily a covering functor) 
between compactly generated triangulated categories with $\widehat{\T}^c=\T$ and $\widehat{\U}^c=\U$, which restricts to $\pi\colon\T\to\U$. Then a subcategory $\B\subset\U$ contained in the image of $\pi$ generates $\U$ if and only if $\pi^{-1}\B$ generates $\T$.
\end{Rem}
\begin{proof}
	It is clear that $\B=\pi\pi^{-1}\B$ generates $\U$ if $\pi^{-1}\B$ generates $\T$.
	We prove the converse. By our ``algebraic'' assumption, we only have to show $\T(A,X)=0$ for all $A\in\pi^{-1}\B$ implies $X=0$. In this case we have $\U(\pi A,\pi X)=\bigoplus_{A^\prime\in\pi^{-1}\pi A}\T(A^\prime,X)=0$ for all $A\in\pi^{-1}\B$, hence $\pi X=0$ since $\B$ generates $\U$. It follows that $\T(X,X)\subset\U(\pi X,\pi X)=0$, thus $X=0$.
\end{proof}
\end{appendix}
	
\end{document}

%% file: definitions2.tex
\DeclareMathOperator{\image}{Im}
\renewcommand{\Im}{\image}
\DeclareMathOperator{\Ker}{Ker}
\DeclareMathOperator{\Coker}{Coker}

\DeclareMathOperator*{\colim}{colim}

\DeclareMathOperator{\Hom}{Hom}
\DeclareMathOperator{\sHom}{\underline{\Hom}}

\DeclareMathOperator{\RHom}{\mathrm{R}\!\Hom}
\DeclareMathOperator{\cHom}{\mathscr{H}\hspace{-2pt}{\it om}}
\DeclareMathOperator{\End}{End}
\DeclareMathOperator{\sEnd}{\underline{\End}}
\DeclareMathOperator{\REnd}{\mathrm{R}\!\End}
\DeclareMathOperator{\cEnd}{\mathscr{E}\hspace{-2pt}{\it nd}}
\def\lotimes{\otimes^\mathrm{L}}

\DeclareMathOperator{\Ext}{Ext}

\DeclareMathOperator{\pd}{proj.dim}
\DeclareMathOperator{\id}{inj.dim}

\DeclareMathOperator{\gd}{gl.dim}

\DeclareMathOperator{\md}{mod}
\renewcommand{\mod}{\md}
\DeclareMathOperator{\Mod}{Mod}
\DeclareMathOperator{\smod}{\underline{\mod}}

\DeclareMathOperator{\proj}{proj}

\DeclareMathOperator{\fl}{fl}

\DeclareMathOperator{\CM}{CM}
\DeclareMathOperator{\sCM}{\underline{\CM}}

\DeclareMathOperator{\Sim}{\mathrm{sim}}
\DeclareMathOperator{\sg}{sg}

\DeclareMathOperator{\thick}{thick}
\DeclareMathOperator{\per}{per}

\DeclareMathOperator{\add}{add}

\DeclareMathOperator{\rad}{rad}

\DeclareMathOperator{\ass}{Ass}
\DeclareMathOperator{\Ass}{Ass}
\DeclareMathOperator{\Min}{Min}
\DeclareMathOperator{\Ann}{Ann}

\DeclareMathOperator{\Spec}{Spec}
\DeclareMathOperator{\Max}{Max}

\DeclareMathOperator{\Sing}{Sing}
\DeclareMathOperator{\Supp}{Supp}

\DeclareMathOperator{\height}{ht}
\DeclareMathOperator{\depth}{depth}

\DeclareMathOperator{\principaldivisor}{div}
\renewcommand{\div}{\principaldivisor}

\DeclareMathOperator{\SL}{SL}
\DeclareMathOperator{\GL}{GL}
\DeclareMathOperator{\diag}{diag}

\def\dgcat{\mathrm{dgcat}}

\def\Hmo{\mathrm{Hmo}}

\newcommand\recollement[3]{\xymatrix{{#1}\ar[r]&{#2}\ar[r]\ar@/_7pt/[l]\ar@/^7pt/[l]&{#3}\ar@/_7pt/[l]\ar@/^7pt/[l] }}
\def\A{\mathscr{A}}
\def\B{\mathscr{B}}
\def\C{\mathscr{C}}
\def\D{\mathscr{D}}

\def\K{\mathscr{K}}

\def\N{\mathscr{N}}

\def\P{\mathscr{P}}

\def\T{\mathscr{T}}
\def\U{\mathscr{U}}
\def\V{\mathscr{V}}
\def\W{\mathscr{W}}

\def\tC{\widetilde{\C}}

\def\G{\Gamma}
\def\Ga{\Gamma}

\def\bG{\mathbf{\G}}
\def\bGa{\mathbf{\Ga}}

\def\L{\Lambda}
\def\La{\Lambda}
\def\bPi{\mathbf{\Pi}}

\def\Om{\Omega}
\def\a{\alpha}
\def\al{\alpha}
\def\b{\beta}

\def\om{\omega}

\def\NN{\mathbb{N}}
\def\Z{\mathbb{Z}}

\def\AA{\mathbb{A}}

\def\LL{\mathbb{L}}

\def\p{\mathfrak{p}}
\def\q{\mathfrak{q}}
\def\r{\mathfrak{r}}
\def\m{\mathfrak{m}}
\def\n{\mathfrak{n}}

\def\op{\mathrm{op}}
\def\dg{\mathrm{dg}}
\def\ac{\mathrm{ac}}

\def\lsimeq{\rotatebox{90}{$\simeq$}}
\def\rsimeq{\rotatebox{-90}{$\simeq$}}
\def\xsimeq{\xrightarrow{\simeq}}

\def\x{\vec{x}}
\def\y{\vec{y}}

\newtheorem{Thm}{Theorem}[section]
\newtheorem{Lem}[Thm]{Lemma}
\newtheorem{Prop}[Thm]{Proposition}
\newtheorem{Cor}[Thm]{Corollary}

\newtheorem{Prop-Def}[Thm]{Proposition-Definition}
\newtheorem{Thm-Def}[Thm]{Theorem-Definition}

\theoremstyle{definition}
\newtheorem{Def}[Thm]{Definition}
\newtheorem{Ex}[Thm]{Example}
\newtheorem{Con}[Thm]{Construction}

\newtheorem{Setup}[Thm]{Setting}

\newtheorem{Pb}[Thm]{Problem}

\theoremstyle{remark}
\newtheorem{Rem}[Thm]{Remark}

\newcounter{step}
\newcommand{\Step}[1]{\refstepcounter{step}{{\it Step \thestep: {#1}}}}

\def\disoplus{\displaystyle\bigoplus}
\def\textoplus{\textstyle\bigoplus}

\def\Coprod{\displaystyle\coprod}

\newcommand\Frac[2]{\displaystyle\frac{#1}{#2}}

\makeatletter
\renewcommand{\theequation}{\arabic{section}.\arabic{equation}}
\@addtoreset{equation}{section}
\makeatother

